\documentclass{amsart}

\usepackage{fullpage}
\usepackage{mathtools}
\usepackage{amssymb,amsmath,amsthm,amscd,mathrsfs,graphicx}
\usepackage[dvipsnames]{xcolor}
\usepackage{bm}
\usepackage{hyperref}
\usepackage{dsfont}
\usepackage{enumerate}
\usepackage{epsfig}
\usepackage{float, graphicx}
\usepackage{latexsym, amsxtra}
\usepackage{pdfpages}
\usepackage{multicol}
\usepackage[normalem]{ulem}
\usepackage{psfrag}
\usepackage{stmaryrd}
\usepackage{tikz}
\usepackage[T1]{fontenc}
\usepackage{url}
\usepackage{verbatim}
\usepackage{xy}
\usepackage{indentfirst}
\usepackage{tikz-cd}
\usepackage{url}

\makeatletter
\def\thm@space@setup{%
  \thm@preskip=2ex \thm@postskip=2ex
}
\makeatother

\hypersetup{hidelinks}

\numberwithin{equation}{section}
\theoremstyle{plain}

\newtheorem{thm}{Theorem~}[section] 
\newtheorem{lem}[thm]{Lemma~}

\newtheorem{prop}[thm]{Proposition~}

\newtheorem{cor}[thm]{Corrolary~}

\theoremstyle{remark}
\newtheorem{rmk}[thm]{Remark~}

\theoremstyle{definition}
\newtheorem{defn}[thm]{Definition~}
\newtheorem{notation}[thm]{Notation~}

\newtheorem{cond}[thm]{Condition~}

\newcommand{\calM}{\mathcal{M}}
\newcommand{\calD}{\mathcal{D}}
\newcommand{\calH}{\mathcal{H}}
\newcommand{\calP}{\mathcal{P}}
\newcommand{\calC}{\mathcal{C}}
\newcommand{\calF}{\mathcal{F}}

\newcommand{\CC}{\mathbb{C}}
\newcommand{\ZZ}{\mathbb{Z}}
\newcommand{\RR}{\mathbb{R}}
\newcommand{\LL}{\mathbb{L}}
\newcommand{\PP}{\mathbb{P}}
\newcommand{\FF}{\mathbb{F}}
\newcommand{\QQ}{\mathbb{Q}}
\newcommand{\BB}{\mathbb{B}}
\newcommand{\DD}{\mathbb{D}}

\newcommand{\Prd}{\mathscr{P}}

\newcommand{\sO}{\mathscr{O}}

\newcommand{\sA}{\mathscr{A}}

\newcommand{\C}{\mathcal{C}}
\newcommand{\D}{\mathcal{D}}

\newcommand\PGL{\mathrm{PGL}}
\newcommand\PSL{\mathrm{PSL}}

\newcommand\Hom{\mathrm{Hom}}
\newcommand\id{\mathrm{id}}

\newcommand\Ker{\mathrm{Ker}}

\newcommand\Co{\mathrm{Co}}
\newcommand\Suz{\mathrm{Suz}}
\newcommand\mo{\mathrm{mod}}
\newcommand\sig{\mathrm{sig}}

\newcommand\rank{\mathrm{rank}}
\newcommand\I{\mathrm{I}}
\newcommand\II{\mathrm{II}}
\newcommand\SL{\mathrm{SL}}
\newcommand\IV{\mathrm{IV}}

\newcommand\diag{\mathrm{diag}}
\newcommand\GL{\mathrm{GL}}
\newcommand\Sym{\mathrm{Sym}}
\newcommand\Hol{\mathrm{Hol}}
\newcommand\Gr{\mathrm{Gr}}
\newcommand\Span{\mathrm{Span}}
\newcommand\AGL{\mathrm{AGL}}
\newcommand\ASL{\mathrm{ASL}}
\newcommand\PSU{\mathrm{PSU}}
\newcommand\PU{\mathrm{PU}}
\newcommand\ord{\mathrm{ord}}
\newcommand\QD{\mathrm{QD}}

\DeclareMathOperator{\NS}{NS}
\DeclareMathOperator{\Aut}{Aut}

\title{Automorphisms and Periods of Cubic Fourfolds}

 \author[R. Laza]{Radu Laza}
\address{Stony Brook University,  Stony Brook, NY 11794, USA}
\email{radu.laza@stonybrook.edu}
 \author[Z. Zheng]{Zhiwei Zheng}
\address{Tsinghua University, Beijing, China}
\email{zheng-zw14@mails.tsinghua.edu.cn}
\date{}
\begin{document}
\bibliographystyle{amsalpha}

\begin{abstract} 
We classify the symplectic automorphism groups for cubic fourfolds. The main inputs are the global Torelli theorem for cubic fourfolds and the  classification of the fixed-point sublattices of the Leech lattice. Among the highlights of our results, we note that there are $34$ possible groups of symplectic automorphisms, with $6$ maximal cases. The six maximal cases correspond to $8$ non-isomorphic, and isolated in moduli, cubic fourfolds; six of them previously identified by other authors. Finally, the Fermat cubic fourfold has the largest possible order ($174,960$) for the automorphism group (non-necessarily symplectic) among all smooth cubic fourfolds. 
\end{abstract}

\maketitle

\section{Introduction}
Cubic fourfolds are some of the most intensely studied objects in algebraic geometry in connection to rationality questions and to constructing compact hyper-K\"ahler manifolds. What sets the cubic fourfolds apart is that they are Fano fourfolds whose middle cohomology is of level $2$ with $h^{3,1}=1$ (i.e.,  up to a Tate twist looks like the cohomology of a $K3$ surface). Consequently, the moduli space of cubic fourfolds behaves very similarly to the moduli space of polarized $K3$ surfaces. Specifically, Voisin \cite{voisin} proved a global Torelli theorem for cubic fourfolds. Later, Hassett \cite{hassett} identified some natural  Noether--Lefschetz divisors $\calC_d$ (for $d\in \ZZ_+$ with $d\equiv 0,2 \pmod 6$) in the moduli space of cubic fourfolds, and conjectured that the image of the period map is the complement of $\calC_2$ and $\calC_6$. This was subsequently verified by Laza \cite{gitcubic,laza2010moduli} and Looijenga \cite{lcubic}. More recently, the second author \cite{zheng2017} proved a stronger version of the Torelli theorem: the automorphisms of cubic fourfolds are detected by (polarized) Hodge isometries.

The purpose of this paper is to use the period map to study and classify the possible {\it symplectic} automorphism groups (Definition ~\ref{definition: symplectic}) for cubic fourfolds. The model for our study is the well-known case of $K3$ surfaces. Namely, a consequence of the Torelli theorem for $K3$ surfaces is that there is a close connection between the automorphism group $\Aut(Y)$ of a $K3$ surface $Y$ and the Hodge isometries on $H^2(Y,\ZZ)$. Nikulin \cite{nikulin} started a systematic investigation of the possible finite automorphism groups for $K3$ surfaces by means of lattice theory (\cite{nikulin1980integral}). This study culminated with the celebrated result of Mukai \cite{mukaiaut} relating the classification of the finite groups of symplectic automorphisms acting on $K3$ surfaces with certain subgroups of the Mathieu group $M_{23}$. Kond\=o \cite{kondo} simplified Mukai's proof by relating this classification problem to the isometries of the Niemeier lattices. Kond\=o's approach avoids the Leech lattice (the unique Niemeier lattice containing no roots), but it turns out that a related construction that involves only the Leech lattice $\LL$ behaves more uniformly and adapts to higher dimensions (\cite{gaberdiel2012symmetries}, \cite{huybrechtsaut}). In particular, one sees that all the symplectic automorphism groups $G$ occurring are subgroups of the Conway group $\Co_0(=O(\LL))$ satisfying a certain rank condition on the fixed-point sublattice $\LL^G$.

The higher dimensional analogue of the $K3$ surfaces are the hyper-K\"ahler manifolds (simply connected, compact K\"ahler manifold, carrying a unique holomorphic symplectic $2$-form). Due to Verbitsky's Torelli Theorem and recent results on Mori cones of hyper-K\"ahler manifolds (e.g., \cite{bayer2014mmp}, \cite{bayer2015mori}, \cite{HT_Aut}), the approach to automorphisms via lattices that works for $K3$ surfaces can be extended to the case of hyper-K\"ahler manifolds of $K3^{[n]}$ type, leading to a flurry of activity on the subject. In particular, we note the work of Mongardi \cite{mongardi2013thesis,mongardiaut}  who started a systematic study of the symplectic automorphisms of hyper-K\"ahler manifolds of $K3^{[n]}$ type. Around the same time, H\"ohn and Mason \cite{HM1} have completed the classification of the fixed-point sublattices $\LL^G$ of $\LL$ with respect to  subgroups $G$ of $\Co_0$ (the case $G$ is cyclic was previously done by Harada--Lang \cite{harada1990leech}). Using this classification, in subsequent work \cite{HM2}, H\"ohn and Mason have completed Mongardi's analysis for hyper-K\"ahler manifolds of $K3^{[2]}$, obtaining an analogue of Mukai's results in the $4$-dimensional case. There are $15$ maximal groups (\cite[Table 2]{HM2}) that are listed in Table  \ref{table: MSS aut of HK} in our paper.

The cubic fourfolds are intricately related to hyper-K\"ahler fourfolds of $K3^{[2]}$ type. 
Specifically, Beauville--Donagi \cite{beauville1985variety} proved that the Fano variety $F(X)$ of lines on a smooth cubic fourfold is in fact a hyper-K\"ahler fourfold of $K3^{[2]}$ type. An interesting aspect here is that by varying the cubic fourfold, one obtains a locally complete moduli for polarized hyper-K\"ahler manifolds of $K3^{[2]}$ type (i.e.,  $20$ moduli vs. $19$ moduli coming from $K3$ surfaces). In the context of automorphism groups, this leads to the construction of exotic automorphisms for hyper-K\"ahler manifolds of $K3^{[2]}$ type (i.e., not induced from $K3$ surfaces).

Via the Fano variety construction,  the classification of the automorphisms of cubic fourfolds is closely related (but some differences arise due to the polarization) to the classification for hyper-K\"ahler manifolds of $K3^{[2]}$ type, and the above mentioned results. In particular, we note that H\"ohn--Mason  \cite[Table 11]{HM2} have shown that $6$ of the $15$ maximal groups arising in the classification of automorphisms for the $K3^{[2]}$ case are actually realized by some smooth cubic fourfolds. In a different direction, using more geometric arguments, L. Fu \cite{fu2016classification} classified  all possible symplectic automorphism groups of cubic fourfolds which are cyclic  of prime-power order. He also  gave the corresponding normal forms for the associated cubic fourfolds (for some earlier results and other examples see \cite{gonzalez} and \cite{mongardi2013thesis, mongardi2013symplectic}). Building on these  results, we complete (and give a systematic account of) the classification of the possible groups of symplectic automorphisms for cubic fourfolds. Specifically, we classify all possible groups $G=\Aut^s(X)$ of symplectic automorphisms for cubic fourfolds, and for many of them, we give the corresponding normal forms.

\begin{notation}  
We follow the standard notation from group theory for finite groups. For reader's convenience, we recall in Appendix \ref{appendix: finite groups} the relevant notations and definitions. Briefly, we mention that $p^n$ corresponds to $(\ZZ/p\ZZ)^n$,  $L_p(k)$ corresponds to $\PGL(k,\FF_p)$, $\QD_{16}$ (which is denoted $\Gamma_3 a_2$ in \cite{HM2}) is the 
semidihedral group of order 16, $M_9, M_{10}, M_{3,8}$ are the Mathieu groups (see \S\ref{smathieu}), and $A_{m,n}$ is the subgroup of $S_m\times S_n\subset S_{m+n}$ consisting of elements of even signature. We use  $N:Q$ to denote a semidirect product $N\rtimes Q$, and  $N.Q$ for an extension of $Q$ by $N$.
\end{notation}

\begin{thm} 
\label{theorem: main}
Let $X$ be a smooth cubic fourfold with symplectic automorphism group $G=\Aut^s(X)$. Let $S:=S_G(X)$ be the covariant lattice (i.e.,  the orthogonal complement of the invariant sublattice of $H^4(X, \ZZ)$ under the induced action of $G$). Then one of the following situations holds:
\begin{enumerate}

\item[(0)] $\rank(S)=0$, $G=1$.
\item $\rank(S)=8$, $G=2$ and $S=E_8(2)$. For an appropriate choice of coordinates,  $X$ is given by 
\begin{equation*}
X=V\left(F_1(x_1, x_2, x_3, x_4)+x_5^2 L_1(x_1, x_2, x_3, x_4)+x_5 x_6 L_2(x_1, x_2, x_3, x_4)+ x_6^2 L_3(x_1, x_2, x_3, x_4)\right).
\end{equation*}
With respect to these coordinates,  $G$ is generated by $g=\frac{1}{2}(0,0,0,0,1,1)$.

\item $\rank(S)=12$, $G=2^2$ or $G=3$.

\begin{enumerate}[(a)]
\item If $G=2^2$, for an appropriate choice of coordinates, 
\begin{equation*}
X=V\left(F_1(x_1, x_2, x_3)+x_4^2 L_1(x_1, x_2,x_3)+x_5^2 L_2(x_1, x_2, x_3)+ x_6^2 L_3(x_1, x_2, x_3)+x_4 x_5 x_6\right),
\end{equation*}
and 
$G$ is generated by $g_1=\frac{1}{2}(0,0,0,0,1,1)$ and $g_2=\frac{1}{2}(0,0,0,1,1,0)$.
\item If $G=3$, then $X$ is either 
\begin{equation*}
X=V\left(F_1(x_1, x_2, x_3, x_4)+x_5^3+x_6^3+x_5 x_6 L_1(x_1, x_2, x_3, x_4)\right),
\end{equation*}
in which case $G$ is generated by $g=\frac{1}{3}(0,0,0,0,1, 2)$, or 
\begin{equation*}
X=V\left(F_1(x_1, x_2)+F_2(x_3,x_4)+F_3(x_5,x_6)+\Sigma_{i=1,2; j=3,4; k=5,6} (a_{ijk}x_i x_j x_k)\right),
\end{equation*}
with $G$ generated by $g=\frac{1}{3}(0,0,1, 1,2, 2)$.
\end{enumerate}

\item $\rank(S)=14$, $G=4$ or $S_3$.
\begin{enumerate}[(a)]
\item If $G=4$, for an appropriate choice of coordinates, the defining equations of the corresponding cubic fourfolds belong to 
\begin{equation*}
\Span\{x_1 N_1(x_3,x_4), x_2 N_2(x_3, x_4), F_1(x_1, x_2), x_5 x_6 L_1(x_1, x_2), x_5^2 L_2(x_3, x_4), x_6^2 L_3(x_3, x_4)\}.
\end{equation*} 
With respect to these coordinates, $G$ is generated by $g=\frac{1}{4}(0,0,2,2, 1, 3)$.
\item If $G=S_3$, we can choose coordinate $x_1, \cdots, x_6$ of $\CC^6$, such that the action of $S_3$ on $(\CC^6)^\vee$ is by permuting $(x_1, x_2), (x_3, x_4), (x_5, x_6)$ simultaneously, and the defining equations of the corresponding cubic fourfolds are invariant under such an action. 
\end{enumerate}

\item $\rank(S)=15$, $G=D_8$.

\item $\rank(S)=16$, $G=A_{3,3}$, $D_{12}$, $A_4$, or  $D_{10}$. 
\begin{enumerate}[(a)]
\item  If $G=D_{12}$, then the defining equations of the corresponding cubic fourfolds either belong to
\begin{equation*}
\Span \{x_1^2 x_3, x_1^2 x_4, x_1 x_2 x_3, x_1 x_2 x_4, x_2^2 x_3, x_2^2 x_4, x_3^3, x_3^2 x_4, x_3 x_4^2, x_3 x_5 x_6, x_4^3, x_4 x_5 x_6, x_5^3, x_6^3\},
\end{equation*}
while an order $6$ element of $G$ is $\frac{1}{6}(3, 3, 0, 0, 2, 4)$, or belong to
\begin{equation*}
\Span \{x_1^3, x_1 x_2^2, x_1 x_3 x_5, x_1 x_3 x_6, x_2 x_4 x_5, x_2 x_4 x_6, x_3^3, x_3 x_4^2, x_5^3, x_5^2 x_6, x_5 x_6^2, x_6^3\},
\end{equation*}
while an order $6$ element of $G$ is $\frac{1}{6}(0, 3, 2, 5, 4, 4)$. Moreover, a generic cubic fourfold admitting such an order $6$ automorphism has symplectic automorphism group $D_{12}$.

\item If $G=D_{10}$, then for an appropriate choice of coordinates, 
\begin{equation*}
X=V\left(F_1(x_1, x_2)+x_3 x_6 L_1(x_1, x_2)+x_4 x_5 L_2(x_1, x_2)+x_3^2 x_5+x_3 x_4^2+x_4 x_6^2+x_5^2 x_6\right).
\end{equation*}
An order $5$ element in $G$ is $g=\frac{1}{5}(0,0,1, 2, 3, 4)$. Moreover, any smooth cubic fourfolds with a symplectic automorphism of order $5$ have this form, and a generic such cubic fourfold has symplectic automorphism group equal to $D_{10}$.
\end{enumerate}

\item $\rank(S)=17$, $G=S_4$ or $Q_8$. 

\item $\rank(S)=18$, $G=3^{1+4}:2$,  $A_{4,3}$, $A_5$, $3^2.4$, $S_{3,3}$, $F_{21}$, $\Hol(5)$\footnote{There is a typo in H\"ohn-Mason list \cite{HM1}: they wrote $\Hol(4)$, and claimed it has order $20$. The correct group is $\Hol(5)\cong \AGL_1(\FF_5)\cong C_4:C_5$.} or $\QD_{16}$.

\begin{enumerate}[(a)]
\item If $G=3^{1+4}:2$, then for an appropriate choice of coordinates, the defining equations of the corresponding cubic fourfolds belong to
\begin{center}
$\Span\{$\it{monomials in} $x_1, x_2, x_3$, \it{monomials in} $x_4, x_5, x_6\}$,
\end{center} 
An element of order $3$ in $G$ is $\frac{1}{3}(0, 0, 0, 1, 1, 1)$. Moreover, any smooth cubic fourfold with a symplectic automorphism which can be diagonalized as $\frac{1}{3}(0, 0, 0, 1, 1, 1)$ has this form, and a generic such cubic fourfold has symplectic automorphism group equal to $3^{1+4}:2$.
\item If $G=F_{21}$, then 
\begin{equation*}
X=V\left(x_1^2 x_2+ x_2^2 x_3+ x_3^2 x_4+ x_4^2 x_5+ x_5^2 x_6+x_6^2 x_1+ a x_1 x_3 x_5+b x_2 x_4 x_6\right).
\end{equation*} 
The automorphisms $g_1=\frac{1}{7} (1, 5, 4, 6, 2, 3)$ and $g_2:\ x_i\longmapsto x_{i+2}$ generate $F_{21}$. Moreover, any smooth cubic fourfold with symplectic automorphism of order $7$ has this form, and a generic such cubic fourfold has symplectic automorphism group equal to $F_{21}$.
\item If $G=\QD_{16}$, for an appropriate choice of coordinates, the defining equations of the corresponding cubic fourfolds belong to 
\begin{equation*}
\Span\{x_1^3, x_1 x_2^2, x_2 x_3^2, x_2 x_4^2, x_1 x_3 x_4, x_4 x_5^2, x_3 x_6^2, x_2 x_5 x_6\}.
\end{equation*} 
An element of order $8$  in $G$ is $g=\frac{1}{8}(0, 4, 2, 6, 1, 3)$. Moreover, any smooth cubic fourfold with a symplectic automorphism of order $8$ has this form, and a generic such cubic fourfold has symplectic automorphism group equal to $\QD_{16}$.
\end{enumerate}

\item $\rank(S)=19$, $G=3^{1+4}:2.2$, $A_6$, $L_2(7)$, $S_5$, $M_9$, $N_{72}(\cong 3^2. D_8)$, or $T_{48}(=M_{3,8}\cong Q_8:S_3)$\footnote{The notation $N_{72}$ and $T_{48}$ was introduced by Mukai \cite{mukaiaut} in his classification of symplectic automorphism groups for $K3$ surfaces.}. Except for the case $G=3^{1+4}:2.2$, $1$-parameter families  of cubics with automorphism group $G$ can be obtained by smoothing fake cubic fourfolds (they correspond to degree $2$ or $6$ $K3$ surfaces) with maximal symplectic symmetry, see \S\ref{subsec_k3cubic}. 

\item $\rank(S)=20$, $G=3^4:A_6$, $A_7$, $3^{1+4}:2.2^2$, $M_{10}$, $L_2(11)$ or $A_{3,5}$. More information on these cases is included in Theorem \ref{theorem: maximal uniqueness}.
\end{enumerate}
Moreover, all the $34$ pairs $(G, S)$ in above cases can arise from smooth cubic fourfold with $G$ the symplectic automorphism group. In fact, the dimension of the moduli space of cubic fourfolds with associated pair $(G,S)$ is $20-\rank(S)$. 

\noindent (Here, $F_i, N_i, L_i$ denote cubic, quadric, and linear polynomials respectively. We denote by $\frac{1}{n}(k_1,\dots,k_6)$ the diagonal matrix $(\zeta^{k_1},\dots,\zeta^{k_6})\in \SL(6)$, where $\zeta$ is a primitive $n$-root of unity.)
\end{thm}

\begin{rmk} 
The $6$ maximal cases of Theorem \ref{theorem: main}(9) were already identified by H\"ohn--Mason \cite[Table 11]{HM2} (including explicit realizations for each case), but it is not shown that they are the only possible cases for $\rank(S)=20$. This is indeed the case, but as we note in Theorem \ref{theorem: maximal uniqueness} below, in two of the six cases there are two non-isomorphic cubic fourfolds realizing the pair $(G,S)$.  
\end{rmk}

\begin{rmk}
A direct corollary of Theorem \ref{theorem: main} is that the possible orders $n$ of symplectic automorphisms $g$ for smooth cubic fourfolds are $1,2,3,4,5,6,7,8,9,11,12,15$. Furthermore, we obtain all geometric realizations for cubic fourfolds with a given order $n$ symplectic automorphism (see \S\ref{subsec-min}, esp. Theorem \ref{prop_HL}). This is a strengthening of results of Fu \cite{fu2016classification} (see also \cite{gonzalez}) who discussed the prime-power order case.
\end{rmk}

Let us briefly review the key ingredients for the proof of Theorem \ref{theorem: main}. Suppose $G$ is a finite group acting symplectically on a smooth cubic fourfold $X$. Then $G$ acts on the middle cohomology  $H^4(X,\ZZ)$, which in turn determines the covariant lattice  $S_G(X)$. Following Mongardi (with the main ideas going back to Nikulin and Kond\=o), we see that the pair $(G, S_G(X))$ can be embedded into $(\Co_0, \LL)$. More precisely, there is a primitive embedding of $S_G(X)$ into the Leech lattice $\LL$ such that the action of $G$ on $S_G(X)$ extends to a faithful action on $\LL$ with $G$ acting trivially on the orthogonal complement of $S_G(X)$ in $\LL$ (see Proposition \ref{lemma: embedding into leech lattice} and Lemma \ref{lemma: leech type pairs from cubic fourfold and K3 surface}). This leads to the abstract notion of {\it Leech pair} $(S,G)$ (Definition  \ref{def:leechpair}). Our classification theorem is essentially equivalent to the classification of Leech pairs that can arise from groups of symplectic automorphisms of cubic fourfolds. In the context of the work of Mongardi and H\"ohn-Mason, the main difference is that we are dealing with polarized hyper-K\"ahler manifolds of $K3^{[2]}$ type (specifically, $F(X)$ is of $K3^{[2]}$ type with a degree $6$ polarization). Instead of dealing directly with the natural polarization, we are using the so-called {\it Kond\=o--Scattone trick}. Namely, we note that  the primitive cohomology $\Lambda_0=H^4(X,\ZZ)_{prim}$ of a cubic fourfold $X$ admits a unique primitive embedding into the Borcherds lattice (i.e.,  $\II_{26,2}$, the unique even unimodular lattice of signature $(26,2)$) with orthogonal complement $E_6$ (Lemma \ref{lem:borcherdse6}). This allows us to view $X$ (or equivalently $F(X)$) as being {\it $E_6$ Borcherds polarized}\footnote{This should  to be understood in the context of $M$-polarized $K3$ surfaces in the sense of Dolgachev \cite{dolgachev_M}, but here we use the Borcherds lattice, instead of the $K3$ lattice, as the ambient lattice. Regarding the $K3$ surfaces (or hyper-K\"ahler manifolds) as being Borcherds polarized is a powerful arithmetic trick well-known to experts. The first author learned about it from Kond\=o long time ago. Presumably, the first use of this construction occurs in the thesis of Scattone \cite{scattone}. }. Using this perspective, we are able to formulate  a lattice theoretic criterion (Theorem \ref{theorem: criterion lattice from cubic fourfold}) for a Leech pair $(G, S)$  to arise as $(G, S_G(X))$ for $G=\Aut^s(X)$ for a smooth cubic fourfold $X$. Finally, using this criterion, the classification of fixed-point sublattices in $\LL$ (\cite{HM1}, \cite{harada1990leech}), and a case by case analysis, we are able to complete the proof of Theorem \ref{theorem: main}. One complication that we deal with is the possibility that a Leech pair $(G,S)$ (which is compatible with the $E_6$ Borcherds polarization) might lead to some fake cubic fourfolds,  i.e.,  either singular cubics (with ADE singularities) or degenerations to the secant to the Veronese surface (see \cite{laza2010moduli}). These form the divisors $\calC_6$ and $\calC_2$ excluded from the image of the period map (see Theorem \ref{theorem: thmperiods}). Geometrically and motivically, the divisors $\calC_6$ and $\calC_2$  are naturally associated with $K3$ surfaces $Y$ (or hyper-K\"ahler $Y^{[2]}$) of degree $6$ and $2$ respectively.  It turns out (see \S\ref{subsec_k3cubic}) that any symplectic automorphism of a degree $2$ or $6$ $K3$ surface can be lifted to an automorphism of a singular cubic fourfold $X_0$, which can then be smoothed, while preserving the automorphism. In particular, all $\rank (S)=19$ cases, except for the $3^{1+4}:2.2$ case, of Theorem \ref{theorem: main}(8) can be recovered by starting with a degree $2$ or $6$ $K3$ surface with a maximal group of symplectic automorphisms.

\begin{rmk}[Automorphisms of low degree $K3$ surfaces]
\label{remark: polarized K3}
The Kond\=o--Scattone trick can also be applied to polarized $K3$ surfaces. Namely, the primitive middle cohomology of a degree $d$ $K3$ surface can be embedded (up to a Tate twist) into the Borcherds lattice. The complement of this embedding is a rank $7$ positive lattice $M$ with discriminant form $(-\frac{1}{d})$.   For  the low degree cases, degree $2$, $4$, and $6$, $M$ can be chosen to be $E_7$,  $D_7$, and $E_6\oplus A_1$ respectively. Similarly, the elliptic $K3$ surfaces can be viewed as $E_8$ Borcherds polarized. For these cases, our arguments can be easily adapted. In particular, in section \ref{section: K3}, we discuss briefly the case of degree $2$ and $6$ $K3$ surfaces, as they are closely related to cubic fourfolds. 
\end{rmk}

 \begin{rmk}[Automorphisms of low dimensional cubics] The possible automorphism groups for cubic surfaces were classified by Segre \cite{segre1942cubic} (see \cite{hosoh} for a modern and corrected account). From our perspective, the salient point is that,  for smooth cubic surfaces (and similarly cubic threefolds), the induced action of the automorphism groups on the middle cohomology is faithful. This realizes the automorphism group of a cubic surface as a subgroup of $W(E_6)$. For cubic threefolds, we are not aware of a systematic study of their automorphism groups (see \cite{gonzalez,gonzalez2}, \cite{adler1978automorphism} for some results). Using the period map of Allcock--Carlson--Toledo \cite{allcock2011moduli} (see also \cite{looijenga2007period}), and ideas from this paper, we are able to relate the classification of the automorphisms groups for cubic threefolds to the Suzuki sporadic group $\Suz$ (N.B. an index $6$ extension of $\Suz$ is isomorphic to the centralizer of an order $3$ element in $\Co_0$; see \cite{wilson1983suzuki}). To our knowledge this relationship is new, we plan to return to it in future work.
 \end{rmk}
 
 We note that once a Leech pair $(G,S)$ as in Theorem \ref{theorem: main} is specified, one obtains a moduli space $\calM_{(G,S)}$ of dimension $20-\rank(S)$ parametrizing cubic fourfolds $X$ with $G\subset \Aut^s(X)$ (e.g., see \cite[Ch. 5]{mongardi2013thesis}, \cite{yu2018moduli}). However, it is not necessary that this moduli space is irreducible. This corresponds to $S$ having different primitive embeddings into the primitive lattice $\Lambda_0$ for cubic fourfolds (the existence of the embedding $S\hookrightarrow \Lambda_0$ is essentially the content of Theorem \ref{theorem: main}). 
 It is thus a natural question to study uniqueness of $S\hookrightarrow \Lambda_0$ for the pairs $(G,S)$ occurring in Theorem \ref{theorem: main}. The analogue question for (unpolarized) $K3$ surfaces was studied by Hashimoto \cite{hashimoto2012K3} (similarly, for polarized symplectic involutions, see \cite{vGS}). Here, we are restricting ourselves to the maximal cases (i.e., $\rank(S)=20$), as those are the most interesting cases. For instance, these cases give interesting examples of maximal algebraic cubics (in the sense of maximal possible rank for the group of algebraic cycles $H^4(X,\ZZ)\cap H^{2,2}(X)$; equivalently the transcendental lattice $T$ is negative definite of rank $2$). We obtain a somewhat surprising result: while there are $6$ groups that occur (cf. Theorem \ref{theorem: main}(9)), there are $8$ cubic fourfolds (automatically isolated in moduli) corresponding to them. Six out of the eight cases are identified in \cite[Table 2]{HM2}; we are not able to give equations for the remaining two special cubics (cases $X^2(A_7)$ and $X^2(M_{10})$ below).

\begin{notation}
We denote by $a^b c$ the rank $2$ quadratic form
$\left(
\begin{array}{cc}
a & b\\
b & c
\end{array}
\right)$.
We write $-(a^b c):=(-a)^{(-b)}(-c)$. 
\end{notation}
\begin{thm}
\label{theorem: maximal uniqueness}
Let $(G,S)$ be a Leech pair such that $\rank(S)=20$ and there exists a smooth cubic fourfold $X$ with $G=\Aut^s(X)$ and $(G, S)\cong (G, S_G(X))$. We denote by $T$ the orthogonal complement of $S$ in $H^4_0(X,\ZZ)$.
Then we have and only have the following possibilities:
\begin{enumerate}[(1)]
\item $G=3^4:A_6$, the corresponding cubic fourfold is the Fermat one 
\begin{equation*}
X(3^4:A_6)=V(x_1^3+x_2^3+x_3^3+x_4^3+x_5^3+x_6^3)
\end{equation*}
and $T=-(6^3 6)=A_2(-3)$. Moreover, this is the only smooth cubic fourfold with a symplectic automorphism of order $9$. It holds $\Aut(X)/\Aut^s(X)\cong \ZZ/6$.

\item $G=A_7$, there are two smooth cubic fourfolds with symplectic action of $G$. One of them is 
\begin{equation*}
X^1(A_7)=V(x_1^3+x_2^3+x_3^3+x_4^3+x_5^3+x_6^3-(x_1+x_2+x_3+x_4+x_5+x_6)^3) 
\end{equation*}
and $T=-(2^1 18)$. It holds $\Aut(X^1)/\Aut^s(X^1)\cong \ZZ/2$. The other one $X^2(A_7)$ has $T=-(18^3 18)$ and admits no non-symplectic automorphisms.

\item $G=3^{1+4}:2.2^2$, the cubic fourfold is 
\begin{equation*}
\label{equation: 3^{1+4}:2.2^2}
X(3^{1+4}:2.2^2)=V(x_1^3+x_2^3+x_3^3+x_4^3+x_5^3+x_6^3-3(\sqrt{3}+1)(x_1 x_2 x_3+x_4 x_5 x_6))
\end{equation*}
and $T=-(6^0 6)=(A_1\oplus A_1)(-3)$. Moreover, this is the only smooth cubic fourfold with a symplectic automorphism of order $12$.  It holds $\Aut(X)/\Aut^s(X)\cong \ZZ/4$.

\item $G=M_{10}$, there are two smooth cubic fourfolds with symplectic action of $G$, and both of them have $T=-(12^0 30)$. See Equation \eqref{equation: m10} for explicit description for one such cubic fourfold which is denoted by $X^1(M_{10})$. The other one is denoted by $X^2(M_{10})$. Both $X^1(M_{10})$ and $X^2(M_{10})$ have no non-symplectic automorphisms.

\item $G=L_2(11)$, the cubic fourfold is 
\begin{equation*}
X(L_2(11))=V(x_1^3+x_2^2 x_3+x_3^2 x_4+ x_4^2 x_5+x_5^2 x_6+x_6^2 x_2)
\end{equation*}
and $T=-(22^{11}22)=A_2(-11)$. Moreover, this is the only smooth cubic fourfold with a symplectic automorphism of order $11$. It holds $\Aut(X)/\Aut^s(X)\cong \ZZ/3$.

\item $G=A_{3,5}$, the cubic fourfold is 
\begin{equation*}
X(A_{3,5})=V(x_1^3+x_2^3+x_3^3+x_4^3+x_5^3+x_6^3+x_7^3+x_8^3)\cap V(x_1+x_2+x_3)\cap V(x_4+x_5+x_6+x_7+x_8)
\end{equation*}
and $T=-(10^5 10)=A_2(-5)$. Moreover, this is the only smooth cubic fourfold with a symplectic automorphism of order $15$. It holds $\Aut(X)/\Aut^s(X)\cong \ZZ/6$.

\end{enumerate}
\end{thm}

\begin{rmk}
The transcendental lattice $T$ for  cubic fourfolds with nontrivial symplectic automorphisms is relatively small (of rank at most $22-\rank(S)$). It follows that except the case of symplectic involutions (i.e. Theorem \ref{theorem: main}(1)), $T$ 
embeds into the $K3$ lattice $(E_8)^2\oplus U^3$ (cf. \cite[Prop. 2.5, Cor. 2.9]{maxalg}). Thus, a priori, the cubics with large group of symplectic automorphisms  are not interesting from the perspective of the standard rationality conjectures (we refer to \cite{maxalg} and \cite{Ouchi} for further discussion on the subject). Nonetheless, they are Hodge theoretically interesting as the cases of Theorem  \ref{theorem: maximal uniqueness} give examples of maximally algebraic cubics (i.e., with maximal rank for $H^{2,2}(X)\cap H^4(X,\ZZ)$) for which the transcendental lattice is explicitly known. 
\end{rmk}

\subsection*{Structure of the paper}
In section \ref{section: cubic fourfolds}, we introduce and briefly review the properties of the period map for cubic fourfolds. Additionally, in \S\ref{subsec:borcherds}, we review the notion of Borcherds marking for cubic fourfolds. In the following section \ref{section: automorphisms and conway group}, we review the necessary material on the Leech lattice, Niemeier lattices, and Conway group. These two review sections (specific to our situation) are complemented by two appendix sections, which cover very standard material, but which nonetheless might be helpful to the reader. Specifically, in Appendix \ref{section: collection of results in lattice theory}, we collect results in lattice theory (mostly due to Nikulin) which are essential in our arguments. In Appendix \ref{appendix: finite groups}, we review some basic facts and notations for finite groups. 

The main content going into the proof of Theorem \ref{theorem: main} is discussed in sections \ref{section: automorphisms and conway group} and \ref{section: symplectic automorphism of cubic fourfold}. First, following Mongardi's work, we introduce the notion of Leech pair (Definition \ref{def:leechpair}), and give a key lemma (Lemma \ref{lemma: embedding into leech lattice}). We then focus on the polarized case. In particular, we establish a criterion (Theorem \ref{theorem: criterion lattice from cubic fourfold}) for a Leech pair $(G,S)$ to arise from a group of symplectic automorphisms for some cubic fourfold $X$. In \S 4.5 we prove Theorem \ref{theorem: maximal uniqueness} using methods from Lattice theory.

The remaining two sections are complementing our main classification result. Namely,  in section \ref{section: K3}, we partially discuss the completely analogous (and somewhat easier) situation for degree $2$ and $6$ $K3$ surfaces. Finally, while the focus of this paper is on symplectic automorphisms, we make some comments on the non-symplectic case in section \ref{section: non-symplectic}. In particular, we determine the full automorphism groups of the $8$ maximal cases of Theorem \ref{theorem: maximal uniqueness} (see Proposition \ref{proposition: non-sym order}). This allows us to distinguish geometrically the two cases of Theorem \ref{theorem: maximal uniqueness}(2) with $\Aut^s(X)\cong A_7$ (i.e., one has a anti-symplectic involution, while the other does not). As a consequence of this classification, we also  obtain that the maximal possible order of automorphism group for a cubic fourfold is $174,960$ which is reached only by the Fermat cubic fourfold (an analogous result for $K3$ surfaces was obtained by Kond\=o \cite{kondo_ns}).

\subsection*{Acknowledgement}  Most of the work was done while the second author visited Stony Brook during the Spring 2018 semester. His stay was supported by Tsinghua Scholarship for Overseas Graduate Studies. He thanks Stony Brook for hosting him and he is grateful to his advisor, Eduard Looijenga, for constant support and helpful discussions on related topics. The research of the first author was partially supported by NSF grants DMS-1254812 and DMS-1802128.

After the posting of our manuscript, we have learned of the work of Ouchi \cite{Ouchi}, who explores the interplay between automorphisms of cubic fourfolds and the automorphisms of the associated K3 category (the Kuznetsov component). We thank G. Ouchi for sharing an early version of his work, and for some comments on our paper. We are also grateful to S. Mukai for sharing with us some of his partial work on the classification of automorphisms of cubic fourfolds (from late eighties). As a consequence, we have updated some of our notations (and added some remarks) to be aligned with Mukai's work.

\section{Automorphisms and periods} 
\label{section: cubic fourfolds}
In this section we review some well-known facts, which are the starting point of  our classification of the automorphism groups for cubic fourfolds. 
 First, the Global Torelli Theorem (Thm. \ref{theorem: thmperiods} and Prop. \ref{proposition: strong global torelli}) allows one to reduce the classification of automorphisms for cubic fourfolds to the classification of automorphisms of Hodge structures, which in turn is essentially a lattice theoretic question. Classically, this approach was successfully applied to the case of $K3$ surfaces (Nikulin, Mukai, Kond\=o and others). More recently, it was (partially) adapted to the case of hyper-K\"ahler manifolds of $K3^{[n]}$ type. The Fano variety $F(X)$ of a cubic fourfold $X$ is a hyper-K\"ahler of $K3^{[2]}$ type. Thus, the classification of automorphisms of $X$ is closely related to the classification of automorphisms of $F(X)$. We review this in \S\ref{sect-fano} below. Finally, the  difference to most of related work that we cite is that we need to keep track of the polarization. It turns out that it is better to keep track of a ``Borcherds polarization'' instead of the natural polarization of $X$ (or equivalently $F(X)$). We introduce this notion in \S\ref{subsec:borcherds}.

\subsection{Periods for cubics}
\label{subsection: periods for cubics}
Let $X$  be a smooth cubic fourfold. The middle cohomology group $H^4(X,\ZZ)$, with the natural intersection pairing, is a unimodular odd lattice $\Lambda$ of signature $(21,2)$ (uniquely specified by these conditions). Let $\eta_X\in H^4(X,\ZZ)$ be the square of the hyperplane class of $X$. The primitive cohomology $H^4(X,\ZZ)_{prim}=\langle \eta_X\rangle^\perp$ carries a polarized Hodge structure of $K3$ type (i.e., Hodge numbers $(0,1,20,1,0)$).  As lattice,  $H^4(X,\ZZ)_{prim}\cong \Lambda_0$ where $\Lambda_0:=(E_8)^2\oplus U^2\oplus A_2$ (with $A_2$ and $E_8$ the standard root lattices, and $U$ the hyperbolic plane). Similarly to the well-known case of $K3$ surfaces, the period domain for Hodge structures on $H^4(X, \ZZ)_{prim}$ is the $20$-dimensional Type $\IV$ period domain 
\begin{equation*}
\calD=\{x\in \PP((\Lambda_0)_{\CC})\big{|}(x,x)=0, (x,\overline{x})<0\}^+
\end{equation*}
(where the script $+$ indicates a choice of one of the two connected components).

Associated to the lattice $\Lambda_0$,  there are several natural groups: 
\begin{enumerate}[(1)]
\item $\sO(\Lambda_0)$ the automorphism group of lattice $\Lambda_0$;
\item $\widetilde{\sO}(\Lambda_0)$ the subgroup of $\sO(\Lambda_0)$ which acts trivially on the discriminant group $A_{\Lambda_0}(=(\Lambda_0)^\vee/\Lambda_0\cong \ZZ/3)$;
\item $\sO^{+}(\Lambda_0)$ the subgroup of $\sO(\Lambda_0)$ which preserves the spinor norm on $\Lambda_0$ (or equivalently preserves $\calD$);
\item $\sO^*(\Lambda_0)\coloneqq \sO^{+}(\Lambda_0)\cap\widetilde{\sO}(\Lambda_0)$.
\end{enumerate}
The global monodromy group $\Gamma$ for cubic fourfold is $\sO^*(\Lambda_0)$ (cf. Beauville \cite{beauville1986groupe}). Since $\Gamma=\sO^*(\Lambda_0)$ is an arithmetic group, $\Gamma$ acts properly discontinuously on $\calD$. The resulting analytic variety $\calD/\Gamma$ is in fact a quasi-projective variety; we refer to it as the global period domain for cubic fourfolds.

\begin{defn}
\label{definition: short/long roots}\begin{enumerate}[(i)]
\item A norm $2$ vector $v$ in $\Lambda_0$ is called a {\it short root}. The set of short roots in $\Lambda_0$ determines a $\Gamma$-invariant hyperplane arrangement $\calH_6$ in $\D$. Let  $\calC_6:=\calH_6/\Gamma\subset \D/\Gamma$ be the associated  Heegner divisor. 
\item A norm $6$ vector $v$ in $\Lambda_0$ with divisibility $3$ is called a {\it long root}. The set of long roots in $\Lambda_0$ determines a $\Gamma$-invariant hyperplane arrangement $\calH_2$ in $\D$. Let $\calC_2:=\calH_2/\Gamma\subset \D/\Gamma$.
\end{enumerate}
\end{defn}

\begin{rmk} It is well known that there exists a single $\{\pm1\}\times\Gamma$-orbit of short and long roots respectively, and thus $\calC_6$ and $\calC_2$ are irreducible divisors. Furthermore,   $\Gamma(=\sO^*(\Lambda_0))$ is generated by reflections in short roots (\cite{beauville1986groupe}), and $\Gamma$ has index $2$ in $\widetilde{\sO}(\Lambda_0)$ with $\widetilde{\sO}(\Lambda_0)/\Gamma$ generated by the class of a reflection in a long root.
\end{rmk}

Let $\calM$ be the moduli space of smooth cubic fourfolds. It is a quasi-projective $20$-dimensional variety, which can be constructed by GIT (see \cite{gitcubic} for a full GIT analysis). By associating with a cubic fourfold $X$, the Hodge structure on its middle cohomology, one obtains a period map 
\begin{equation*}
\Prd\colon \calM\longrightarrow \calD/\Gamma.
\end{equation*}
Voisin \cite{voisin} proved that the Global Torelli Theorem is valid for cubic fourfolds. It follows that $\Prd$ is an open embedding. 
For the purpose of this paper, it is important to understand also the image of the period map $\Prd(\calM)\subset \calD/\Gamma$. This type of question was first investigated by Hassett \cite{hassett}. In particular, he defined certain Heegner divisors $\calC_d$ in $\calD/\Gamma$ (indexed by $d\in \ZZ_+$ with $d\equiv 0,2\pmod 6$) corresponding to cubic fourfolds containing additional Hodge classes. The relevant divisors here are $\calC_2=\calH_2/\Gamma$ and $\calC_6=\calH_6/\Gamma$ as defined above. Geometrically, $\calC_6$ corresponds to singular cubic fourfolds, while $\calC_2$ correspond to degenerations of cubics to the secant to Veronese surface in $\PP^5$. The image of the period map misses the divisors $\calC_2$ and $\calC_6$. Conversely, as shown by Laza \cite{laza2010moduli} and Looijenga \cite{lcubic}, any period outside these two divisors is realized for some smooth cubic fourfold. 

\begin{thm}[Voisin, Hassett, Laza, Looijenga]
\label{theorem: thmperiods}
The period map for cubic fourfolds gives an isomorphism of quasi-projective varieties
\begin{equation}\label{eq_globalT}
\calP\colon \calM\xrightarrow{\sim} \left(\calD\setminus(\calH_2\cup \calH_6)\right)/\Gamma.
\end{equation} 
\end{thm}

We note that both sides of \eqref{eq_globalT} have natural orbifold structures. For instance, since any smooth cubic fourfold is GIT stable (\cite{gitcubic}), the moduli space of smooth cubic fourfolds is a smooth Deligne-Mumford stack $\mathfrak M$ with quasi-projective coarse moduli space $\calM$. A natural question is  whether the period map $\Prd$ identifies the two sides  of  \eqref{eq_globalT} as orbifolds. This is equivalent to the Strong Global Torelli Theorem, i.e., any isomorphism between the polarized Hodge structures of two smooth cubic fourfolds is induced by a unique isomorphism between the two cubic fourfolds. Using the fact that automorphisms of cubic fourfolds $X$ are induced by linear transformations of the ambient projective space $\PP^5$, and that $\Aut(X)$ acts faithfully on the middle cohomology $H^4(X, \ZZ)$  (e.g., 
\cite[Proposition 2.16]{javanpeykar2017complete}), the second author \cite{zheng2017} has verified the Strong Global Torelli Theorem.

\begin{prop}[{\cite{zheng2017}}]
\label{proposition: strong global torelli}
Let $X_1$ and $X_2$ be two smooth cubic fourfolds. Assume that there is an isomorphism 
$$\varphi\colon H^4(X_2, \ZZ)\cong H^4(X_1, \ZZ)$$ 
of polarized Hodge structures (in particular $\varphi(\eta_{X_2})=\eta_{X_1}$). Then, there exists a unique isomorphism $f\colon X_1 \cong X_2$ such that $\varphi=f^*$. In particular, for any smooth cubic fourfold $X$, 
\begin{equation}
\Aut(X)\cong \Aut_{HS}(H^4(X,\ZZ),\eta_X),
\end{equation}
where $\Aut_{HS}$ stands for group of Hodge isometries.
\end{prop}

\begin{rmk}
We note that while the period map extends to an isomorphism of quasi-projective varieties 
$$\calM^{ADE}\cong (\calD\setminus \calH_2)/\Gamma$$
where $\calM^{ADE}$ is the moduli space of cubics with ADE singularities (see \cite{gitcubic,laza2010moduli}), the orbifold structure along the discriminant divisor is different. Simply, a general cubic fourfold with a node (i.e., $A_1$ singularity) has no automorphism, while on the periods side, there is a $\ZZ/2$ stabilizer corresponding to the reflection in a short root. 
\end{rmk}

\begin{rmk}[{$M$-polarized cubic fourfolds}]
Let $M$ be a positive definite lattice with a fixed primitive embedding into the primitive cubic lattice $\Lambda_0$. Assume that $M$ does not contain short or long roots. Then, in analogy to the theory of Dolgachev \cite{dolgachev_M} for $K3$ surfaces,  one can define a moduli space $\calM_M$ of cubics with the specified primitive algebraic lattice $M$ (i.e., $M\subseteq H^{2,2}(X)\cap H^4(X,\ZZ)_{prim}\subset H^4(X,\ZZ)_{prim} \cong \Lambda_0$ such that the composition $M\subset \Lambda_0$ is equivalent to the fixed embedding). Up to passing to normalization, $\calM_M$ is (the complement of some Hegneer divisors in) a locally symmetric variety $\calD_M/\Gamma_M$, where $\calD_M$ is the Type $\IV$ domain associated with the transcendental lattice $T=M^\perp_{\Lambda_0}$. Thus,  $\dim \calM_M=20-\rank M$. Furthermore, if $M\subset M'\subset \Lambda_0$ (primitive embeddings) then $\calM_{M'}\subset \calM_M$ (i.e., the more algebraic cycles, the smaller the moduli). The moduli of cubic fourfolds $\calM$ corresponds to $M=\emptyset$, and the Hassett divisors $\calC_d$ correspond to $\rank (M)=1$. (Equivalently, as in Hassett's work, one can consider the full lattice of algebraic cycles $\widetilde M=\mathrm{Sat}(M\oplus \langle \eta\rangle)_\Lambda\subset \Lambda\cong H^4(X,\ZZ)$. Here, it is more convenient to work with the primitive lattices $M$ and $\Lambda_0$.)
\end{rmk}

\subsection{The hyper-K\"ahler fourfold associated with a cubic fourfold $X$}\label{sect-fano}
For a smooth cubic fourfold $X$, the Fano variety $F(X)$ of lines on $X$  is a smooth hyper-K\"ahler fourfold, deformation equivalent to $K3^{[2]}$ (cf. \cite{beauville1985variety}). There is a natural  polarization on $F(X)$ induced from the Pl\"ucker embedding $F(X)\hookrightarrow \Gr(1, \PP^5)\subset \PP(\wedge^2(\CC^6))$. Since any automorphism of $X$ is linear, there is a natural group homomorphism 
$$\Aut(X)\longrightarrow \Aut(F(X)).$$ 
Conversely, the following holds  (e.g., \cite[Proposition  4]{charles2012remark},  \cite[Corollary  2.3]{fu2016classification}):
\begin{prop}
\label{proposition: aut of cubic and fano}
The homomorphism $\Aut(X)\longrightarrow \Aut(F(X))$ is injective with image the subgroup preserving the Pl\"ucker polarization on $F(X)$.
\end{prop}

An automorphism of a hyper-K\"ahler manifold sends $H^{2,0}$ to $H^{2,0}$, hence induces a scalar action on $H^{2,0}$. If the scalar is the identify, the automorphism is called  {\it symplectic}. Otherwise, it is called {\it non-symplectic}. Adapting this to the case of cubic fourfolds, we make the following definition:
\begin{defn}
\label{definition: symplectic}
An automorphism of a smooth cubic fourfold $X$ is called {\it symplectic}, iff the induced automorphism on $F(X)$ is symplectic. Equivalently, an automorphism of $X$ is symplectic iff the induced action on $H^{3,1}(X)$ is the identity. We denote the group of symplectic automorphisms of $X$ by $\Aut^s(X)$. 
\end{defn}

\begin{rmk}\label{m-pol-cubic}
In view of Theorem \ref{theorem: thmperiods} and Proposition \ref{proposition: strong global torelli}, it is clear that essential arithmetic input in the classification of automorphisms of cubic fourfolds is the primitive cohomology lattice $\Lambda_0=H^4(X,\ZZ)_{prim}\cong A_2\oplus (E_8)^2\oplus U^2$. Let us note that the associated hyper-K\"ahler $F(X)$ has the same primitive lattice. More precisely, $H^2(F(X), \ZZ)$ carries a natural quadratic form, the so-called  Beauville--Bogomolov quadratic form. With respect to this form, there is a natural lattice isometry $H^2_0(F(X), \ZZ)(-1)\cong H^4_0(X, \ZZ)$, which is also an isomorphism of Hodge structures (see \cite[Proposition  6]{beauville1985variety}). In particular, via this isomorphism $H^{2,0}(F(X))$ maps to $H^{3,1}(X)$, justifying our definition above. In summary, the discussion of this subsection says that the classification of the automorphisms of cubic fourfolds is essentially equivalent to the classification of automorphisms of degree $6$ (the degree of the Pl\"ucker polarization) polarized hyper-K\"ahler manifolds of $K3^{[2]}$ type.
\end{rmk}

\begin{rmk}\label{rmk-lattice}
One should note that there is a subtle difference to the case of $K3$ surfaces. While for $K3$ surfaces the full cohomology lattice $H^2(S,\ZZ)$ is even unimodular, the full cohomology lattice for cubic fourfolds $H^4(X,\ZZ)$ is \underline{odd} unimodular. If one prefers to work with hyper-K\"ahler manifolds of $K3^{[2]}$ type, we note that the full cohomology lattice (w.r.t. the Beauville--Bogomolov form) is even, but \underline{not} unimodular (it is (up to sign) $A_1\oplus (E_8)^2\oplus U^3$). 
\end{rmk}

\subsection{Borcherds polarizations}\label{subsec:borcherds} In view of Nikulin's theory  \cite{nikulin1980integral}, it is preferable to work with even unimodular lattices (compare Remark \ref{rmk-lattice}). The smallest (with definite orthogonal complement) even unimodular lattice that contains the primitive cubic lattice $\Lambda_0$ is the {\it Borcherds lattice} $\BB$, i.e., the unique even unimodular lattice $\II_{26,2}\cong (E_8)^3\oplus U^2$ of signature $(26,2)$. (Here, we prefer to denote it $\BB$ and call it the Borcherds lattice in honor of Borcherds, who studied the automorphic forms on the associated Type $\IV$ symmetric domain.)

\begin{rmk}
Even in the $K3$ case, the embedding of the primitive cohomology lattice for a polarized $K3$ surface into the Borcherds lattice $\BB$ turns out to be a powerful arithmetic trick (the geometric reason why it works is not yet completely understood). As examples of applications of this artifice (that we baptized {\it Kond\=o--Scattone trick}), we mention Scattone's work \cite{scattone} on the Baily-Borel compactification for polarized $K3$ surfaces, Kond\=o's work \cite{kondo} on symplectic automorphisms, and the Gristsenko--Hulek--Sankaran work \cite{GHS} on the Kodaira dimensions on the moduli spaces of $K3$ surfaces. 
\end{rmk}

\begin{rmk}
We recall that there exist $24$ even unimodular lattices of rank $24$, called the {\it Niemeier lattices} (see \S\ref{subsec_niemeier} below). What is relevant here is to note that these lattices are intricately related to the Borcherds lattice $\BB$. Namely, given a Niemeier lattice $N$, then $\BB\cong N\oplus U^2$. Conversely,  the classification of the Niemeier lattices follows from the classification of isotropic vectors in the hyperbolic lattice $\II_{25,1}$ (see \cite{conway1999spherepackings}), or equivalently the Type $\II$ boundary components (i.e., rank $2$ totally isotropic subspaces in $\BB$) of the Baily--Borel compactification for the Borcherds period domain. 
\end{rmk}

Returning to cubic fourfolds, in analogy with the work of $M$-polarized $K3$ surfaces of Dolgachev \cite{dolgachev_M} and Remark \ref{m-pol-cubic}, we can view a cubic fourfold as being {\it Borcherds $E_6$-polarized} (i.e., $\Lambda_0$ admits a primitive embedding into $\BB$ with orthogonal complement $E_6$). More interestingly, the periods missing from the image of the period map for cubic fourfolds (see Theorem \ref{theorem: thmperiods}), i.e., the divisors $\calC_2$ and $\calC_6$, correspond to $E_7$ and $E_6+A_1$ Borcherds polarizations respectively. This allows a more uniform view on ``singular'' cubic fourfolds (i.e., singular cubics, or degenerations to the Veronese surface) -- simply $X$ is singular if it acquires an additional root (i.e., the existing ``algebraic'' lattice $E_6$ is enlarged to either $E_7$ or $E_6+A_1$ by adding a  root). This is of course equivalent to the more classical view of Hassett \cite{hassett} where $H^4_{alg,prim}=H^4(X,\ZZ)_{prim}\cap H^{2,2}$ acquires a short root (equivalently, in terms of Borcherds polarizations $E_6\subset E_6+A_1$) or long root (case $E_7$). From either perspective, the transcendental lattices (for $R$ Borcherds polarized objects, the transcendental lattice is $R^\perp_{\BB}$) for the two cases are 
\begin{eqnarray*}
\Lambda_2&:=&\langle 2\rangle\oplus (E_8)^2\oplus U^2, \textrm{ and} \\
\Lambda_6&:=&\langle 6\rangle\oplus (E_8)^2\oplus U^2
\end{eqnarray*} for $\calC_2$ and $\calC_6$ respectively. One recognizes the two lattices (up to a sign) $\Lambda_2$ and $\Lambda_6$ as the primitive lattices for degree $2$ and respectively $6$ $K3$ surfaces. There is indeed a close geometric relationship between degree $6$ (and respectively degree $2$) $K3$ surfaces and singular cubic fourfolds (respectively degenerations to the Veronese surface); see \cite{hassett}, \cite{laza2010moduli}.

From the perspective of this paper, the relevant fact is the following easy proposition (see \cite[\S6]{laza2010moduli}).
\begin{prop}\label{lem:borcherdse6}
\begin{enumerate}[(i)]
\item There is a unique primitive embedding of $\Lambda_0$ into $\BB$, with orthogonal complement $E_6$; in another words, $\Lambda_0\oplus E_6$ can be saturated as $\BB$ in a unique way.
\item There is a unique primitive embedding of $\Lambda_6$ into $\BB$, with orthogonal complement $A_1\oplus E_6$; in another words, $\Lambda_0\oplus A_1\oplus E_6$ can be saturated as $\BB$ in a unique way.
\item There is a unique  primitive embedding of $\Lambda_2$ into $\BB$, with orthogonal complement $E_7$; in another words, $\Lambda_0\oplus E_7$ can be saturated as $\BB$ in a unique way.
\end{enumerate}
\end{prop}

\section{Automorphisms and the Conway group}\label{section: automorphisms and conway group}
Via the Global Torelli Theorem, we have reduced the study of automorphisms for cubic fourfolds to the study of automorphisms of Hodge structures. This is in turn a question about the symmetries (satisfying certain properties) of the underlying cohomology lattice $L$.
\begin{notation}
For a lattice $L$ with action by a group $G\subset O(L)$, we call $L^G\coloneqq\{x\in L\big{|} gx=x, \forall g\in G\}$ {\it the invariant sublattice}, and $S_G(L)\coloneqq (L^G)_L^{\perp}$ {\it the covariant lattice}\footnote{Some authors call the quotient $M/S_G(M)$ the covariant lattice, since it is the maximal quotient such that the induced action of $G$ on it is trivial.}.
\end{notation}

In the case of a finite group of {\it symplectic} automorphisms $G$ acting on the cohomology lattice $L$, Nikulin made two key observations:
 \begin{itemize}
 \item[i)] {\it the covariant lattice $S_G(L)$ is a definite lattice} (this is equivalent to the symplectic condition), and 
 \item[ii)] {\it $S_G(L)$ does not contain any \underline{effective} algebraic cycle} (in fact, the symplectic condition implies that the algebraicity is automatic). In particular, for $K3$ surfaces, by Riemann-Roch, $S_G(L)$ (which is negative definite in this case) should not contain any $-2$ classes (or equivalently roots). 
 \end{itemize}
 The same holds for hyper-K\"ahler manifolds of $K3^{[n]}$ type (e.g., by involving Markman's theory of prime exceptional divisors) and for cubic fourfolds (i.e., there is no norm $2$ vector in $S_G(L)$; e.g., as a consequence of Theorem \ref{theorem: thmperiods}). Normally, one would try to classify $S_G(L)$ and its embeddings into the cohomology lattice $L$. However, using Nikulin's theory, Kond\=o made the observation that (in the geometric situations considered here: $K3$s, $K3^{[n]}$, or cubics) {\it $S_G(L)$ embeds into one of the Niemeier lattices $N$}, and furthermore {\it $G$ extends to an isometry of $N$} (thus $G\subset O(N)$). Niemeier lattices $N$ show up here since they are  the smallest even unimodular \underline{definite} lattices $N$ containing $S_G(L)$ for any $G$. The lattice $N$ being definite is important as the associated orthogonal group $O(N)$ is finite.  Kond\=o \cite{kondo} successfully applied this approach to the classification of symplectic automorphisms for $K3$ surfaces. Kond\=o avoids the Leech lattice $\LL$ (namely, he noted that $A_1\oplus S_G(L)$ embeds into $N$ for $K3$ surfaces, and thus $N\neq \LL$), but in fact, since $S_G(L)$ contains no roots, it is possible to embed it into the Leech lattice $\LL$ (cf. \cite{gaberdiel2012symmetries}, \cite{huybrechtsaut}). Considering embeddings into the Leech lattice $\LL$ leads to a more uniform behavior. Note however that there is a trade-off here: we deal with a single larger group $\Co_0:=O(\LL)$ versus $23$ smaller groups $O(N)$ for $N\neq \LL$. With the advent of more powerful computational tools, and a better understanding of the Leech lattice (esp. relevant here is \cite{HM1}), we can work throughout with the Leech lattice.

In this section, we briefly review the Leech lattice, the Conway group, and introduce the key concept (due to Mongardi, but with origins going back to Nikulin) of {\it Leech pair}. We then close with the H\"ohn--Mason \cite{HM1} classification of the fixed-point lattices for the Leech lattice $\LL$. The material here is standard (and applies equally to $K3$s and $K3^{[n]}$s); we will apply it in the following section to the actual classification of the automorphisms of cubic fourfolds. 
\subsection{The Leech Lattice and the Conway group}\label{subsec_niemeier} We recall the following classification result of Niemeier. 

\begin{thm}[Niemeier]
\label{thm_niemeier} 
Up to isometry, there exist $24$ even unimodular positive definite lattices $N$ of rank $24$. Let $R\subset N$ be the sublattice spanned by the roots (i.e., norm $2$ vectors) of $N$.  Then $R$ is of one of the following $24$ types: $\varnothing, 24A_1, 12A_2, 8A_3, 6A_4, 6D_4, 4A_5\oplus D_4, 4A_6, 2A_7\oplus 2D_5, 3A_8, 4D_6, 2A_9\oplus D_6, 4E_6, A_{11}\oplus D_7\oplus E_6$, $2A_{12}$, $3D_8, A_{15}\oplus D_9, D_{10}\oplus 2E_7, A_{17}\oplus E_7, 2D_{12}, A_{24}, 3E_8, D_{16}\oplus E_8, D_{24}$. In particular, $R$ uniquely determines $N$. 
\end{thm}

A lattice $N$ as in the theorem is called a {\it Niemeier lattice}. In all but one of the cases $N$ is spanned (over $\QQ$) by roots. The remaining case, i.e., the Niemeier lattice containing no roots, is called the {\it Leech lattice}, and we denote it by $\LL$.  The automorphism of the Leech lattice is the {\it Conway group} 
$$\Co_0:=O(\LL).$$ The center of $\Co_0$ is just $\mu_2=\{\pm \id\}$, and the quotient 
$$\Co_1:=\Co_0/Z(\Co_0)$$ 
is one of the largest sporadic simple groups. In fact, 
 $$|\Co_0|=2^{22}\cdot3^9\cdot 5^4\cdot 7^2\cdot11\cdot 13\cdot 23(\sim 8\cdot 10^{18}).$$
As we will see below, a group $G$ of symplectic automorphisms for $K3$ surfaces, hyper-K\"ahler manifolds of type $K3^{[n]}$, or cubic fourfolds can be realized as a subgroup of the Conway group $\Co_0$. Thus, only the prime factors $2$, $3$, $5$, $7$, $11$, $13$, and $23$ can occur in $\ord(G)$. For $K3$ surfaces, only the primes $p\le 7$ can occur, while for cubics all primes $p\le 11$ occur (compare Theorem \ref{theorem: lie fu}). In particular, the Fano variety $F(X)$ of a cubic fourfold $X$ admiting an order $11$ symplectic automorphism will give an example of an exotic automorphism (i.e.,  not induced from $K3$ surfaces) 
on a hyper-K\"ahler of $K3^{[2]}$ type (see \cite[\S4.5]{mongardi2013thesis}). 
\subsection{Leech Pairs}
As already mentioned, the study of symplectic automorphisms on $K3$s and $K3^{[n]}$'s leads to the following notion (first formalized in the thesis of Mongardi  \cite{mongardi2013thesis}):
\begin{defn}\label{def:leechpair}
A pair $(G,S)$ consisting of a finite group $G$ acting faithfully on an even lattice $S$ is called a {\it Leech pair}, if it satisfies the following conditions:
\begin{enumerate}[(i)]
\item $S$ is positive definite,
\item $S$ does not contain any 2-vector,
\item $G$ fixes no nontrivial vector in $S$,
\item the induced action of $G$ on the discriminant group $A_S$ is trivial.
\end{enumerate}
\end{defn}

The condition (iv) of the Definition \ref{def:leechpair} should be understood as saying that given a primitive embedding $S\hookrightarrow L$ into a unimodular lattice $L$, the action of $G$ on $S$ extends (acting trivially on $S^\perp_L$) to $L$. The condition (iii) complements this by saying that $S$ is the covariant lattice for the action of $G$ on $L$. Note then that the smallest unimodular lattice satisfying the first $2$ conditions of the definition above is the Leech lattice $\LL$. Obviously, any sublattice of the Leech lattice will also satisfy (i) and (ii) of the definition. Thus choosing a subgroup $G\subset \Co_0(=O(\LL))$, the associated covariant lattice $S_G(\LL)$ in $\LL$ will give an example of Leech pair $(G,S_G(\LL))$. The following proposition says the converse: under a mild condition on the rank of $S$ (satisfied in the geometric context relevant to this paper), the Leech pair $(S,G)$ is obtained as a covariant lattice in $\LL$. This argument seems to occur first in \cite[Appendix B]{gaberdiel2012symmetries} (see also \cite[Prop. 2.2]{huybrechtsaut}; related arguments go back to Scattone \cite{scattone} and Kond\=o \cite{kondo}).  For completeness, we sketch the proof.
\begin{prop}
\label{lemma: embedding into leech lattice}
For a Leech pair $(G,S)$ the following two statements are equivalent:
\begin{enumerate}[(i)]
\item $\rank(S)+l(q_S)\le 24$,
\item There exists a primitive embedding of $S$ into the Leech lattice $\LL$.
\end{enumerate}
Once these two condition are fulfilled, there is an action of $G$ on $\LL$ with $(G,S)\cong (G,S_G(\LL))$.
\end{prop}
\begin{proof}
Assume $(ii)$, and denote by $K$ the orthogonal complement of the given primitive embedding of $S$ into $\LL$. Then $l(q_S)=l(q_K)\le \rank(K)=24-\rank(S)$. Thus $(ii)$ implies $(i)$.

Now assume $(i)$. Since $l(q_S)\le 24-\rank(S)<\rank(\LL\oplus U)-\rank(S)$, by Nikulin's existence Theorem \ref{theorem: existence of lattice via discriminant form}, there exists a primitive embedding $S \hookrightarrow \LL\oplus U$. Denote by $N$ the orthogonal complement of $S$ in $\LL\oplus U$. Then $N$ has signature $(25-\rank(S),1)$. Thus $N_{\RR}$ intersects with the positive cone of $\LL\oplus U$. Since $S$ contains no 2-vector, $N_{\RR}$ intersects with one of the chambers of the positive cone of $\LL\oplus U$.

Let $w\in U$ be primitive and isotropic. The vector $w\in\LL\oplus U$ is called a Weyl vector\footnote{Up to conjugacy by $O(\I_{25,1})$, the choice of primitive isotropic vector $w$ in $\I_{25,1}\cong \LL\oplus U$ is equivalent to the choice of the isometry type of a Niemeier lattice $N\cong \langle w\rangle^\perp/\langle w\rangle$ (N.B. $N\oplus U\cong \I_{25,1}$). Thus, intrinsically a Weyl vector $w$ is a primitive isotropic vector in  $\I_{25,1}$ such that the associated Niemeier lattice $N$ is the Leech lattice.}. We call a vector $v\in \LL\oplus U$ with $(v, v)=2$ and $(v, w)=-1$ a Leech root. By \cite[Chap. 27]{conway1999spherepackings}, the automorphism group of $\LL\oplus U$ is generated by reflections with respect to Leech roots. Therefore, there exists a chamber $\C_0$ given by $\C_0=\{x\in (\LL\oplus U)\otimes{\RR}\big{|}(x,v)>0, \textrm{for any Leech root v}\}$. By adjusting the embedding $S\hookrightarrow \LL\oplus U$ via an automorphism of $\LL\oplus U$, we may assume that $N_{\RR}$ intersects with $\C_0$, hence $G$ leaves the chamber $\C_0$ stable. By \cite{borcherds1984leech}, $G$ fixes the Weyl vector $w$. Equivalently, $w\in N$. Then we have:
\begin{equation*}
S\hookrightarrow w^{\perp}\longrightarrow w^{\perp}/\langle w \rangle\cong \LL
\end{equation*}
which gives rise a primitive embedding of $S$ into the Leech lattice $\LL$. The group action of $G$ on $S$ extends to an action on $\LL$ with $(G,S)\cong (G,S_G(\LL))$.
\end{proof}

\begin{cor}
For a Leech pair $(G,S)$ satisfying the statements in Lemma  \ref{lemma: embedding into leech lattice}, there is an embedding $G\hookrightarrow \Co_0$, with image avoiding $-id$ unless $\rank(S)=24$.
\end{cor}

\subsection{H\"ohn--Mason classification of saturated Leech pairs}\label{subsec-HM} In view of the discussion above, to classify the Leech pairs relevant to the classification of automorphisms, one can proceed by considering subgroups $G\subset \Co_0$ and the associated covariant lattice $S_G(\LL)$. The only issue is that there might be several groups $G$ leading to the same covariant lattice. For the classification of automorphism groups, we are interested in the maximal cases (i.e., in $G= \Aut^s(X)$ and not subgroups $G'\subset G$ that happens to have the same invariant/covariant lattice). The following two definitions formalize this idea. 
\begin{defn}
A Leech pair (G, S) is called {\it saturated}, if $G$ is the maximal group acting faithfully on $S$ and trivially on the discriminant group $A_S$.
\end{defn}

Let $G$ be a finite group acting on the Leech lattice $\LL$.  One can consider the (point-wise) stabilizer $G'$ of $\LL^G$. Obviously, $G\subseteq G'$, $\LL^G=\LL^{G'}$, and $G'$ is the largest group stabilizing $\LL^G$. The induced action of $G^{\prime}$ on $A_{S_G(\LL)}\cong A_{\LL^G}$ is trivial. Conversely, every automorphism of $S_G(\LL)$ which trivializes $A_{S_G(\LL)}$ can be extended to an automorphism of $\LL$ which stabilizes $\LL^G$. Thus $G^{\prime}$ is equal to the automorphism group of $S_G(\LL)$ trivializing the discriminant. The Leech pair $(G, S_G(\LL))$ is saturated if and only if $G=G^{\prime}$. 

\begin{defn}
Let  $(G_1, S_1)$ and $(G_2, S_2)$ be two Leech pairs. We say $(G_1, S_1) \le (G_2, S_2)$  if $G_1$ is a subgroup of $G_2$ and $S_1=S_2^{G_1}$. We call $(G_1, S_1)$ a sub-pair of $(G_2, S_2)$. Two sub-pairs $(G_1, S_1), (G_2, S_2)$ of a Leech pair $(G, S)$ are conjugate if there exists $g\in G$ such that $g G_1 g^{-1}=G_2$ and $g S_1=S_2$.
\end{defn}

We denote by $\mathscr{A}$ the set of conjugacy classes of sub-pairs of $(\Co_0, \LL)$. There is a natural poset structure on $\mathscr{A}$. Denote by $\mathscr{A}_{sat}$ the sub-poset of $\mathscr{A}$ consisting of saturated Leech pairs. {\it A fixed-point sublattice} of $\LL$ is the invariant sublattice $\LL^G$ for some $G\subset \Co_0$. It is clear that associating with 
$(G, S)\in \sA$ the fixed-point sublattice $\LL^G$ gives rise 
 to a one-to-one correspondence between $\mathscr{A}_{sat}$ and the set of ($\Co_0$-)orbits in the set of fixed-point sublattices of the Leech lattice $\LL$. The fixed-point sublattices of $\LL$ were classified by H\"ohn and Mason \cite{HM1}. This classification will play a key role for us. For further reference, we mention:
 
\begin{thm}[H\"ohn-Mason]
\label{theorem: 290}
Under the action of $\Co_0$, there are exactly $290$ orbits on the set of fixed-point sublattices of $\LL$. In another word, $|\sA_{sat}|=290$.
\end{thm}

\begin{rmk}  
Harada and Lang \cite{harada1990leech} classified all fixed-point sublattices $K$ which are induced by actions of cyclic groups $G\cong \ZZ/n$ on the Leech lattice. The information contained in \cite{harada1990leech} is sometimes richer and more handy than that in \cite{HM1}.
\end{rmk}

\section{The case of cubic fourfolds}
\label{section: symplectic automorphism of cubic fourfold}
In this section, we are classifying the symplectic automorphism groups of smooth cubic fourfolds. First,  following the standard argument for $K3$ surfaces and hyper-K\"ahler manifolds, we establish that a group $G$ acting symplectically on a cubic $X$, determines a Leech pair $(G,S=S_G(X))$, which further can be embedded into the Leech lattice $\LL$ (Corollary \ref{corollary: embedding into Leech}). Since $S$ arises from a cubic fourfold $X$, it is clear that $S$ embeds into the primitive lattice $\Lambda_0$. By Theorem \ref{theorem: thmperiods} (we use the surjectivity part), this is essentially  also a sufficient condition. We state this, in terms of the Borcherds polarization (see \S\ref{subsec:borcherds}) as an iff criterion in Theorem \ref{theorem: criterion lattice from cubic fourfold}. Using this criterion, the actual classification (\S\ref{subsec-completeproof}) is accomplished by using the H\"ohn--Masson \cite{HM1} (see also \cite{harada1990leech}) classification of the fixed-point sublattices in the Leech lattice, and Fu's classification (\cite{fu2016classification}) of automorphism groups of prime-power orders. The uniqueness of embeddings in the maximal cases (Theorem \ref{theorem: maximal uniqueness}) is discussed in \S\ref{subsec-uniqueness}. 

\subsection{Leech pairs associated to symplectic automorphisms on cubic fourfolds and $K3$ surfaces} A finite group of symplectic automorphisms on a $K3$ surface, on a hyper-K\"ahler manifold of $K3^{[n]}$ type, or on a cubic fourfold leads to a Leech pair. The argument essentially goes back to Nikulin \cite{nikulin}, and was refined recently in the context of groups of symplectic automorphisms for hyper-K\"ahler manifolds (see esp. \cite{huybrechtsaut} and \cite{mongardi2013thesis}). We review the situation for the cases relevant to us: cubic fourfolds and polarized $K3$ surfaces. 

\begin{notation} Let $X$ be a smooth cubic fourfold, and $G\subset \Aut^s(X)$.  We denote by $S_G(X)$ the covariant lattice for the induced action of $G$ on $H^4(X, \ZZ)$. Similarly, if $Y$ is a smooth algebraic $K3$ surface, and $G\subset \Aut^s(Y)$ a finite group, we denote by $S_G(Y)$ the covariant lattice for the induced action of $G$ on $H^2(Y, \ZZ)(-1)$.
\end{notation}

\begin{lem}
\label{lemma: leech type pairs from cubic fourfold and K3 surface} 
Let $X$ be either a smooth cubic fourfold or an algebraic $K3$ surface with an action of a finite group $G\subset \Aut^s(X)$.  Then $(G,S_G(X))$ is a Leech pair.
\end{lem}
\begin{proof} The assumption of symplectic automorphism implies that $S_G(X)\subset H^{2,2}(X)\cap H^4(X,\ZZ)_{prim}$. By Hodge index Theorem, $S_G(X)$ is positive definite, and by Theorem \ref{theorem: thmperiods}, $S_G(X)$ contains no short roots (i.e., the period point avoids $\calC_6$). Since $G$ acts trivially on the invariant cohomology $H^4(X,\ZZ)^G$ and $S_G(X)=(H^4(X,\ZZ)^G)^\perp$, it follows that $G$ acts trivially on $A_{S_G(X)}$. Finally, since $\Aut(X)$ acts faithfully on $H^4(X)$, it is clear that $G$ acts faithfully on $S_G(X)$.  We conclude that $(G,S_G(X))$ is a Leech pair (cf. Def. \ref{def:leechpair}).

The argument for $K3$ surfaces is similar (and due to Nikulin), except for invoking Riemann--Roch to prove that there is no norm $2$ vector (corresponding, via our scaling, to a $-2$ class) in $S_G(X)$.
\end{proof}

\begin{cor}
\label{corollary: embedding into Leech}
Let $X$ be either a smooth cubic fourfold or an algebraic $K3$ surface with a faithful action of a finite group $G\subset \Aut^s(X)$. There exists a primitive embedding of $S_G(X)$ into $\LL$, and hence an embedding of $G$ into $\Co_0$ with image avoiding $-id$. 
\end{cor}
\begin{proof}
By Lemma  \ref{lemma: leech type pairs from cubic fourfold and K3 surface}, $(G,S_G(X))$ is a Leech pair. Since $S_G(X)$ has a primitive embedding into a unimodular lattice of rank $23$  (or $22$) for cubic fourfolds (or $K3$ surfaces respectively), the rank condition of Proposition \ref{lemma: embedding into leech lattice} is satisfied; the claim follows.
\end{proof}

Let us now discuss the role of the polarization. If $X$ is a cubic fourfold, any automorphism $f$ is induced from a linear automorphism of the ambient projective space, and thus $\varphi=f^*$ preserves the class $\eta\in H^4(X,\ZZ)$ (recall $\eta$ is the square of a hyperplane class). It follows that there is a primitive embedding
\begin{equation}\label{primitive_emb}
S_G(X)\hookrightarrow \Lambda_0, 
\end{equation}
where $\Lambda_0$ is the primitive cohomology (recall $\Lambda_0\cong A_2\oplus (E_8)^2\oplus U^2$). 

For $K3$ surfaces $Y$, the situation is similar, but there is a subtle difference. Namely, under the assumption that $Y$ is algebraic (i.e., $\NS(Y)$ contains an ample class $h$), and $G$ is finite, {\it any automorphism $\varphi\in G$ will preserve \underline{some} ample class $h'$} (e.g., obtained by ``averaging'' $h$). This is the set-up of the classical results of Nikulin and Mukai. However, when taking about \underline{polarized} $K3$ surfaces, we will fix an ample class $h$ on $Y$ and insist that the automorphism $f$ preserves $h$ (i.e., $f^*h=h$ in cohomology). With this assumption, we have again a primitive embedding
$$S_G(X)\hookrightarrow \Lambda_d$$
where $d=h^2\in 2\ZZ_+$, and $\Lambda_d=(\langle h\rangle^\perp_{H^2(Y,\ZZ)})(-1)$ is the primitive cohomology (we twist the form by $-1$ to get consistency with the cubic fourfold case). 

\begin{rmk}
We are not aware of a systematic study of the symplectic automorphisms in the polarized case for any degree (in section \ref{section: K3} below, we will partially discuss the degree $2$ and $6$ cases as they are tightly connected to the cubic fourfold case). One situation where the polarized case was studied is the symplectic involutions. We recall that Nikulin proved that there is a single class of symplectic involutions for algebraic $K3$ surfaces (with notations as above, $S_G(X)\cong E_8(2)$). The polarized symplectic involutions were classified by van Geemen and Sarti \cite{vGS}; a richer picture emerges (as one needs to keep track of the embedding of $E_8(2)$ into $\Lambda_d$, versus the unimodular $K3$ lattice). 
\end{rmk}
\subsubsection{A criterion for Leech pairs to arise from symplectic automorphisms} So far we have discussed how a finite group of symplectic automorphisms $G\subset \Aut^s(X)$ leads to a Leech pair $(G,S_G(X))$, which in turn can be classified by H\"ohn--Mason \cite{HM1} results. Now we are interested in the converse, given a Leech pair $(G,S)$, {\it when does it come from a symplectic automorphism group $G$ acting on $X$?} By Global Torelli Theorem (and surjectivity of the period map), this becomes a question about embeddings of lattices. For instance, note that \eqref{primitive_emb} is a necessary condition if $X$ is a cubic fourfold. In fact, by Theorem \ref{theorem: thmperiods} (and Prop. \ref{proposition: aut of cubic and fano}), \eqref{primitive_emb} is essentially also sufficient, but some care is needed as $S$ needs to avoid both short roots (automatic since $(G,S)$ is a Leech pair) and long roots. To deal with both cases uniformly, it is better to view a smooth cubic fourfold $X$ as being $E_6$ Borcherds polarized (see \S\ref{subsec:borcherds}). Based on these considerations, we obtain the following key result which allows us to go back and forth between geometry (automorphisms of $X$) and arithmetic (fixed-point sublattices of the Leech lattice $\LL$).

\begin{thm}[Criterion for Leech pairs associated with cubic fourfolds]
\label{theorem: criterion lattice from cubic fourfold}
Let $(G,S)$ be a Leech pair.  The following are equivalent:
\begin{enumerate}[(i)]
\item There exists a smooth cubic fourfold $X$ with a faithful and symplectic action of $G$ such that $(G,S)\cong (G,S_G(X))$,
\item There exists a faithful action of $G$ on the Leech lattice $\LL$ with $(G, S)\cong(G, S_G(\LL))$ and $K=\LL^G$, such that there exists a primitive embedding of $E_6$ into $K\oplus U^2$,
\item There exists an embedding of $S\oplus E_6$ into the Borcherds lattice $\BB$, such that the image of $S$ is primitive.
\end{enumerate}
\end{thm}

\begin{proof}
$(i)\Longrightarrow (ii)$: From Corollary \ref{corollary: embedding into Leech}, there exists a primitive embedding $S\hookrightarrow \LL$ with an extension of the $G$-action on $\LL$ such that $\LL^G$ is the orthogonal complement of $S$ in $\LL$. We have now two ways to embed $S$ into $\BB$, explicitly:
\begin{equation*}
S\hookrightarrow \Lambda_0\hookrightarrow \Lambda_0\oplus E_6 \subset \BB
\end{equation*}
and
\begin{equation*}
S\hookrightarrow \LL\hookrightarrow \LL\oplus U^2 \cong \BB.
\end{equation*}
Clearly, both embeddings are primitive (e.g.,  $\Lambda_0\subset \BB$ is primitive by Proposition \ref{lem:borcherdse6}, and $S$ is primitive in $\Lambda_0$ by \eqref{primitive_emb}). By Nikulin's results (see Theorem \ref{theorem: uniqueness of embedding even case}), we know that there is a single conjugacy class of primitive embeddings $S\hookrightarrow \BB$. 
Therefore, we can choose the isomorphism $\LL\oplus U^2\cong \BB$, such that the following diagram commutes:
\begin{equation*}
\begin{tikzcd}
S \arrow[hook]{r}\arrow[hook]{d} &\LL \arrow[hook]{r} &  \LL\oplus U^2\arrow{d}{\cong}\\
\Lambda_0 \arrow[hook]{r} & \Lambda_0\oplus E_6 \arrow[hook]{r} & \BB  \\
\end{tikzcd}
\end{equation*}
We have $K=S^\perp_\LL$, giving $S^\perp_\BB\cong K\oplus U^2$. On the other hand, $E_6\cong (\Lambda_0)^\perp_\BB$, thus $E_6\subset S^\perp_\BB\cong K\oplus U^2$. Since $E_6$ does not admit any overlattice, $E_6$ embeds primitively into $K\oplus U^2$.

$(ii)\Longrightarrow (iii)$: There is the embedding:
\begin{equation*}
S\oplus E_6\hookrightarrow S\oplus K\oplus U^2\subset\LL\oplus U^2\cong \BB
\end{equation*}
Notice that $S$ has primitive image in $\LL$, hence also has primitive image in $\BB$.

$(iii)\Longrightarrow (i)$: The action of $G$ on $S$ induces trivial action on $(A_S, q_S)$, hence extend to be an action on $\BB$ such that its restriction to the orthogonal complement of $S$ trivial. Since $S\subset \BB$ is primitive (by assumption), we get $S=S_G(\BB)$ (recall $S_G(\BB)=(\BB^G)^\perp=(S^\perp)^\perp$). On the other hand, we note that  $G$ acts trivially on $E_6\subset \BB$ (since by construction $E_6\subset S^\perp_{\BB}$). We view $\Lambda_0$ as the orthogonal complement of $E_6$ in $\BB$ (cf. Prop. \ref{lem:borcherdse6}). Via this identification, the $G$ action on $\BB$ induces a $G$ action on $\Lambda_0$. By construction $S\hookrightarrow \Lambda_0$ (primitive, as $S$ is primitive in $\BB$), and clearly $(G,S)\cong (G,S_G(\Lambda_0))$.  We can choose a Hodge structure $H$ on $\Lambda_0$ of type $(0,1,20,1,0)$ (i.e., $H$ is a decomposition of $\Lambda_{0,\CC}$ with the obvious properties) such that $H^{2,2}\cap \Lambda_0=S$ (i.e., $S$ is the algebraic lattice). Assuming that $S$ contains no short or long roots, the Global Torelli Theorem (Theorem \ref{theorem: thmperiods}) says that there exists a smooth cubic fourfold with $H^4(X,\ZZ)_{prim}\cong H$ (as Hodge structures). Finally, by Proposition  \ref{proposition: strong global torelli}, we conclude that $X$ has a faithful and symplectic action of $G$ such that $(G,S_G(X))\cong (G,S)$.

It remains to prove that $S\subset \Lambda_0$ contains no short or long roots of $\Lambda_0$ (see Definition \ref{definition: short/long roots}). By assumption $S$ is a sublattice of the Leech lattice $\LL$, so it contains no short roots (i.e., norm $2$ vectors). Assume now $S$ contains a long root $\delta$, i.e., $(\delta, \delta)=6$ and $\mathrm{div}_{\Lambda_0}(\delta)=3$. Since $\BB$ is obtained by gluing $E_6$ and $\Lambda_0$, we conclude that $\delta$ and $E_6$ span a $E_7$ lattice in $\BB$.  More precisely, there exists $\epsilon\in E_6$ (with $\epsilon^2=12$ and $\mathrm{div}_{E_6}(\epsilon)=3$) such that  $(\delta+\epsilon)/3\in \BB$. Since $G$ acts on $S$ without fixed nonzero vector, there exists $g\in G$ such that $g\delta\ne \delta$. We distinguish two cases, either $g\delta=-\delta$ or not. Assume first  $g\delta=-\delta$; then  $g((\delta+\epsilon)/3)=(-\delta+\epsilon)/3\in \BB$. We conclude $v=2\delta/3\in \BB$, but this is a contradiction due to the fact that $(\delta, \delta)=6$ ($v$ will not have integral norm). Thus, we can assume that 
 $\delta^{\prime}=g\delta$ is a long root non-proportional to $\delta$.
Consider the lattice $M=\mathrm{Sat}_\BB(\langle \delta,\delta^{\prime}, E_6\rangle)\subset \mathrm{Sat}_\BB(S\oplus E_6)$. Then $M$ is a positive definite rank $8$ lattice containing two sublattices $\mathrm{Sat}_\BB(\langle \delta,E_6\rangle)$ and $\mathrm{Sat}_\BB(\langle \delta^{\prime},E_6\rangle)$ of type $E_7$. Clearly, $M\cong E_8$ (first, the root sublattice of $M$ is of type $E_8$ as it is strictly larger than $E_7$, then $E_8 \subset M$ forces equality for reasons of rank and determinant). It is well known that $E_6$ admits a unique embedding in $E_8$ with orthogonal complement $A_2$. We get $A_2\subset S=(E_6)^\perp_{\mathrm{Sat}_\BB(S\oplus E_6)}$ (using $S$ primitive in $\BB$).  In particular, $S$ contains some short roots, contradicting the fact that $(G,S)$ is a Leech pair.
\end{proof}

\subsubsection{Moduli of cubics associated with a Leech pair $(G,S)$}
We denote by $\mathscr{A}_{cub}$ the sub-poset of $\mathscr{A}$ consisting of Leech pairs isomorphic  to $(G, S_G(X))$ for some smooth cubic fourfold $X$ with $G=\Aut^s(X)$. It is clear that such a Leech pair $(G, S_G(X))$ is saturated. Therefore we have $\mathscr{A}_{cub}\subset \mathscr{A}_{sat}$. Our purpose is to determine the poset $\mathscr{A}_{cub}$. We now discuss the geometric loci (``moduli'') associated with the elements of this poset. By studying the minimal and maximal loci, in \S\ref{subsec-max} and \S\ref{subsec-min} respectively, we will be able to complete the proof of our main Theorem \ref{theorem: main}.

Let $(G,S)$ be a Leech pair with $G= \Aut^s(X)$. As already discussed, it follows that $S\subset H^{2,2}(X)\cap H^4(X,\ZZ)_{prim}$ (i.e., $S$ is a lattice of algebraic cycles on $X$) and in fact the equality holds generically. Similarly to the well-known situation for $K3$ surfaces (see Remark \ref{m-pol-cubic}), one can consider the moduli space of $S$-polarized cubic fourfolds (i.e., cubics with $S\hookrightarrow H^4(X,\ZZ)_{prim}\cong \Lambda_0$ primitive), or even $(G,S)$-polarized (since $G$ acts trivially on $A_S$, the $G$-action extends to $\Lambda_0$, and thus there is essentially no difference). We obtain a moduli space $\calM_{(G,S)}$ which is a locally symmetric variety of Type $\IV$ (of the same type as the moduli of cubic fourfolds). Some care is needed here. First, $\calM_{(G,S)}$ can have several irreducible components (corresponding to different primitive embeddings of $S$ into $\Lambda_0$). Then, since we view $\calM_{(G,S)}$ as a closed subvariety of $\calM$ (the moduli of cubic fourfolds), a normalization is needed in order to view it as a locally symmetric variety. Finally, one needs to exclude the restrictions of the hyperplane arrangements $\calH_2$ and $\calH_6$ (see Theorem \ref{theorem: thmperiods}) to the locus of $S$-polarized cubics (as discussed $S$ does not contain short or long roots, thus this locus is not contained in either $\calH_2$ or $\calH_6$; on the other hand, the restrictions of $\calH_2$ and $\calH_6$ can lead to multiple irreducible arrangements). We refer to Mongardi \cite{mongardi2013thesis} and \cite{yu2018moduli} for further details. To summarize the above discussion, we have:

\begin{thm}[{\cite{yu2018moduli}}] 
\label{theorem: Yu-Zheng}
Let $(G,S)$ be a Leech pair such that $\calM_{(G,S)}\neq \emptyset$. Let $\calF$ be the normalization of an irreducible component of $\calM_{(G,S)}$.
Then period map for cubic fourfolds with symplectic action by $G$ gives rise to a natural isomorphism 
\begin{equation*}
\calF\cong(\DD\setminus \calH)/\Gamma',
\end{equation*}
where $\DD$ is a Type $\IV$ domain with a faithful action of an arithmetic group $\Gamma'$, and $\calH$ is a $\Gamma'$-invariant hyperplane arrangement in $\DD$.  Moreover, $\dim(\calF)=\dim (\DD)=20-\rank(S)$.
\end{thm}
\begin{rmk}
The definition of $\calM_{(G,S)}$ makes sense for all Leech pairs $(G,S)$, but in fact it depends only on the saturation, i.e., $\calM_{(G,S)}=\calM_{(G',S)}$, where $G'=\Aut^s(X)$ for $X$ a general cubic in $\calM_{(G,S)}$. In Theorem \ref{theorem: main}, our classification is about saturated pairs, but in the arguments below it is convenient not to require $(G,S)$ to be saturated. 
\end{rmk}
It is clear that the moduli spaces $\calM_{(G,S)}$ have a natural poset structure that matches with the poset structure on $\mathscr{A}_{cub}$.  Theorem \ref{theorem: main} is organized by the dimensions of $\calM_{(G,S)}(\neq \emptyset)$ (or equivalently $\rank(S)$).

\subsection{The maximal Leech pairs for cubic fourfolds}\label{subsec-max} We now note that the Leech pairs arising from automorphisms of cubic fourfolds satisfy an easy necessary condition (in terms of the rank of covariant lattice and the rank of the discriminant group). In order to easier relate to the H\"ohn--Mason classification \cite{HM1}, we state the condition in terms of the fixed-point sublattice $K$ in the Leech lattice $\LL$.

\begin{cond}
\label{condition: 1}
Let $(G,S)$ be a Leech sub-pair of $(\Co_0,\LL)$, and $K=S^\perp_{\LL}$. 
We require $K$ to satisfy the following conditions:
\begin{enumerate}[(i)]
\item $\rank(K)\ge 4$ (or equivalently $\rank(S)\le 20$);
\item for every prime number $p\ne 3$, $\alpha_p(K)(\coloneqq \rank(K)-l_p(A_K))\ge 2$, and $\alpha_3(K)\ge 1$ (in particular $\alpha(K)=\rank(K)-l(A_K)\ge 1$).
\end{enumerate}
\end{cond}

\begin{prop}
\label{proposition: rough condition}
The equivalent conditions in Theorem  \ref{theorem: criterion lattice from cubic fourfold} imply Condition  \ref{condition: 1}.
\end{prop}
\begin{proof} Since $S$ embeds into $\Lambda_0$ which has signature $(20,2)$, the rank condition is clear. Assume now that $(G,S)$ is a  Leech pair with a primitive embedding of $S$ into $\LL$, and $K$ is the orthogonal complement of $S$ in $\LL$. By Theorem  \ref{theorem: criterion lattice from cubic fourfold}, there exists a primitive embedding of $E_6$ into $K\oplus U^2$. Denote by $M$ the orthogonal complement of $E_6$ in $K\oplus U^2$. We have a saturation $E_6\oplus M\hookrightarrow K\oplus U^2$. By Nikulin's glueing theory, there exists an isotropic subspace $H$ of $A_{E_6}\oplus A_M$, such that $A_K\cong H^{\perp}/H$. Since $M$ is primitive in $K\oplus U^2$, there is no nontrivial element in $H\cap A_M$. Therefore, we have either $H=0$ or $H=\{(x,f(x))\big{|}x\in A_{E_6}\}$, where $f\colon A_{E_6}\longrightarrow A_M$ is an isometry onto $f(A_{E_6})$ equipped with $-q_M$.

Assume first that the glueing group $H$ is trivial, then $A_K=A_{E_6}\oplus A_M$. Since $A_{E_6}\cong \ZZ/3$, we conclude $l_p(A_M)=l_p(A_K)$ for $p\ne 3$ and $l_3(A_M)=l_3(A_K)-1$. Otherwise, we have  $H=\{(x,f(x))\big{|}x\in A_{E_6}, f(x)\in A_M, q_{E_6}(x)=-q_M(f(x))\}\cong \ZZ/3$. We have an isometry 
\begin{equation*}
A_K\cong H^{\perp}/H\cong A_M/f(A_{E_6}).
\end{equation*}
Thus, $l_p(A_M)=l_p(A_K)$ for $p\ne 3$ and $l_3(A_M)=l_3(A_K)+1$.

In any case, we get
\begin{equation*}
l_p(A_K)=l_p(A_M)\le \rank(M)=\rank(K)-2
\end{equation*}
for $p\ne 3$, and
\begin{equation*}
l_3(A_K)\le l_3(A_M)+1\le \rank(M)+1=\rank(K)-1
\end{equation*}
hence Condition  \ref{condition: 1}.
\end{proof}
\begin{rmk}\label{rmk-invo}
To understand the restriction imposed by Condition  \ref{condition: 1} on Leech pairs, let us consider the case $G=\ZZ/2$ (i.e., symplectic involutions). According to \cite{harada1990leech} (also \cite{HM1}), there are three nontrivial  conjugacy classes of involutions in $\Co_0=O(\LL)$. The fixed-point sublattices $K$ in the three cases are $E_8(2)$, $D_{12}^+(2)$, and $BW_{16}$ (the Barnes--Wall lattice), while the covariant lattices $S_G(\LL)=K^\perp_{\LL}$ are $BW_{16}$, $D_{12}^+(2)$, and $E_8(2)$ respectively. For $E_8(2)$ and $D_{12}^+(2)$, it holds $\rank(K)=l(K)$ (this holds true whenever $K=K'(n)$ for some integral lattice $K'$, $n\in \ZZ_{>1}$), while $BW_{16}$ obviously satisfies Condition  \ref{condition: 1}. We conclude that the only possible Leech pair arising from symplectic involutions on cubic fourfolds is $(\ZZ/2, E_8(2))$. To conclude that there is a unique class of symplectic involutions, we would need to prove that there exists a unique primitive embedding of $E_8(2)$ in $\Lambda_0$. In this particular case, a direct geometric argument (via a diagonalization of the involution) is easier. This concludes item (1) of Theorem \ref{theorem: main}.
\end{rmk}

In \S\ref{subsec-HM}, we have defined a natural poset $\mathscr{A}$ on the set of Leech pairs in $(\Co_0,\LL)$. We are now interested in identifying the maximal Leech pairs $(G,S)$ arising from cubic fourfolds. As noted above, these pairs satisfy Condition \ref{condition: 1}. Focusing on the maximal rank cases, by inspecting  \cite{HM1}, we note that there are $15$ Leech pairs $(G,S)\in A_{sat}$ with $\rank(S)=20$ (or equivalently $\rank(K)=4$) and satisfying Condition  \ref{condition: 1}. In fact, these cases precisely coincide with those of \cite[Table 9]{HM2}. For reader's convenience, we list them (sometimes corrected\footnote{There are some typos in the listing of the discriminant forms in \cite{HM1}. For example, the discriminant form corresponds to case of $M_{10}$ is listed as $2_5^{+1}4_1^{+1}3^{-1}5^{+1}$ in \cite{HM1}, but this is not allowed in the Conway--Sloane \cite{conway1999spherepackings} notation.}) in Table \eqref{table: MSS aut of HK} below.

\begin{table}[ht] \caption{Maximal rank Leech pairs satisfying Condition  \ref{condition: 1}} 
\label{table: MSS aut of HK}
\renewcommand{\arraystretch}{1.2}\centering
\begin{tabular}{c c c c}
\hline\hline
Number & Order & Group & Discriminant form $q_K$  \\ [0.5ex]
\hline
1 & 29,160 & $3^4:A_6$ & $3^{+2}9^{+1}$  \\
2 & 20,160 & $L_3(4)$ & $2_{\II}^{-2}3^{-1}7^{-1}$  \\
3 & 5,760 & $2^4:A_6$  & $4_5^{-1}8_1^{+1}3^{+1}$  \\
4 & 2,520 & $A_7$ & $3^{+1}5^{+1}7^{+1}$ \\
5 & 1,944 & $3^{1+4}:2.2^2$ & $2_2^{+2}3^{+3}$ \\
6 & 1,920 & $2^4:S_5$ & $4_3^{-1}8_1^{+1}5^{-1}$ \\
7 & 1,344 & $2^3:L_2(7)$ & $4_2^{+2}7^{+1}$ \\
8 & 1,152 & $Q:(3^2:2)$ & $8_6^{-2}3^{-1}$\\
9 & 720 & $S_6$ & $2_{\II}^{-2}3^{+2}5^{+1}$\\
10 & 720 & $M_{10}$ & $2_3^{-1}4_7^{+1}3^{-1}5^{+1}$\\
11 & 660 & $L_2(11)$ & $11^{+2}$\\
12 & 576 & $2^4:(S_3\times S_3)$ & $4_7^{+1}8_1^{+1}3^{+2}$\\
13 & 360 & $A_{3,5}$ & $3^{-2}5^{-2}$\\
14 & 336 & $2\times L_2(7)$ & $2_{\II}^{+2}7^{+2}$\\
15 & 144 & $3^2:\QD_{16}$ & $2_1^{+1}4_1^{+1}3^{-1}9^{-1}$\\
[1ex]
\hline
\end{tabular}
\label{table:nonlin}
\end{table}

\begin{rmk}
In Table \eqref{table: MSS aut of HK}, the items in the last column represent for discriminant forms of the invariant sublattices of the actions of $G$ on the Leech lattice. See \cite[Page 379-380]{conway1999spherepackings} and also our Appendix \ref{section: collection of results in lattice theory} for an explanation of the notations of discriminant forms. 
\end{rmk}

\begin{rmk}
The group $Q$ appearing in item $8$ is a group of order $128$, see \cite[Theorem 5.1, Case 5(b)]{HM2}. We expect that the semi-direct product $3^2:\QD_{16}$ appearing in item $15$ is in fact isomorphic to $M_{2,9}$ (see \S\ref{smathieu}). 
\end{rmk}

It turns out that the $15$ groups listed in Table \eqref{table: MSS aut of HK} occur as maximal groups of symplectic automorphisms for some hyper-K\"ahler manifold of $K3^{[2]}$ type (algebraic, but not polarized). Specifically, it holds:

\begin{thm}[{H\"ohn--Mason \cite[Theorem  8.7]{HM2}}]
\label{theorem: HM group acts on K3 square}
A finite group acts symplectically on a hyper-K\"ahler manifold of type $K3^{[2]}$ if and only if it is a subgroup of a group in Table \eqref{table: MSS aut of HK}.
\end{thm}

We are interested in the maximal rank cases that can occur for cubic fourfolds, or equivalently the saturated Leech pairs $(G,S)$ for which $\calM_{(G,S)}\neq \emptyset$ and $\dim \calM_{(G,S)}=0$. H\"ohn and Mason \cite[Table 11]{HM2} have identified six cases that do occur for cubic fourfolds, and in fact they gave explicit equations of cubic fourfolds realizing these groups of automorphisms. Using our Criterion  \ref{theorem: criterion lattice from cubic fourfold}, we prove the converse: these six cases are all the maximal rank possibilities for cubic fourfolds. Note however (see \S\ref{subsec-uniqueness} below) that in two of the cases, there are two distinct embeddings of $S$ into $\Lambda_0$, leading to two more isolated cubic fourfolds with large symmetry in addition to the six cubics found by H\"ohn and Mason. 
\begin{thm}
\label{proposition: maximal case}
Let $X$ be a smooth cubic fourfold, and $G= \Aut^s(X)$. Assume that $\rank (S_G(X))=20$, then the Leech pair $(G,S_G(X))$ corresponds to one of the entries $1$, $4$, $5$, $10$, $11$ and $13$ in Table \eqref{table: MSS aut of HK}.
\end{thm}
\begin{proof}
By Theorem  \ref{theorem: criterion lattice from cubic fourfold}, we need to determine for which $(G,S)$ among the $15$ candidates, there exists an embedding of $S\oplus E_6$ into the Borcherds lattice $\BB$ with the image of $S$ primitive. There are two possibilities for such an embedding $S\oplus E_6\subset \BB$. Either $S\oplus E_6\subset \BB$ is primitive or not. If $S\oplus E_6\subset \BB$ is not primitive, there exists a coindex 3 saturation $\widetilde{S}$ of $S\oplus E_6$, in which $S$ is primitive. Then 
$\widetilde{S}$ embeds primitively into $\BB$. Since $S\oplus E_6$ (or $\widetilde{S}$) has rank $26$, and $\BB$ is the unique even unimodular lattice of signature $(26,2)$, by Nikulin's theory, we conclude that $S\oplus E_6$ (or $\widetilde{S}$ respectively) embeds primitively into $\BB$ iff there exists a negative definite rank $2$ even lattice $T$ with discriminant form $q_T=-q_{S\oplus E_6}$ (or $q_T=-q_{\widetilde S}$ respectively). By Theorem  \ref{theorem: existence of lattice via discriminant form}, such a lattice $T$ exists iff four conditions are satisfied. The first condition on the signature is automatically satisfied here. The remaining conditions are on the discriminant form $q_T$ (that is determined by $S\oplus E_6$ or the index $3$ overlatice $\widetilde S$ of $S\oplus E_6$). We do a case by case analysis of the $15$ possibilities from Table \eqref{table: MSS aut of HK}. The computations are standard manipulations with finite groups, and finite quadratic forms, we list only the essential details. (For a prime $p$, $\ZZ_p$ denotes the ring of $p$-adic integers.)

\begin{enumerate}[(1)]
\item The discriminant form of $S\oplus E_6$ is $3^{+2}9^{-1}\oplus 3^{+1}$. There is a nontrivial saturation $\widetilde{S}$ of $S\oplus E_6$ with discriminant form $3^{-1}9^{-1}$. There exists a negative rank $2$ even lattice $T$ with discriminant form $3^{+1}9^{+1}$. Thus there exists a primitive embedding of $\widetilde{S}$ into $\BB$ with orthogonal complement $T$.
\item The discriminant form of $S\oplus E_6$ is $2_{\II}^{-2}3^{+2}7^{+1}$. There is no nontrivial saturation of $S\oplus E_6$. Since $3^2\times 7$ is a square in $\ZZ_2$, there does not exists a negative rank $2$ even lattice with discriminant form $2_{\II}^{-2}3^{+2}7^{-1}$. Thus there does not exist embedding of $S\oplus E_6$ into $\BB$.
\item The discriminant form of $S\oplus E_6$ is $4_3^{+1}8_7^{-1}3^{-1}\oplus 3^{+1}$. There is a nontrivial saturation $\widetilde{S}$ of $S\oplus E_6$ with discriminant form $4_3^{+1}8_7^{-1}$. By the third condition in Nikulin's criterion, there does not exist a negative rank $2$ even lattice with discriminant form $4_5^{-1}8_1^{+1}$ nor $4_5^{-1}8_1^{+1}3^{-1}\oplus 3^{+1}$. Thus there does not exist embedding of $S\oplus E_6$ into $\BB$.
\item The discriminant form of $S\oplus E_6$ is $3^{-1}5^{+1}7^{-1}\oplus 3^{+1}$. There is a nontrivial saturation $\widetilde{S}$ of $S\oplus E_6$ with discriminant form $5^{+1}7^{-1}$. There exists a negative rank $2$ even lattice $T$ with discriminant form $5^{+1}7^{+1}$. Thus there exists a primitive embedding of $\widetilde{S}$ into $\BB$ with orthogonal complement $T$. Since $5\times 7=35$ is not a square in $\ZZ_3$, there exists a negative rank $2$ even lattice $T^{\prime}$ with discriminant form $3^{-2}5^{+1}7^{+1}$. Thus there exists a primitive embedding of $S\oplus E_6$ into $\BB$ with orthogonal complement $T^{\prime}$.
\item The discriminant form of $S\oplus E_6$ is $2_6^{+2}3^{-3}\oplus 3^{+1}$. There is a nontrivial saturation $\widetilde{S}$ of $S\oplus E_6$ with discriminant form $2_2^{+2}3^{+2}$. Since $2\times 2=4$ is apparently a square in $\ZZ_3$, there exists a negative rank $2$ even lattice $T$ with discriminant form $2_6^{+2}3^{+2}$. Thus there exists a primitive embedding of $\widetilde{S}$ into $\BB$ with orthogonal complement $T$.
\item The discriminant form of $S\oplus E_6$ is $4_5^{+1}8_7^{-1}5^{-1}\oplus 3^{+1}$, and there is no nontrivial saturation of $S\oplus E_6$. Since $5\times 3=15\equiv -1$ (mod 8), there is no negative rank $2$ even lattice with discriminant form $4_3^{-1}8_1^{+1}3^{-1}5^{-1}$. Thus there does not exist embedding of $S\oplus E_6$ into $\BB$.
\item The discriminant form of $S\oplus E_6$ is $4_6^{+2}7^{-1}\oplus 3^{+1}$ and there is no nontrivial saturation of $S\oplus E_6$. Since $3\times 7=21\equiv -3$ (mod 8), there is no negative rank $2$ even lattice with discriminant form $4_2^{+2}3^{-1}7^{+1}$. Thus there does not exist embedding of $S\oplus E_6$ into $\BB$.
\item The discriminant form of $S\oplus E_6$ is $8_2^{-2}3^{+1}\oplus 3^{+1}$ and there is no nontrivial saturation of $S\oplus E_6$ with $S$ primitive. Since $3\times 3=9\equiv 1$ (mod 8), there is no negative rank $2$ even lattice with discriminant form $8_6^{-2}3^{+2}$. Thus there is no embedding of $S\oplus E_6$ into $\BB$ with image of $S$ primitive.
\item The discriminant form of $S\oplus E_6$ is $2_{\II}^{-2}3^{+2}5^{-1}\oplus 3^{+1}$. There is a nontrivial saturation $\widetilde{S}$ of $S\oplus E_6$ with discriminant form $2_{\II}^{-2}3^{-1}5^{-1}$. Since $3\times 5=15\equiv -1$ (mod 8), there is no negative rank $2$ even lattice with discriminant form $2_{\II}^{-2}3^{+1}5^{+1}$. Thus there is no primitive embedding of $\widetilde{S}$ into $\BB$.
\item The discriminant form of $S\oplus E_6$ is $2_5^{-1}4_1^{+1}3^{+1}5^{+1}\oplus 3^{+1}$ and there is no nontrivial saturation of $S\oplus E_6$. Since $2\times 4\times 5=40\equiv 1$ is a square in $\ZZ_3$, there exists a negative rank $2$ even lattice $T$ with discriminant form $2_3^{-1}4_7^{+1}3^{+2}5^{+1}$. Thus there exists primitive embedding of $S\oplus E_6$ into $\BB$ with orthogonal complement $T$.
\item The discriminant form of $S\oplus E_6$ is $11^{+2}\oplus 3^{+1}$, and there is no nontrivial saturation of $S\oplus E_6$. Since $3$ is a square in $\ZZ_{11}$ (notice that $5^2\equiv 3$ (mod 11)), there exists a negative rank $2$ even lattice $T$ with discriminant form $3^{-1}11^{+2}$. Thus there exists a primitive embedding of $S\oplus E_6$ into $\BB$ with orthogonal complement $T$.
\item The discriminant form of $S\oplus E_6$ is $4_1^{-1}8_7^{-1}3^{+2}\oplus 3^{+1}$. There is a nontrivial saturation $\widetilde{S}$ of $S\oplus E_6$ with discriminant form $4_1^{-1}8_7^{-1}3^{-1}$. Since $3$ does not congruent to $\pm 1$ modulo 8, there is no negative rank $2$ even lattice with discriminant form $4_7^{+1}8_1^{+1}3^{+1}$. Thus there is no embedding of $\widetilde{S}$ into $\BB$.
\item The discriminant form of $S\oplus E_6$ is $3^{-2}5^{-2}\oplus 3^{+1}$. There is uniquely a nontrivial saturation $\widetilde{S}$ of $S\oplus E_6$ with discriminant form $3^{+1}5^{-2}$. Since $3$ is not a square in $\ZZ_5$, there exists a negative rank $2$ even lattice $T$ with discriminant form $3^{-1}5^{-2}$. Thus there exists a primitive embedding of $\widetilde{S}$ into $\BB$ with orthogonal complement $T$.
\item The discriminant form of $S\oplus E_6$ is $2_{\II}^{+2}7^{+2}\oplus 3^{+1}$, and there is no nontrivial saturation of $S\oplus E_6$. Since $2\times 2\times 3=12$ is not a square in $\ZZ_7$, there is no negative rank $2$ even lattice with discriminant form $2_{\II}^{+2}3^{-1}7^{+2}$. Thus there is no embedding of $S\oplus E_6$ into $\BB$.
\item The discriminant form of $S\oplus E_6$ is $2_7^{-1}4_7^{-1}3^{+1}9^{+1}\oplus 3^{+1}$, and there is no nontrivial saturation of $S\oplus E_6$. Since $l_3(2_7^{-1}4_7^{-1}3^{+1}9^{+1}\oplus 3^{+1})=3$, there is no negative rank $2$ even lattice with discriminant form the opposite of $2_7^{-1}4_7^{-1}3^{+1}9^{+1}\oplus 3^{+1}$. Thus there is no embedding of $S\oplus E_6$ into $\BB$.
\end{enumerate}

The proposition follows.
\end{proof}

\subsection{Cubics with special groups (cyclic, Klein, and $S_3$) of automorphisms}\label{subsec-min}
Theorem \ref{proposition: maximal case} classifies the $0$-dimensional moduli spaces $\calM_{(G,S)}$. The top dimensional moduli spaces $\calM_{(G,S)}$ will correspond to small groups $G$. In particular, the minimal elements in the poset $\mathscr{A}_{cub}$ can be determined by considering $G$ to be a cyclic group of prime order. The cubics with symplectic action of prime order were studied previously, especially by Fu \cite{fu2016classification} (see also \cite{gonzalez}), who classified all the possibilities for prime-power  symplectic automorphisms.

\begin{thm}[{Fu \cite[Theorem  1.1]{fu2016classification}}]
\label{theorem: lie fu}
Let $X=V(F)\subset \PP^5$ be a smooth cubic fourfold with a symplectic action by a prime-power order cyclic group $G=\langle g\rangle$. We can choose coordinates $(x_1, x_2, \cdots, x_6)$ on $\PP^5$, and generator $g\in G$, such that $(g, F)$ belong to one of the following cases (N.B. the cases are arranged such that the associated moduli space $\calF$ is irreducible and non-empty):
\begin{enumerate}
\item[(0)] $\ord(g)=1$, $g=id$, $\dim(\calF)=20$, and $F$ any smooth cubic.
\item $\ord(g)=2$, $g=\frac{1}{2}(0,0,0,0,1,1)$, $\dim(\calF)=12$, and 
\begin{equation*}
F=F_1(x_1, x_2, x_3, x_4)+x_5^2 L_1(x_1, x_2, x_3, x_4)+x_5 x_6 L_2(x_1, x_2, x_3, x_4)+ x_6^2 L_3(x_1, x_2, x_3, x_4);
\end{equation*} 
\item $\ord(g)=4$, $g=\frac{1}{4}(0,0,2, 2, 1, 3)$, $\dim(\calF)=6$, and  
\begin{equation*}
F\in \Span\{x_1 N_1(x_3,x_4), x_2 N_2(x_3, x_4), F_1(x_1, x_2), x_5 x_6 L_1(x_1, x_2), x_5^2 L_2(x_3, x_4), x_6^2 L_3(x_3, x_4)\};
\end{equation*}
\item $\ord(g)=8$, $g=\frac{1}{8}(0, 4, 2, 6, 1, 3)$, $\dim(\calF)=2$, and 
\begin{equation*} 
F\in\Span\{x_1^3, x_1 x_2^2, x_2 x_3^2, x_2 x_4^2, x_1 x_3 x_4, x_4 x_5^2, x_3 x_6^2, x_2 x_5 x_6\};
\end{equation*} 
\item $\ord(g)=3$, $g=\frac{1}{3}(0,0,0,0,1, 2)$, $\dim(\calF)=8$, and 
\begin{equation*}
F=F_1(x_1, x_2, x_3, x_4)+x_5^3+x_6^3+x_5 x_6 L_1(x_1, x_2, x_3, x_4);
\end{equation*}
\item $\ord(g)=3$, $g=\frac{1}{3}(0,0,1, 1, 2, 2)$, $\dim(\calF)=8$, and \begin{equation*}
F=F_1(x_1, x_2)+F_2(x_3,x_4)+F_3(x_5,x_6)+\Sigma_{i=1,2; j=3,4; k=5,6} (a_{ijk}x_i x_j x_k);
\end{equation*}
\item $\ord(g)=3$, $g=\frac{1}{3}(0,0,0,1, 1, 1)$,  $\dim(\calF)=2$, and 
\begin{center}
$F\in\Span\{$\it{monomials in} $x_1, x_2, x_3$, \it{monomials in} $x_4, x_5, x_6\}$;
\end{center} 
\item $\ord(g)=9$, $g=\frac{1}{9}(0, 6, 3, 1, 4, 7)$, $\dim(\calF)=0$, and  
\begin{equation*}
F\in\Span\{x_1^2 x_2, x_2^2 x_3, x_3^2 x_1, x_4^2 x_5, x_5^2 x_6, x_6^2 x_4\}; 
\end{equation*}
\item $\ord(g)=9$, $g=\frac{1}{9}(0, 3, 6, 1, 1, 4)$,  $\dim(\calF)=0$, and 
\begin{equation*}
F\in\Span\{x_1^2 x_2, x_2^2 x_3, x_3^2 x_1, x_4^2 x_5, x_4 x_5^2, x_4^3, x_5^3, x_6^3\};
\end{equation*} 
\item $\ord(g)=5$, $g=\frac{1}{5}(0,0,1, 2, 3, 4)$, $\dim(\calF)=4$, and 
\begin{equation*}
F=F_1(x_1, x_2)+x_3 x_6 L_1(x_1, x_2)+x_4 x_5 L_2(x_1, x_2)+x_3^2 x_5+x_3 x_4^2+x_4 x_6^2+x_5^2 x_6;
\end{equation*} 
\item $\ord(g)=7$, $g=\frac{1}{7}(1, 5, 4, 6, 2, 3)$, $\dim(\calF)=2$, and  
\begin{equation*}
F=x_1^2 x_2+ x_2^2 x_3+ x_3^2 x_4+ x_4^2 x_5+ x_5^2 x_6+x_6^2 x_1+ a x_1 x_3 x_5+b x_2 x_4 x_6;
\end{equation*} 
\item $\ord(g)=11$, $g=\frac{1}{11}(1, 9, 4, 3, 5, 0)$, $\dim(\calF)=0$, and 
\begin{equation*}
F\in\Span\{x_1^2 x_2, x_2^2 x_3, x_3^2 x_4, x_4^2 x_5, x_5^2 x_1, x_6^3\}.
\end{equation*}
\end{enumerate}
Moreover, in all situations, the generic cubic fourfolds defined are smooth.
\end{thm}

\begin{rmk}\label{rmk_sympl_cond}
For further reference, we give the condition for a diagonal matrix $g=\frac{1}{n} (w_1,\dots,w_6)\in \GL(6)$ to act symplectically on a cubic $X=V(F)$. Denote by $\underline{w}=(w_1,\dots,w_6)\in (\ZZ/n)^6$ the set of weights. Then a simple application of Griffiths' residue calculus (see \cite[Lemma 3.2]{fu2016classification}) gives that  
 $g$ acts symplectically on $X$ iff 
\begin{equation}\label{eq_cond_symp}
|\underline w|\equiv 2 \deg_{\underline w} (F) \pmod n
\end{equation}
where $|\underline w|=\sum_{i=1}^6 w_i$ and $\deg_{\underline w} (F)=\sum_{i=1}^6 w_i \alpha_i$ for some monomial $x_1^{\alpha_1}\dots x_6^{\alpha_6}$ occurring with non-zero coefficient in $F$ (N.B. since $V(F)$ is stabilized by $g$, $\deg_{\underline w} (F)$ is well defined in $\ZZ/n$). For most of the cases above, it holds $|\underline{w}|=\deg_{\underline w}(F)=0$ (equivalently $g\in \SL(6)$), but this does not hold always (e.g. the case $\ord(g)=9$ above). 
\end{rmk}

From the lattice theoretic approach (our main approach in this paper), Fu's classification is closely related to Harada--Lang classification \cite{harada1990leech} of fixed-point sublattices in the Leech lattice with respect to cyclic groups (see  Remark \ref{rmk-invo} for the case of involutions).  In fact, using the lattice theoretic approach and \cite{harada1990leech}, we can improve Fu's result.  Specifically, the following holds: 

\begin{thm}\label{prop_HL}
Let $G$ be a cyclic group acting symplectically on some smooth cubic fourfold $X$ (i.e., $G\subset \Aut^s(X)$). Then, the order of $|G|$ is one of the following:
$$|G|\in \{1,2,3,4,5,6,7,8,9,11,12,15\}$$
Furthermore, the following holds:
\begin{enumerate}
\item[(1)] (Prime-power Cases). For the cases $|G|=p^k$, we have the following 
correspondences among Fu's classification, Harada-Lang classification and H\"ohn-Mason classification.
\begin{table}[htb!]
\renewcommand{\arraystretch}{1.2}\centering
\begin{tabular}{l|c|c|c|c|c|c|c|c|c|c|c|c}
Case in Thm. \ref{theorem: lie fu}&(0)&(1)&(2)&(3)&(4)&(5)&(6)&(7)&(8)&(9)&(10)&(11)\\
\hline
Case in \cite{harada1990leech}&$1_A$&$2_A$&$4_C$&$8_E$&$3_B$&$3_B$&$3_C$&$9_C$&$9_C$&$5_B$&$7_B$&$11_A$\\
\hline
Case in \cite{HM1} &$1$&$2$&$9$& $55$ & $4$&$4$&$35$&$101$&$101$&$20$&$52$&$120$\\
\hline
The saturated group &$1$&$2$&$4$& $\QD_{16}$ & $3$ &$3$&$3^{1+4}:2$&$3^4:A_6$&$3^4:A_6$&$D_{10}$&$F_{21}$&$L_2(11)$
\end{tabular}
\end{table}
\item[(2)] (Composite Cases). There are $4$ Leech pairs $(G,S)$ occurring for cubic fourfolds with $G$ cyclic of order $n$ divisible by two distinct primes. 
\begin{table}[htb!]
\renewcommand{\arraystretch}{1.2}\centering
\begin{tabular}{l|c|c|c|c}
Case in \cite{harada1990leech}&$-6_D$&$6_E$&$-12_H$&$15_D$\\
\hline
Case in \cite{HM1} &$35$&$18$& $109$ & $128$\\
\hline
The saturated group &$3^{1+4}:2$&$D_{12}$& $3^{1+4}:2.2^2$ & $A_{3,5}$
\end{tabular}
\end{table}
\item[(3)] (Maximal Cases). A cubic fourfold with a symplectic automorphism of order $9$, $11$, $12$, or $15$ is isolated in moduli (i.e., $\dim \calM_{(G,S)}=0$). 
\end{enumerate}
\end{thm}

\begin{rmk}
The maximal cases in item (3) above are in fact unique. This is proved in \S\ref{subsec-uniqueness} below. Thus, considering cubic fourfolds with a symplectic action by a cyclic group of order $\ge 9$ gives four of the maximal cases listed in Theorem \ref{theorem: maximal uniqueness}. 
\end{rmk}

\begin{proof}
Harada--Lang \cite{harada1990leech} classified the conjugacy classes of cyclic subgroups in the Conway group $\Co_0$ and their associated fixed lattices $K$ (recall $S=K^\perp_\LL$). The necessary Condition \ref{condition: 1} says (in particular) $\rank(K)\ge 4$ and that $K$ is not divisible as a lattice (i.e., $K=K'(n)$ for some integral, not necessarily even, lattice $K'$ and integer $n\ge 2$, because in this situation $\rank(K)=l(A_K)$). Inspecting the list of \cite{harada1990leech} in the prime-power order case gives an easy match with the list of Theorem \ref{theorem: lie fu} (essentially, there is only one possibility for $(G,S)$ once the order of $G$ and the rank of $S$ are specified). The pairs $(G,S)$ are not saturated, but the knowledge of $K$ (essentially, rank and discriminant) suffices to identify the relevant case in H\"ohn--Mason \cite{HM1} list, and to find the saturated pair $(G',S)$ (with $G\subset G'$). 

Assuming that $n=|G|$ has at least $2$ prime divisors, and that $K$ is a non-divisible lattice of rank at least $4$, leaves only the following cases in  \cite{harada1990leech}: $-6_D$, $6_E$, $-10_E$, $-12_H$, $14_B$ and $15_D$. As before, for each case we can associate a unique saturated Leech pair from \cite{HM1}. Using Theorem \ref{theorem: criterion lattice from cubic fourfold} (our main criterion), cases $-10_E$ and $14_B$ can not arise from cubic fourfolds, while the others can occur. Finally, the cases $-12_H$ and $15_D$ correspond to maximal cases (i.e., $\rank(K)=4$, or equivalently $\rank(S)=20$). Considering also the cases of order $9$ and $11$ identified in Theorem \ref{theorem: lie fu}, we obtain item (3) (compare also with Theorem \ref{proposition: maximal case}). 
\end{proof}

\begin{rmk} Let us comment on the two apparent repetitions in the matching of the cases in Theorem \ref{prop_HL}. First, the two order $9$ cases (case (8) and (9)) correspond to a unique cubic fourfold, in fact the Fermat cubic fourfold
$$X=V(x_1^3+\dots+x_6^3)\subset \PP^5,$$
which has $\Aut^s(X)=3^4: A_6$. The fact that we list two cases of order $9$ in Theorem  \ref{theorem: lie fu} corresponds to the existence of two non-conjugate cyclic subgroups of order $9$ in $3^4: A_6$ (induced from the two conjugacy classes of order $3$ elements in $A_6$). For reference, we note (cf. \cite[Case $9_C$]{harada1990leech}) that the fixed-point lattice $K$ is
\begin{equation*}
\left(
\begin{array}{cccc}
4 & 1 & 1 & 2\\
1 & 4 & 1 & 2\\
1 & 1 & 4 & -1\\
2 & 2 & -1 & 4
\end{array}
\right)
\end{equation*}
which has discriminant form $3^{+2} 9^{+1}$. The cases (4) and (5) of order $3$ lead to the same Leech pair $(G,S)$ (with $K=S^\perp_\Lambda$ being the Coxeter--Todd lattice), but in this case the two ($8$-dimensional) families of cubics are different corresponding to the fact that $S$ has two different primitive embeddings into the lattice $\Lambda_0(=A_2\oplus (E_8)^2\oplus U^2$). The other order $3$ case (namely (6)) is easily distinguished; it corresponds to $K$ being $E_6^*(3)$ which has discriminant form $3^{+5}$. 
\end{rmk}

\begin{rmk}\label{case6d}
Let us also note that the order $6$ case $-6_D$ in fact coincides with the case $3_C$. This is clear by noticing that they both correspond to case 35 in \cite{HM1} (with saturated group $3^{1+4}:2$). This also follows by inspecting \cite{harada1990leech}; in both cases $K=E_6^*(3)$ (N.B. $E_6^*$ is not an integral lattice, thus scaling by $3$ does not contradict our non-divisibility assumption on $K$). 
\end{rmk}

\begin{rmk} The order $11$ case is very interesting, as $11$ can not occur as a prime order for symplectic automorphisms of $K3$ surfaces (and thus this example can be used to construct exotic automorphisms for hyper-K\"ahler's of $K3^{[2]}$ type; e.g. \cite[\S4.5]{mongardi2013thesis}). The equation of the unique cubic with an order $11$ symplectic automorphism is well known, namely
\begin{equation*}
X=V(x_1^3+x_2^2 x_3+x_3^2 x_4+ x_4^2 x_5+x_5^2 x_6+x_6^2 x_2).
\end{equation*}
From our perspective, this corresponds to  case $(11_A)$ in \cite{harada1990leech}. The saturated Leech pair is $(\PSL(2, \FF_{11}), S)$ and the fixed-point lattice $K$ is 
\begin{equation*}
\left(
\begin{array}{cccc}
4 & 0 & 2 & -1\\
0 & 4 & -1 & 2\\
2 & -1 & 4 & -1\\
-1 & 2 & -1 & 4
\end{array}
\right)
\end{equation*}
which has discriminant form $11^{+2}$.
\end{rmk}

In view of Theorem \ref{prop_HL}, we note that the only cyclic case that needs further investigation is $G\cong \ZZ/6$ (the prime-power cases are covered by Theorem \ref{theorem: lie fu}, while the maximal cases are discussed later in \S\ref{subsec-uniqueness}). According to Theorem \ref{prop_HL}, there are two order $6$ cases relevant for us ($6_E$ and $-6_D$). However, the case $-6_D$ was already covered by Theorem \ref{theorem: lie fu} (cf. Rem. \ref{case6d}). The last cyclic group case is handled by the following result.

\begin{lem}
\label{lemma: 6}
Let $X$ be a smooth cubic fourfold with a symplectic automorphism of order $6$. Suppose the moduli of cubic fourfolds with such an automorphism has dimension more than $2$ (i.e., $\dim \calM_{(G,S)}>2$). Then for an appropriate choice of coordinates, the defining equation for $X$ either belongs to 
\begin{equation*}
\Span \{x_1^2 x_3, x_1^2 x_4, x_1 x_2 x_3, x_1 x_2 x_4, x_2^2 x_3, x_2^2 x_4, x_3^3, x_3^2 x_4, x_3 x_4^2, x_3 x_5 x_6, x_4^3, x_4 x_5 x_6, x_5^3, x_6^3\},
\end{equation*}
while the order $6$ automorphism is $\frac{1}{6}(3, 3, 0, 0, 2, 4)$, or belongs to
\begin{equation*}
\Span \{x_1^3, x_1 x_2^2, x_1 x_3 x_5, x_1 x_3 x_6, x_2 x_4 x_5, x_2 x_4 x_6, x_3^3, x_3 x_4^2, x_5^3, x_5^2 x_6, x_5 x_6^2, x_6^3\},
\end{equation*}
while the order $6$ automorphism is $\frac{1}{6}(0, 3, 2, 5, 4, 4)$. In both cases, the corresponding moduli spaces $\calF$ have dimension $4$. They both correspond to the case $6_E$ in \cite{harada1990leech}, and the associated saturated group is $D_{12}$. 
\end{lem}

\begin{proof}
Denote by $\rho$ the order $6$ automorphism. Since the moduli of cubic fourfolds with such an automorphism has dimension more than $2$, the order $3$ automorphism $\rho^2$ belongs to cases $(4)$ or $(5)$ in Theorem \ref{theorem: lie fu}. Thus, we can choose coordinates $(x_1, x_2, \cdots, x_6)$ such that $\rho^2=\frac{1}{3}(0,0,0,0,1,2)$ or $\frac{1}{3}(0, 0, 1, 1, 2, 2)$, meanwhile $\rho^3$ has two $-1$ on the diagonal. Denote by $F=F(x_1, \cdots, x_6)$ a defining equation for $X$. Then $F$ is a linear combination of $\rho$-invariant monomials in $x_1, \cdots, x_6$. Since $X$ is smooth, there exists a $\rho$-invariant monomial divisible by $x_i^2$ for any $i=1,2,\cdots, 6$ . 

If $\rho^2=\frac{1}{3}(0,0,0,0,1,2)$, then a $\rho^2$-invariant monomial divisible by $x_5^2$ must be $x_5^3$. Therefore $x_5^3$ is $\rho$-invariant, hence also $\rho^3$-invariant. So does $x_6^3$. We may then take $\rho^3=\frac{1}{2}(1,1,0,0,0,0)$. Then $\rho=\frac{1}{6}(3, 3, 0, 0, 2, 4)$ and 
\begin{equation*}
F\in \Span \{x_1^2 x_3, x_1^2 x_4, x_1 x_2 x_3, x_1 x_2 x_4, x_2^2 x_3, x_2^2 x_4, x_3^3, x_3^2 x_4, x_3 x_4^2, x_3 x_5 x_6, x_4^3, x_4 x_5 x_6, x_5^3, x_6^3\}.
\end{equation*}
This $14$-dimensional vector space contains $x_1^2 x_3+ x_2^2 x_4+ x_3^3+ x_4^3+ x_5^3+ x_6^3$ which determines a smooth cubic fourfold. Therefore, a generic cubic fourfold with this automorphism $\rho$ is smooth. The dimension of the centralizer of $\rho$ in $\GL(6, \CC)$ is $4+4+1+1=10$, hence the dimension of the moduli space $\calF$ is $14-10=4$.

If $\rho^2=\frac{1}{3}(0,0,1,1,2,2)$, then a $\rho^2$-invariant monomial divisible by $x_1^2$ must be $x_1^3$ or $x_1^2 x_2$. Therefore, the two $-1$ on the diagonal of $\rho^3$ can not occupy the first two positions simultaneously. So do the third and fourth positions, the fifth and sixth positions. By symmetry, we may take $\rho=\frac{1}{6}(0, 3, 2, 5, 4, 4)$ and 
\begin{equation*}
F\in \Span \{x_1^3, x_1 x_2^2, x_1 x_3 x_5, x_1 x_3 x_6, x_2 x_4 x_5, x_2 x_4 x_6, x_3^3, x_3 x_4^2, x_5^3, x_5^2 x_6, x_5 x_6^2, x_6^3\}.
\end{equation*}
This $12$ dimensional vector space contains $x_1^3+x_1 x_2^2+ x_3^3+ x_3 x_4^2+ x_5^3+ x_6^3$ which determines a smooth cubic fourfold, hence a generic element also determines smooth cubic fourfold. Moreover, the dimension of the centralizer of $\rho$ in $\GL(6, \CC)$ is $1+1+1+1+4=8$, hence the dimension of the moduli space $\calF$ is $12-8=4$.
\end{proof}

\subsubsection{Small non-cyclic groups} In addition to the cyclic groups identified above, we discuss also the cases of cubics with symplectic action by the simplest non-cyclic groups, Klein group and respectively $S_3$. First, for the Klein group $\ZZ/2\times \ZZ/2$, relevant to item $(c1)$ in Theorem  \ref{theorem: main}, the following holds. 

\begin{lem}
\label{lemma: 2^2}
Suppose $X=V(F)$ is a smooth cubic fourfold with symplectic action of $G\cong 2^2$. Then we can choose coordinate $(x_1, x_2, x_3, x_4, x_5, x_6)$ for $V$ such that $G=\langle\diag(1,1,1,1,-1,-1), \diag(1,1,1,-1,-1,1)\rangle$, and $F$ can be written as $F_1(x_1, x_2, x_3)+x_4^2 L_1(x_1, x_2,x_3)+x_5^2 L_2(x_1, x_2, x_3)+ x_6^2 L_3(x_1, x_2, x_3)+x_4 x_5 x_6$. The dimension of the associated moduli space $\calF$ is $8$.
\end{lem}
\begin{proof}
Since $G$ is a finite abelian subgroup of $\PSL(V)$, we can choose coordinate $(x_1, x_2, x_3, x_4, x_5, x_6)$ for $V$, such that all element in $G$ are diagonal matrices. For any $g\in G$, since $g^2=id$ and $g$ acts symplectically on the smooth cubic fourfold $X$, there are four eigenvalues $1$ and two eigenvalues $-1$ (see Theorem \ref{theorem: lie fu}(1)). We now choose generators $g_1, g_2$ of $G$. Up to coordinate choices, we may assume $g_1=\diag(1,1,1,1,-1,-1)$, and $g_2$ is either $\diag(1,1,1,-1,-1,1)$ or $\diag(1,1,-1,-1,1,1)$. Suppose $g_2=\diag(1,1,-1,-1,1,1)$, then there is no smooth cubic fourfold preserved by the action of $G$. Thus $g_2=\diag(1,1,1,1,-1,-1)$. The defining polynomial $F$ can then be written as 
\begin{equation*}
F_1(x_1, x_2, x_3)+x_4^2 L_1(x_1, x_2,x_3)+x_5^2 L_2(x_1, x_2, x_3)+ x_6^2 L_3(x_1, x_2, x_3)+x_4 x_5 x_6.
\end{equation*}
A generic cubic of this type is smooth. Moreover, the dimension of the vector space of such cubic polynomials is $10+3+3+3+1=20$, and the dimension of the reductive group $\GL(3)\times \CC^{\times}\times \CC^{\times}\times \CC^{\times}\subset \GL(6)$ preserving the normal form is $9+1+1+1=12$. Thus $\dim(\calF)=20-12=8$.
\end{proof}

We now consider the symmetric group $S_3$, relevant to item $(d2)$ in Theorem  \ref{theorem: main}.

\begin{lem} 
\label{lemma: S_3}
Let $X=V(F)\subset \PP(V)$ be a smooth cubic fourfold with symplectic action of $G\cong S_3$. Then the action of $G$ on $\PP^5$ can be lifted to a representation of $G$ on $V\cong \CC^6$, and one of the following holds: \begin{enumerate}[(1)]
\item The representation of $G$ on $V$ is the direct sum of two standard representations of $S_3$. The dimension of the moduli space of cubic fourfolds $\calF$ with such an action is $6$.

\item The representation of $G$ on $V$ is the direct sum of a standard representation, an alternating character, and two trivial characters of $S_3$. The dimension of the moduli space of cubic fourfolds $\calF$ with such an action is $4$.

\end{enumerate}

\end{lem}
\begin{proof}
A projective representation of $S_3$ can be lifted as a linear representation. Suppose we have an action of $S_3$ on $V$ with an invariant smooth cubic form $F\in \Sym^3(V^*)$, such that the induced action of $S_3$ on $V(F)$ is faithful and symplectic. There are three involutions in $S_3$, and their actions on $V$ must have dimensional two $(-1)$-eigenspace. 

There are three linear irreducible representations of $S_3$, namely, the trivial character, the alternating character, and the standard representation on $\CC^3$. Since the action of an order $3$ element in $G$ is acting faithfully on $V$, the representation of $G$ on $V$ has the standard representation of $S_3$ as an irreducible component. It is then clear that the two cases mentioned in the lemma is all the possibilities.

Suppose $V$ is a direct sum of two standard representation. We can choose coordinate $(x_1, x_2, x_3, x_4, x_5, x_6)$ of $V^*$, such that $G\cong S_3$ is acting via permutating $(x_1, x_2), (x_3, x_4), (x_5, x_6)$ simultaneously. A cubic form which is invariant under this action can be written uniquely as a linear combinations of $14$ cubic forms which are also invariant. The centralizer group of $S_3$ in $\GL(V)$ can be written as
$\left(
\begin{array}{ccc}
A & B & B\\
B & A & B\\
B & B & A
\end{array}
\right)$
where $A,B$ are two by two matrices. This group has dimension $8$. Hence the dimension of the moduli of cubic fourfolds with this action is $14-8=6$.

Suppose $V$ is a direct sum of a standard representation, an alternating character, and two trivial characters. We can choose coordinate $(x_1, x_2, x_3, x_4, x_5, x_6)$ of $V^*$, such that $G\cong S_3$ is acting via permutating $(x_1, x_2, x_3)$, identically on $x_5, x_6$, and alternatively on $x_4$. A cubic form which is invariant under this action can be written uniquely as a linear combinations of $15$ cubic forms which are also invariant. The centralizer group of $S_3$ in $\GL(V)$ has dimension $11$. Thus, the dimension of the moduli space of cubic fourfolds with this action is $15-11=4$.
\end{proof}

\subsection{Proof of Theorem  \ref{theorem: main}}\label{subsec-completeproof}
At this point, we can complete the proof on our classification theorem (Theorem \ref{theorem: main}). The main ingredients of our proof are the 
 criterion given by Theorem  \ref{theorem: criterion lattice from cubic fourfold}, the H\"ohn-Mason classification \cite{HM1} of the fixed-point sublattices in the Leech lattice $\LL$, and Fu's classification discussed above (Theorem  \ref{theorem: lie fu}). Nikulin's criterion for the existence of even lattices with specified discriminant form (Theorem  \ref{theorem: existence of lattice via discriminant form}) is a well-known tool that we use repeatedly.

H\"ohn and Mason \cite{HM1} list all possibilities ($290$ in total) for saturated Leech pairs $(G,S)$. Condition \ref{condition: 1} allows us to rapidly remove a large number of cases (e.g., about half of the cases have $\rank(S)\ge 21$). We analyze the remaining cases one by one using Theorem  \ref{theorem: criterion lattice from cubic fourfold} (our main criterion) and Nikulin's theory. The most delicate case, $\rank(K)=4$, was analyzed in detail in Theorem  \ref{proposition: maximal case}. The cases when $\rank(K)\ge 5$ are similar and in fact easier. Namely, as $K$ becomes larger, it is easier to embed $E_6$ into $K\oplus U^2$ (in particular, note that except $\rank(K)=4$, $(E_6)^\perp_{K\oplus U^2}$ is indefinite, i.e., the ``easy'' case of Nikulin's theory). By a routine inspection (we only need to compare the rank of $K$ and $l_p(A_K)$) of the list of H\"ohn--Mason, we see that there are $43$ cases (among them, there are $12, 12, 5, 5, 2, 3, 2, 1, 1$ cases with $\rank(K)=5, 6, 7, 8, 9, 10, 12, 16, 24$ respectively) in H\"ohn-Mason list with $\rank(K)\ge 5$ and satisfying Condition \ref{condition: 1}. Out of these $43$ potential cases with $\rank(K)\ge 5$, only $28$ of them satisfy the equivalent conditions in our main criterion Theorem \ref{theorem: criterion lattice from cubic fourfold}. We omit the details. Including the $6$ cases of maximal rank, we obtain the list of $34$ possibilities for  $(G,S)\in\mathscr{A}_{cub}$. We list them in Theorem  \ref{theorem: main} in decreasing order of dimension of moduli $\calM_{(G,S)}$ (or equivalently by $\rank(S)$). (Note however that $\calM_{(G,S)}$ is not necessarily irreducible. When possible, we list also the irreducible components of $\calM_{(G,S)}$.)

The second part of Theorem  \ref{theorem: main} is to give explicit equations for some of the cases. As discussed above, Theorem  \ref{theorem: lie fu}, Lemma \ref{lemma: 6}, Lemma  \ref{lemma: 2^2} and Lemma  \ref{lemma: S_3} give normal equations for cubic fourfolds $X$ which admit faithful actions by some special group $G$ (either cyclic of prime-power order, $\ZZ/6$, Klein group or $S_3$ respectively). Starting with this classification, we proceed in two ways. First, we have the saturation procedure: given a normal form $F$ stabilized by such a $G$, we obtain a stratum $\calF\subset \calM$ which  corresponds to some Leech pair $(G',S)\in \mathscr{A}_{cub}$ (i.e., in the list of the previous paragraph). It holds $G\subset G'=\Aut^s(X)$ for some generic $X$ in $\calF$.  Typically, using the information on order of $G$ (note $\ord(G')$ is a multiple of $\ord(G)$) and $\dim \calF(=20-\rank(S))$ suffices to identify the pair $(G',S)$. As an illustration of this saturation procedure see item (5) case $D_{10}$ in Theorem  \ref{theorem: main}.

A second way to proceed is to start with $(G, S)\in \mathscr{A}_{cub}$, and consider elements of prime-power order $g\in G$ (say $\ord(g)=p^k$). By Theorem  \ref{theorem: lie fu}, we know the possible normal form(s) $F$ of $X$ with an action by $g$  (similar arguments apply to $\ZZ/2\times \ZZ/2\subset G$ or $S_3\subset G$). We then try to specialize $F$ so that it admits an action by $G\supset \langle g\rangle$ (e.g., see proof of Lemma   \ref{lemma: S_3}). Again, the knowledge of the dimension of $\calF$ (from the normal form) and that of $\calM_{(G,S)}\subset \calF$ (Hodge theoretically, as $S$ is a prescribed lattice of algebraic cycles) proved very handy in practice.

Concretely, for $G=1,2,3$ or $4$, we can directly apply the second method ($G=\langle g\rangle$) and $(0), (1), (2b), (3a)$ are clear. For $(2a)$, we can apply the second method for $G=2^2$ and use Lemma \ref{lemma: 2^2}. For $(3b)$, we can apply the second method for $G=S_3$ and use Lemma  \ref{lemma: S_3}. For $(5a)$, we can apply the second method for $G=D_{12}$ and use Lemma \ref{lemma: 6}. Then applying the first method we see that a generic cubic fourfold described in Lemma \ref{lemma: 6} has symplectic automorphism group $D_{12}$. For items $(5b)$, $(7b)$ and $(7c)$ of Theorem \ref{theorem: main}, we apply a combination of the two methods. 

The last case left is $(7a)$. From Harada-Lang \cite{harada1990leech} (case $(3_C)$), there is a Leech pair $(G, S)$ with $G=\ZZ/3\ZZ$ and $K=E_6^*(3)$. From H\"ohn-Mason classification, the only saturated Leech sub-pair of $(\Co_0, \LL)$ with discriminant $3^5$ is $(3^{1+4}:2, S)$. Thus this is the saturation of $(G, S)$. By Theorem \ref{theorem: criterion lattice from cubic fourfold}, there exist cubic fourfolds with certain order $3$ automorphism such that the induced Leech pair is $(G, S)$. The moduli of such cubic fourfold has dimension $2$. These cubic fourfolds must be given by case $(6)$ in Theorem \ref{theorem: lie fu}. Using the first method described above, any cubic fourfold with such an order $3$ automorphism has automatically symplectic automorphism group $3^{1+4}:2$. We conclude case $(7a)$.
\qed

\subsection{Uniqueness in maximal case}\label{subsec-uniqueness} 
As discussed in the previous subsection, we are able to identify explicit equations for a number of cases in Theorem \ref{theorem: main}. The cases that are more difficult are those with large, non-abelian group. One further complication that can arise is the fact that $\calM_{(G,S)}$ might not be irreducible. We discuss in detail this situation for the maximal rank case, $\rank(S)=20$ or equivalently $\dim \calM_{(G,S)}=0$. In Theorem \ref{proposition: maximal case} we have identified six cases for such pairs $(G,S)$. On the other hand, H\"ohn--Mason \cite[page 48]{HM2}  have listed for each of these cases a cubic fourfold in $\calM_{(G,S)}$. It turns out, that in two of the six cases, there is an additional point in $\calM_{(G,S)}$. This is the new content of our Theorem \ref{theorem: maximal uniqueness}. Our arguments are lattice theoretic; we do not have explicit equations for these cubic fourfolds with large automorphism groups. We start with two lemmas:

\begin{lem}
\label{lemma: aut surj}
For Leech pairs $(G,S)$ with numbers $1$, $4$, $5$, $10$, $11$, or $13$ in Table \eqref{table: MSS aut of HK}, the natural group homomorphisms 
$$\Aut(S)\longrightarrow \Aut(q_S)$$ are surjective.
\end{lem}
\begin{proof}
Direct inspection of Table 9 in \cite{HM2}. 
\end{proof}

From the reduction theory of lattices (e.g.,  see \cite[Chap. 15, \S3.2]{conway1999spherepackings}), we have:
\begin{lem}
\label{lemma: red}
Every positive rank $2$ lattice admits a basis, such that the corresponding intersection matrix is $a^bc=\begin{pmatrix} a&b\\b&c\end{pmatrix}$ with $-a< 2b\le a\le c$, and $b\ge 0$ if $a=c$. In particular, we have $3b^2\le d=ac-b^2$.
\end{lem}

\begin{proof}[Proof of Theorem  \ref{theorem: maximal uniqueness}] 
 The issue that we need to investigate is the uniqueness of the primitive embedding $S\hookrightarrow \Lambda_0$ (where $\Lambda_0\cong A_2\oplus (E_8)^2\oplus U^2$ is the primitive cohomology of the cubic fourfold). We let $T=S^\perp_{\Lambda_0}$ be the transcendental lattice. The maximal rank case is very special, as $T$ is in fact a negative definite lattice of rank $24$ (in all other cases, $T$ is indefinite, the easy case of Nikulin's theory).  We now analyze case by case, the six cases of the Theorem  \ref{theorem: maximal uniqueness},  corresponding to  items $1$, $4$, $5$, $10$, $11$, or $13$ in Table \eqref{table: MSS aut of HK}.
 \begin{enumerate}[(1)]
\item[(1)] For $3^4:A_6$, the lattice $T$ has discriminant form $3^{+1}9^{+1}$. Using Lemma  \ref{lemma: red}, we see that the negative rank $2$ even lattices with discriminant $27$ are $-(2^1 14)$ and $-(6^3 6)$. Only $-(6^3 6)$ has discriminant form $3^{+1}9^{+1}$. Hence $T=-(6^3 6)$ is unique. A saturation $S\oplus T\hookrightarrow \Lambda_0$ is given by an injective morphism $-q_T\hookrightarrow q_S$. Every two such morphisms differ by an automorphism of $q_S$, which is induced by an automorphism of $S$ (from Lemma  \ref{lemma: aut surj}). Thus all primitive embeddings of $S$ into $\Lambda_0$ with orthogonal complement $T$ give the same primitive sublattice (up to automorphisms of $\Lambda_0$). Therefore, this case recovers a unique smooth cubic fourfold, which must be the Fermat cubic fourfold $V(x_1^3+ x_2^3+ x_3^3+ x_4^3+ x_5^3+ x_6^3)$.
\item[(2)] For $A_7$ and when the lattice $T$ has discriminant form $5^{+1}7^{+1}$. All negative rank $2$ even lattices with discriminant $35$ are $-(6^1 6)$ and $-(2^1 18)$.  After calculating their discriminant forms, we conclude $T=-(2^1 18)$. Similar to the previous case, all primitive embeddings of $S$ into $\Lambda_0$ with orthogonal complement $T$ give the same primitive sublattice. Therefore, this case recovers a unique smooth cubic fourfold which is the diagonal cubic fourfold $V(x_1^3+x_2^3+x_3^3+x_4^3+x_5^3+x_6^3-(x_1+x_2+x_3+x_4+x_5+x_6)^3)$ as we will show in section \ref{section: non-symplectic} using the existence of certain anti-symplectic involutions (equivalently, Eckardt points). See Corollary \ref{corollary: diagonal}.

\item[(2')] For $A_7$ and when the lattice $T^{\prime}$ has discriminant form $3^{-2}5^{+1}7^{+1}$. All negative rank $2$ even lattices with discriminant $315$ are $-(2^1 158)$, $-(6^3 54)$, $-(18^3 18)$, $-(10^5 34)$ and $-(14^7 26)$. After calculating their discriminant forms, we must have $T^{\prime}=-(18^3 18)$. A saturation $S\oplus T\hookrightarrow \Lambda_0$ is given by an injective morphism $q_S\hookrightarrow -q_T$. Given two such morphisms $\tau_1$ and $\tau_2$. Denote by $e_1,e_2$ the generators of $T$, with intersecting matrix $-(18^3 18)$. One element in the automorphism group of $T$ sends $(e_1,e_2)$ to $(e_1, e_2), (e_2,e_1),(-e_1,-e_2)$ or $(-e_2,-e_1)$. By simple calculation, we can choose an automorphism $\iota$ of $T$, such that $\tau_1$ and $\iota\circ\tau_2$ have the same image. Then $\tau_1$ and $\iota\circ\tau_2$ only differ by an automorphism of $q_S$. By Lemma  \ref{lemma: aut surj}, this is induced by an automorphism of $S$. Thus the two primitive embeddings corresponding to $\tau_1$ and $\tau_2$ have the same image in $\Lambda_0$. Therefore, this case recovers a unique smooth cubic fourfold. As we will show in section \ref{section: non-symplectic}, this cubic fourfold does not admit any Eckardt points, hence is distinguished from case $(2)$ above.

\item[(3)] For $3^{1+4}:2.2^2$, the lattice $T$ has discriminant form $2_2^{+2}3^{+2}$. All negative rank $2$ even lattices with discriminant $36$ are $-(2^0 18)$, $-(6^0 6)$ and $-(4^2 10)$. After calculating their discriminant forms, we must have $T=-(6^0 6)$. As in case $(1)$, all primitive embeddings of $S$ into $\Lambda_0$ with orthogonal complement $T$ give the same primitive sublattice. Therefore, this case gives rise to a unique smooth cubic fourfold. As constructed in \cite{HM2}, this cubic fourfold is $X(3^{1+4}:2.2^2)=V(x_1^3+x_2^3+x_3^3+x_4^3+x_5^3+x_6^3-3(\sqrt{3}+1)(x_1 x_2 x_3+x_4 x_5 x_6))$.

\item[(4)] For $M_{10}$, the discriminant form of $T$ is $2_3^{-1}4_7^{+1}3^{+2}5^{+1}$. All negative rank $2$ even lattices with discriminant $360$ are $-(2^0 180)$, $-(4^0 90)$, $-(6^0 60)$, $-(10^0 36)$, $-(12^0 30)$, $-(18^0 20)$, $-(14^2 26)$, $-(18^6 22)$ and $-(18^{-6} 22)$. After calculating their discriminant forms, we must have $T=-(12^0 30)$. There are two discriminant subforms of $-q_T=2_5^{-1} 4_1^{+1} 3^{+2} 5^{+1}$ that are isomorphic to $q_S=2_5^{-1} 4_1^{+1} 3^{+1} 5^{+1}$. Moreover, these two are not identified via an automorphism of $T$. Therefore, there are two non-conjugate embeddings of $S$ into $\Lambda_0$, both with orthogonal complement isomorphic to $T$. From \cite[page 48]{HM2} there is an explicit description for one smooth cubic fourfold with symplectic automorphism $M_{10}$:
\begin{equation}
\label{equation: m10}
X^1(M_{10})=x_1^3+\cdots+x_6^3+ \frac{1}{5}(-3 \zeta^7-3 \zeta^5+3 \zeta^4- 3 \zeta^3+ 6\zeta-3)\times F
\end{equation}
where $\zeta=e^{2\pi \sqrt{-1}/24}$ and $F=x_1x_2 x_3+x_1 x_2 x_4+(\zeta^4-1)x_1 x_2 x_5+x_1x_2 x_6+(\zeta^4-1)x_1x_3 x_4+x_1 x_3 x_5+x_1 x_3 x_6+(\zeta^4-1)x_1 x_4 x_5 -\zeta^4 x_1 x_4 x_6-\zeta^4 x_1 x_5 x_6+(\zeta^4-1)x_2 x_3 x_4+(\zeta^4-1)x_2 x_3 x_5-\zeta^4 x_2 x_3 x_6+ x_2 x_4 x_5+x_2 x_4 x_6-\zeta^4 x_2 x_5 x_6+x_3 x_4 x_5-\zeta^4 x_3 x_4 x_6+x_3 x_5 x_6+x_4 x_5 x_6$.

\item[(5)] For $L_2(11)$, the lattice $T$ has discriminant form $11^{+2}3^{-1}$. All negative rank $2$ even lattices with discriminant $363$ are $-(2^1 182)$, $-(14^1 26)$, $-(14^{-1} 26)$, $-(6^3 62)$ and $-(22^{11}22)$. After calculating their discriminant forms, we must have $T=-(22^{11}22)$. Similar as case $(3)$, the image of $S$ in $\Lambda_0$ is unique up to automorphisms of $\Lambda_0$. Thus this recovers a unique cubic fourfold, which must be $V(x_1^2 x_2+ x_2^2 x_3+ x_3^2 x_4+ x_4^2 x_5+ x_5^2 x_1+x_6^3)$. Namely, Adler \cite{adler1978automorphism} showed that $L_2(11)$ is the automorphism group of the Klein cubic threefold $V(x_1^2 x_2+ x_2^2 x_3+ x_3^2 x_4+ x_4^2 x_5+ x_5^2 x_1)$. It is then easy to see that 
$L_2(11)$ acts symplectically on the fourfold $V(x_1^2 x_2+ x_2^2 x_3+ x_3^2 x_4+ x_4^2 x_5+ x_5^2 x_1+x_6^3)$ (obtained as a cyclic cover of $\PP^4$ branched over the Klein cubic threefold). 

\item[(6)] For $A_{3,5}$, the lattice $T$ has discriminant form $3^{-1}5^{-2}$. All negative rank $2$ even lattices with discriminant $75$ are $-(2^1 38)$, $-(6^3 14)$ and $-(10^5 10)$. After calculating their discriminant forms, we must have $T=-(10^5 10)$. Similar to case $(1)$, the image of $S$ in $\Lambda_0$ is unique up to automorphism of $\Lambda_0$. Thus, this recovers a unique cubic fourfold, which must be 
\begin{equation*}
X(A_{3,5})=V(x_1^3+x_2^3+x_3^3+x_4^3+x_5^3+x_6^3+x_7^3+x_8^3)\cap V(x_1+x_2+x_3)\cap V(x_4+x_5+x_6+x_7+x_8).
\end{equation*}
as $A_{3,5}$ is acting symplectically on this cubic fourfold by permuting the two tuples $(x_1, x_2, x_3)$ and $(x_4, \cdots, x_8)$. \end{enumerate}
The remaining part of Theorem \ref{theorem: maximal uniqueness} on non-symplectic automorphisms are proved in Proposition \ref{proposition: non-sym order}.
\end{proof}


\section{Symplectic automorphisms for low degree $K3$ surfaces}
\label{section: K3}
In this section we will discuss the case of $K3$ surfaces. As we have indicated, the classification of symplectic automorphisms for $K3$ surfaces was first systematically investigated by Nikulin \cite{nikulin} via lattice theory, and culminated in the celebrated result by Mukai \cite{mukaiaut} on a characterization of maximal finite symplectic groups of $K3$ surfaces via Mathieu group $M_{23}$. Kond\=o \cite{kondo} simplified Mukai's proof by embedding the covariant lattice $S$ into a Niemeier lattice (an approach closely related to ours).  Xiao \cite{xiao1996k3} gave the complete list of finite symplectic automorphism groups of $K3$ surfaces by analyzing the combinatorial structures of the singularities of the quotient surface. Hashimoto \cite{hashimoto2012K3} extended Kond\=o's lattice theoretic approach to give the complete list and analyze the possibilities of geometric realizations.

We briefly discuss here the case of symplectic automorphisms for low degree polarized $K3$ surfaces, along the lines of our analysis for cubic fourfolds. Our method is lattice theoretic and relies on the H\"ohn--Mason \cite{HM1} classification. On the other hand, low degree $K3$ surfaces have projective models. For those $K3$ surfaces, one can study the automorphisms of the projective model via geometric methods; some partial results exist in the literature (e.g. \cite{harui2014}, \cite{doi2000sextic}, \cite{MPK}). Our discussion here only matches some of the maximal cases. A further analysis of the interplay between geometry and arithmetic would be interesting.

\subsection{General discussion}

As in the cubic fourfold case, the main point of our analysis is that for a $K3$ surface $Y$ with a faithful symplectic action of a finite group $G$, one gets a Leech pair $(G, S_G(Y)(-1))$ (see Lemma \ref{lemma: leech type pairs from cubic fourfold and K3 surface}). The task now is to identify those that occur for $Y$ a polarized $K3$ surface with a given degree. Similarly to our main criterion (Theorem \ref{theorem: criterion lattice from cubic fourfold}) for cubic fourfolds, we obtain the following criterion for Leech pairs to arise from  low degree $K3$ surfaces. Our arguments apply  essentially verbatim as in the proof of Theorem \ref{theorem: criterion lattice from cubic fourfold} for the cases when there exists a Borcherds polarization on $Y$ (see \S\ref{subsec:borcherds}) which is a root lattice. 
As already discussed, this is the case for degree $2$ and $6$. It is also true for the degree $4$ case (e.g., \cite[Sect. 1]{LOG1}). Finally, it also applies to elliptic $K3$ surfaces. By abuse of notation, we call an elliptic $K3$ surface {\it a degree $0$ $K3$ surface}, and we insist that the polarized symplectic automorphisms preserve the class of the fiber and of the section (i.e., the natural $U$ polarization for elliptic $K3$ surfaces is point-wise fixed by the automorphism). 

\begin{thm}
\label{theorem: criterion lattice from low degree k3}
Given a Leech pair $(G,S)$. Let $d\in\{0,2,4,6\}$ and $R_d$ be the root lattice $E_8$, $E_7$, $D_7$, or $E_6\oplus A_1$ for $d=0,2,4,6$ respectively. The following three statements are equivalent:
\begin{enumerate}[(i)]
\item there exists a smooth degree $d$ $K3$ surface $S$ with a symplectic action $G$ which preserves the polarization, such that $(G,S)\cong (G,S_G(X))$,
\item there exists an action of $G$ on $\LL$ with $S=S_G(\LL)$ and $K=\LL^G$, such that there exists a primitive embedding of $R_d$ into $K\oplus U^2$,
\item there exists an embedding of $S\oplus R_d$ into the Borcherds lattice $\BB$, such that $S$ has primitive image.
\end{enumerate}
\end{thm}

The maximal rank for $S_G(Y)(-1)$ in the $K3$ case is $19$ (or equivalently the orthogonal complement $K$ in the Leech lattice $\LL$ has rank $5$). From H\"ohn--Mason classification, we identify the following $11$  maximal cases in Table \eqref{table: MSS aut. of K3}\footnote{There are some typos in H\"ohn and Mason. For instance, they  write $4_3^{+1}8_2^{+2}$ in case $2$, and $2_3^{+3}3^{-1}9^{-1}$ in case $9$. These symbols are not allowed.}; they correspond precisely to the $11$ maximal cases of Mukai. It is interesting to note that all $11$ cases have projective models of degree at most $8$ (see \cite[Example 0.4]{mukaiaut}).

\begin{table}[ht] \caption{Maximal Finite Symplectic Automorphism Groups of K3}
\label{table: MSS aut. of K3}
\centering
\begin{tabular}{c c c c c}
\hline\hline
Number &Order &Group $\Aut^s(Y)$&Discriminant form &  $\deg(Y)$ in \cite[Ex. 0.4]{mukaiaut}\\ [0.5ex]
\hline
1 & 960 & $M_{20}$ & $2_{\II}^{-2}8_1^{+1}5^{-1}$ &4\\
2 & 384 & $4^2.S_4$ & $4_7^{+1}8_6^{+2}$ & 4\\
3 & 360 & $A_6$ & $4_5^{-1}3^{+2}5^{+1}$ & 6\\
4 & 288 & $A_{4,4}$ & $2_{\II}^{+2}8_1^{+1}3^{+2}$ & 8\\
5 & 192 & $2^4:D_{12}$ & $4_2^{-2}8_1^{+1}3^{-1}$ & 8\\
6 & 192 & $(Q_8 * Q_8):S_3$ & $4_7^{-3}3^{+1}$ & 4\\
7 & 168 & $L_2(7)$ & $4_1^{+1}7^{+2}$ & 4\\
8 & 120 & $S_5$ & $4_3^{-1}3^{+1}5^{-2}$ & 6\\
9 & 72 & $M_9$ & $2_7^{-3}3^{-1}9^{-1}$ & 2\\
10 & 72 & $N_{72}(\cong3^2.D_8)$ & $4_1^{+1}3^{+2}9^{-1}$ &6\\
11 & 48 & $T_{48}$ & $2_7^{+1}8_{\II}^{-2}3^{-1}$ &2\\
\hline
\end{tabular}
\end{table}

\begin{notation}
The notation of the finite groups appearing in Table \eqref{table: MSS aut. of K3} follows Mukai's appendix to \cite{kondo} (N.B. there are some small typos in loc. cit.: the group $A_{4,4}$ has order $288$, instead of $384$). For reader's convenience, we recall that the group $M_{20}$ is isomorphic to $2^4:A_5$, the group $M_9$ is isomorphic to $3^2:Q_8$, the group $T_{48}$ is isomorphic to $L_2(3)$. The operator $*$ is the central product. Concretely, the group $Q_8 * Q_8$ is the quotient of $Q_8\times Q_8$ by the diagonal of corresponding to the center of $Q_8$, and it is isomorphic to an extraspecial group $2^{1+4}$.
\end{notation}

Below, we discuss the maximal rank cases for degree $2$ and $6$ $K3$ surfaces as those are connected to cubic fourfolds (as discussed, they correspond to ``fake cubics'', i.e., the Hassett divisors $\calC_2$ and $\calC_6$). The cases of degree $4$ $K3$ surfaces and elliptic $K3$ surfaces are equally interesting, but less relevant to the core analysis in this paper. We point out however the classification  on projective automorphisms of quartic $K3$ surfaces in \cite{MPK}, and the work \cite{garbagnati2013elliptic} on automorphisms of elliptic $K3$ surfaces.

\subsection{The degree $2$ $K3$ case}
The maximal symplectic cases for degree $2$ $K3$ surfaces (analogue to Thm. \ref{proposition: maximal case} for cubics) are listed below. 
\begin{thm}
\label{theorem: degree 2}
For a degree 2 $K3$ surface $Y$ with a symplectic action of a finite group $G$. Suppose $\rank (S_G(Y))=19$, then $(G,S_G(Y))$ is one of numbers $3, 7, 9, 11$ in Table \eqref{table: MSS aut. of K3}. In particular, the group $G$ can be $A_6$ (see \eqref{eq_wiman}), $L_2(7)$ (see \eqref{eq_klein}), $M_9$ (see \eqref{eq_mukai}), or $T_{48}$ (see \eqref{eq_t48}).
\end{thm}
\begin{proof}
Given $(G,S,K)$ from Table \eqref{table: MSS aut. of K3}. By Theorem \ref{theorem: criterion lattice from low degree k3}, we need to check whether there exists embedding of $S\oplus E_7$ into $\BB$, such that the image of $S$ is primitive. We have that $q_{E_7}=-q_{A_1}=2_7^{+1}$. For numbers $1, 2, 4, 5, 6$, the lattice $S\oplus E_7$ has no nontrivial saturation in which $S$ is primitive, and $l_2(S\oplus E_7)\ge 3$. For number 10, we have $l_3(S\oplus E_7)=l_3(K)=3$. Therefore, in these cases, there are no embedding of $S\oplus E_7$ into $\BB$ such that the image of $S$ is primitive. We next check for the other cases one by one.
\begin{enumerate}[(1)]
\item For number 3 in Table \eqref{table: MSS aut. of K3}, the discriminant form of $S\oplus E_7$ is $4_3^{-1}3^{+2}5^{+1}\oplus 2_7^{+1}$, and there is no nontrivial saturation of $S\oplus E_7$. Since $2\times 4\times 5=40\equiv 1$ (mod 3), there exists a unique negative rank $2$ even lattice $T=-(12^0 30)$ with discriminant form $2_1^{+1}4_5^{-1}3^{+2}5^{+1}$. Thus there exists a primitive embedding of $S\oplus E_7$ into $\BB$ with orthogonal complement $T$.
\item For number 7 in Table \eqref{table: MSS aut. of K3}, the discriminant form of $S\oplus E_7$ is $4_7^{+1}7^{+2}\oplus 2_7^{+1}$, and there is no nontrivial saturation of $S\oplus E_7$. Since $2\times 4=8$ is a square in $\ZZ_7$, there exists a negative rank $2$ even lattice $T$ with discriminant form $2_1^{+1}4_1^{+1}7^{+2}$. Thus there exists a primitive embedding of $S\oplus E_7$ into $\BB$ with orthogonal complement $T$.
\item For number 8 in Table \eqref{table: MSS aut. of K3}, the discriminant form of $S\oplus E_7$ is $4_5^{-1}3^{-1}5^{-2}\oplus 2_7^{+1}$, and there is no nontrivial saturation of $S\oplus E_7$ in which $S$ is primitive. Since $2\times 4\times 3=24\equiv -1$ is a square in $\ZZ_5$, there is no negative rank $2$ even lattice with discriminant form $2_1^{+1}4_3^{-1}3^{+1}5^{-2}$. Thus there is no embedding of $S\oplus E_7$ into $\BB$ with the image of $S$ primitive.
\item For number $9$ in Table \eqref{table: MSS aut. of K3}, the discriminant form of $S\oplus E_7$ is $2_5^{+3}3^{+1}9^{+1}\oplus 2_7^{+1}$, and there is uniquely a nontrivial saturation $\widetilde{S}$ of $S\oplus E_7$ in which $S$ is primitive. The discriminant form of $\widetilde{S}$ is $2_4^{+2}3^{+1}9^{+1}$. Since $2\times 2=4$ is a square in $\ZZ_3$, there exists a negative rank $2$ even lattice $T$ with discriminant form $q_{\widetilde{S}(-1)}=2_4^{+2}3^{-1}9^{-1}$. Thus there exists a primitive embedding of $\widetilde{S}$ into $\BB$ with orthogonal complement $T$.
\item For number 11 in Table \eqref{table: MSS aut. of K3}, the discriminant form of $S\oplus E_7$ is $2_1^{+1}8_{\II}^{-2}3^{-1}\oplus 2_7^{+1}$, and there is uniquely a nontrivial saturation $\widetilde{S}$ of $S\oplus E_7$ in which $S$ is primitive. The discriminant form of $\widetilde{S}$ is $8_{\II}^{-2}3^{-1}$. There exists a negative rank $2$ even lattice $T$ with discriminant form $8_{\II}^{-2}3^{+1}$. Thus there exists a primitive embedding of $\widetilde{S}$ into $\BB$ with orthogonal complement $T$.
\end{enumerate}
The claim follows.
\end{proof}

We discuss the geometric realizations for those maximal symplectic groups. The double cover of $\PP^2$ branched along a sextic curve is a degree $2$ $K3$. If a group acts on a plane sextic curve, it also acts on the corresponding degree 2 $K3$ surface. A classification of automorphism groups of plane sextic curves can be deduced from \cite[Thm. 2.1]{harui2014}. It is discovered by Wiman \cite{wiman1896a6} that the sextic curve 
\begin{equation}\label{eq_wiman}
V(10 x_1^3 x_2^3+9 x_3(x_1^5+x_2^5)-45x_1^2 x_2^2 x_3^2-135x_1 x_2 x_3^4+27 x_3^6)
\end{equation}
has an action by $A_6$. The corresponding degree $2$ $K3$ surface also admits an action by $A_6$, which must be symplectic since $A_6$ is simple. In \cite{doi2000sextic} the uniqueness of such a sextic curve (with action by $A_6$) is proved. 

The Klein sextic curve 
\begin{equation}\label{eq_klein}
V(x_1^5 x_2+ x_2^5 x_3+ x_3^5 x_1)
\end{equation}
has automorphism group $L_2(7)$. Therefore, the symplectic automorphism group of the corresponding degree $2$ $K3$ surface is $L_2(7)$.

Another smooth plane sextic with large symmetry (see Remark 2.4 in \cite{harui2014}) is
\begin{equation}\label{eq_mukai}
V(x_1^6+x_2^6+x_3^6-10(x_1^3 x_2^3+ x_2^3 x_3^3+ x_3^3 x_1^3))
\end{equation}
which has automorphism group equal to the Hessian group $H_{216}$ of order $216$ (this group can be represented as the affine special linear group $\ASL(2, \FF_3)$, or as the projective unitary group $\PU(3, \FF_2)$). Actually the degree $2$ $K3$ surface corresponding to this sextic curve has symplectic automorphism group equal to $M_9\cong \PSU(3, \FF_2)$ (cf. \cite[Example 0.4]{mukaiaut}).

Finally, the group $T_{48}$ is realized by the double cover of $\PP^2$ with branch curve 
\begin{equation}\label{eq_t48}
V(x_1^5 x_2+ x_2^5 x_1+x_3^6),
\end{equation} 
(cf. \cite[Example 0.4]{mukaiaut}).

\subsection{The degree $6$ $K3$ case} The maximal cases in the degree $6$ case are listed below.

\begin{thm}
\label{theorem: degree 6}
For a degree 6 $K3$ surface $Y$ with a symplectic action of a finite group $G$. Suppose $\rank (S_G(Y))=19$, then $(G,S_G(Y))$ is one of numbers $3, 8, 10$ in Table \eqref{table: MSS aut. of K3}. In particular, the group $G$ can be $A_6$ (see \eqref{eq_a6}), $S_5$ (see \eqref{eq_s5}), or $N_{72}$ (see \eqref{eq_d8}).
\end{thm}
\begin{proof}
Given $(G,S,K)$ in Table \eqref{table: MSS aut. of K3}. By Theorem \ref{theorem: criterion lattice from low degree k3}, we need to check whether there exists embedding of $S\oplus E_6\oplus A_1$ into $\BB$, such that the image of $S$ is primitive. We have that $q_{E_6\oplus A_1}=-q_{A_2}\oplus q_{A_1}=2_1^{+1}3^{+1}$. For cases with numbers $1, 2, 4, 5, 6, 11$, we have $l_2(\widetilde{S})\ge 3$ for any saturation $\widetilde{S}$ of $S\oplus E_6\oplus A_1$ in which $S$ is primitive. For case with number 9, we have $l_3(\widetilde{S})\ge 3$ for any saturation $\widetilde{S}$ of $S\oplus E_6\oplus A_1$. Therefore, for those cases we can not embed $S\oplus E_6\oplus A_1$ into $\BB$ with image of $S$ primitive. We consider the other cases one by one.
\begin{enumerate}[(1)]
\item For number 3 in Table \eqref{table: MSS aut. of K3}, the discriminant form of $S\oplus E_6\oplus A_1$ is $4_3^{-1}3^{+2}5^{+1}\oplus 2_1^{+1}3^{+1}$. We have a nontrivial saturation $\widetilde{S}$ of $S\oplus E_6\oplus A_1$ with discriminant form $2_1^{+1}4_3^{-1}3^{-1}5^{+1}$. There exists negative rank $2$ even lattice $T$ with discriminant form $2_7^{+1}4_5^{-1}3^{+1}5^{+1}$. Thus there exists a primitive embedding of $\widetilde{S}$ into $\BB$ with orthogonal complement $T$.
\item For number 7 in Table \eqref{table: MSS aut. of K3}, the discriminant form of $S\oplus E_6\oplus A_1$ is $4_7^{+1}7^{+2}\oplus 2_1^{+1}3^{+1}$, and there is no nontrivial saturation of $S\oplus E_6\oplus A_1$. Since $2\times 4\times 3=24\equiv 3$ (mod 7), and $3$ is not a square in $\ZZ_7$, there is no negative rank $2$ even lattice with discriminant form $2_7^{+1}4_1^{+1}3^{-1}7^{+2}$. Thus there is no embedding of $S\oplus E_6\oplus A_1$ into $\BB$.
\item For number 8 in Table \eqref{table: MSS aut. of K3}, the discriminant form of $S\oplus E_6\oplus A_1$ is $4_5^{-1}3^{-1}5^{-2}\oplus 2_1^{+1}3^{+1}$. We have a nontrivial saturation $\widetilde{S}$ of $S\oplus E_6\oplus A_1$ with discriminant form $2_1^{+1}4_5^{-1}5^{-2}$, in which $S$ is primitive. Since $2\times 4=8$ is not a square in $\ZZ_5$, there exists a negative rank $2$ even lattice $T$ with discriminant form $2_7^{+1}4_3^{-1}5^{-2}$. Thus there exists a primitive embedding of $\widetilde{S}$ into $\BB$ with orthogonal complement $T$.
\item For number 10 in Table \eqref{table: MSS aut. of K3}, the discriminant form of $S\oplus E_6\oplus A_1$ is $4_7^{+1}3^{+2}9^{+1}\oplus 2_1^{+1}3^{+1}$. We have a nontrivial saturation $\widetilde{S}$ of $S\oplus E_6\oplus A_1$ with discriminant form $2_1^{+1}4_7^{+1}3^{-1}9^{+1}$, in which $S$ is primitive. Since $2\times 4=8$ is not a square in $\ZZ_3$, there exists a negative rank $2$ even lattice $T$ with discriminant form $2_7^{+1}4_1^{+1}3^{+1}9^{-1}$. Thus there exists a primitive embedding of $\widetilde{S}$ into $\BB$ with orthogonal complement $T$.
\end{enumerate}
The theorem follows.
\end{proof}

The geometric realization of all these three groups can be found in Mukai \cite[Example 0.4]{mukaiaut}. The group $A_6$ is the symplectic automorphism group of 
\begin{equation}\label{eq_a6}
Y=V(x_1+\cdots+x_6)\cap V(x_1^2+\cdots+x_6^2)\cap V(x_1^3+\cdots+x_6^3)
\end{equation}
(presented as a diagonal hyperplane section in $\PP^5$). 
Similarly, the group $S_5$ is the symplectic automorphism group of
\begin{equation}\label{eq_s5}
V(x_1+\cdots+x_5)\cap V(x_1^2+\cdots+x_6^2)\cap V(x_1^3+\cdots +x_5^3)
\end{equation}
(here the symplectic action is given by $g\in S_5$ acts by permutations on $(x_1, \dots, x_5)$ and by $x_5\to \mathrm{sgn}(g) x_5$; see \cite[p. 188]{mukaiaut}). 
The group $N_{72}$ is the symplectic automorphism group of
\begin{equation}\label{eq_d8}
V(x_1^3+x_2^3+x_3^3+x_4^3)\cap V(x_1 x_2+x_3 x_4+x_5^2).
\end{equation}

\subsection{Uniqueness for $K3$ surfaces} 
While we don't investigate the uniqueness question here (i.e., analogues of Theorem \ref{theorem: maximal uniqueness}), we point out that Hashimoto \cite[Main Theorem]{hashimoto2012K3} proved that for three of Mukai's maximal cases (specifically (3), (7), and (8), corresponding to groups $A_6$, $L_2(7)$, and $S_5$)  there are exactly two primitive sublattices (up to conjugate) of $\Lambda_{K3}(-1)$ isomorphic to $S$ (where, as before, $S$ is the covariant lattice). Each of these cases has at least one realization for either a degree $2$ or $6$ $K3$ surface (see \eqref{eq_a6}, \eqref{eq_klein}, and \eqref{eq_s5} below). As Hashimoto works in the unpolarized case, the moduli of $K3$ surfaces with symplectic automorphism groups in the above three cases has two connected component, both of dimension $1$. The group $A_6$ is of special interest since it occurs for degree $2$ and degree $6$ cases (see \eqref{eq_wiman} and \eqref{eq_a6}). Interestingly, the two cases are in two different components. 

\begin{prop}
The embeddings of $S$ into $\Lambda_{K3}(-1)$ given by the two geometric realizations \eqref{eq_wiman} and \eqref{eq_a6} (degree $2$ and degree $6$) have different orthogonal complements. In particular, these two $K3$ surfaces belong to different connected components of the moduli space of $K3$ surfaces with symplectic automorphism group $A_6$.
\end{prop}
\begin{proof}
Let $Y_1$, $Y_2$ be the degree $2$ and degree $6$ $K3$ surfaces with $A_6$ symplectic action respectively. Then the orthogonal complement of $S\cong S_{A_6}(Y_1)\hookrightarrow H^2(Y_1, \ZZ)(-1)$ contains a vector with self-intersection $-2$, while the orthogonal complement of $S\cong S_{A_6}(Y_2)\hookrightarrow H^2(Y_2, \ZZ)(-1)$ contains a vector with self-intersection $-6$. (Note that in our conventions we are scaling the cohomology by $-1$, making the polarization a negative vector. Furthermore, in these maximal cases, $S^\perp$ is negative definite of rank $3$.) On the other hand, from Hashimoto \cite[Table 10.3, item 79]{hashimoto2012K3}, the orthogonal complement $S^\perp$ of an embedding of $S$ into $\Lambda_{K3}(-1)$ can be either
\begin{equation*}
\left(
\begin{array}{ccc}
-2 & -1 & 0\\
-1 & -8 & 0\\
0 & 0 & -12
\end{array}
\right),
\end{equation*}
which contains $(-2)$-vector but does not contain any $(-6)$-vector, or
\begin{equation*}
\left(
\begin{array}{ccc}
-6 & 0 & -3\\
0 & -6 & -3\\
-3 & -3 & -8
\end{array}
\right),
\end{equation*}
which contains $(-6)$-vector but does not contain any $(-2)$-vector. The claim follows.
\end{proof}

\subsection{A geometric relation to cubic fourfolds}
\label{subsec_k3cubic}
Notice that the maximal symplectic automorphism groups for degree $2$ (see Theorem \ref{theorem: degree 2}) and degree $6$ (see Theorem \ref{theorem: degree 6}) also appear in case $\rank(S)=19$ in Theorem \ref{theorem: main}. This is not a coincidence. The following proposition explains the geometry behind this phenomena. 

\begin{prop}\label{prop_smooth_fake}
Let $(G, S)$ be a Leech pair satisfying conditions in Theorem \ref{theorem: criterion lattice from low degree k3} for degree $d=2$ or $6$, then $(G, S)$ is one of the $34$ Leech pairs we obtain in our main Theorem \ref{theorem: main}. Especially, the corresponding moduli of cubic fourfolds has dimension one more than that of the degree $2$ or $6$ $K3$ surfaces. 
\end{prop}

\begin{proof}
Let $(G, S)$ be a Leech pair such that there is an embedding of $S\oplus R_d$ into $\BB$ with the image of $S$ primitive. Here $R_d=E_7$ if $d=2$ and $R_d= E_6\oplus A_1$ if $d=6$. Notice that in both situations we have a natural embedding $E_6\hookrightarrow R_d$. Thus we have an embedding $S\oplus E_6 \hookrightarrow S\oplus R_d\hookrightarrow \BB$ with the image of $S$ in $\BB$ primitive. Therefore, the Leech pair $(G, S)$ arises from symplectic actions of $G$ on certain smooth cubic fourfolds. The dimension of the moduli space of such cubic fourfolds is $20-\rank(S)$, while the dimension of degree $d$ $K3$ surfaces with the corresponding symplectic action by $G$ is $19-\rank(S)$.
\end{proof}

\begin{rmk}\label{rem_smooth_fake}
The above proposition tells that if we have a family of fake cubic fourfolds with symplectic action by a finite group $G$, then we can smooth the fake cubic fourfolds to smooth ones, preserving the action of $G$. What we obtain is a family (of one more dimension) of cubic fourfolds with symplectic action of $G$ such that the generic fibers are smooth.
\end{rmk}

Let us briefly discuss the geometry behind Proposition \ref{prop_smooth_fake} (and Remark \ref{rem_smooth_fake}). For simplicity, we restrict to the case of nodal cubic fourfolds (parametrized by the Hassett divisor $\calC_6$). A singular cubic fourfold can be written as 
\begin{equation}\label{eq_sing_cubic}
X_0=V(f_2(x_1,\dots,x_5) x_6+f_3(x_1,\dots,x_5))\subset \PP^5
\end{equation}
for some homogeneous polynomials $f_2$, $f_3$ of degree $2$ and $3$ respectively. Note that the equation above singles out the singular point $p=(0,\dots,1)\in X$. The linear projection from $p$ 
$$\pi:X_0\dashrightarrow \PP^4$$
is a birational equivalence. The inverse map $\pi^{-1}:\PP^4\dashrightarrow X_0$ has indeterminacy locus the degree $6$ $K3$ surface
$$Y=V(f_2(x_1,\dots, x_5),f_3(x_1,\dots,x_5))\subset \PP^4.$$
More precisely, assuming $Y$ is smooth, $X_0$ has a unique singular point $p$ which is either of type $A_1$ (if $V(f_2)$ is smooth) or type $A_2$ (if $V(f_2)$ is singular), and it holds
$$\widetilde{X_0}=\mathrm{Bl}_p X_0\cong \mathrm{Bl}_Y \PP^4.$$
This establishes a Hodge correspondence (essentially an identification) between the Hodge structure on $H^4(X_0)$ (still pure) and $H^2(Y)(-1)$.  In terms of automorphism, note that since the polarized automorphisms of $Y$ are induced from projective transformations, i.e., $G=\Aut(Y)_{pol}\subset \PGL(5)$, $G$ acts by automorphisms on $\widetilde X_0$. The group $G$ preserves the quadric $V(f_2)\subset \PP^5$ and its strict transform $E$ in $\widetilde X_0=\mathrm{Bl}_Y \PP^4$. But then $E$ is precisely the exceptional divisor of $\widetilde X_0=\mathrm{Bl}_p X_0\to X_0$. We conclude that $G$ acts on $X_0$ by automorphisms preserving the singular point $p$. 

Assuming that the equations $f_2$ and $f_3$ of $Y=V(f_2,f_3)$ can be chosen to be invariant with respect to $G$ (in general some character of $G$ might be involved), then  the (pencil of) cubic fourfolds 
$$X_t=V\left((f_2x_6+f_3)+t x_6^3\right)\subset \PP^5$$
admit $G$ as a group of automorphisms, with $G$ acting trivially on $x_6$. For general $t\in \PP^1$, the above cubic is smooth. This allows us to lift the equations for maximal symmetric $K3$ surfaces of degree $2$ and $6$ to $1$-parameter families of cubic fourfolds with large symmetry group (producing examples for most of the cases of Theorem \ref{theorem: main}(8)). The simplest example of such a lifting is the $A_6$ case \eqref{eq_a6}. Specifically, the degree $6$ $K3$ surface is 
$$Y=V(x_1^2+\dots+x_5^2+(x_1+\dots+x_5)^2, x_1^3+\dots+x_5^3-(x_1+\dots+x_5)^3).$$
It can be lifted to the $1$-parameter family of cubics $X_t=V(F_t)$ with $A_6$ symmetry, where
\begin{equation}\label{eq_a6lift}
F_t=x_1^3+\dots+x_5^3-(x_1+\dots+x_5)^3+x_6\left(x_1^2+\dots+x_5^2+(x_1+\dots+x_5)^2\right)+t x_6^3.
\end{equation}
More symmetrically, we can write 
$$X_t=V(x_0+\dots+x_5, x_0^3+\dots+x_5^3+x_6(x_0^2+\dots+x_5^2)+t x_6^3).$$
In this particular case, the symplectic condition is automatic as $A_6$ is a simple group (see also \S\ref{subsection: general discussion full automorphism groups} below). 
\section{Some remarks on the full automorphism groups for smooth cubic fourfolds}
\label{section: non-symplectic}
In this section we discuss about automorphisms and automorphism groups of smooth cubic fourfolds in general (i.e. without the symplectic assumption). We first discuss some general structure results in \S \ref{subsection: general discussion full automorphism groups} (the same arguments apply to $K3$ surfaces or hyper-K\"ahler manifolds). In \S\ref{section: order non-symplectic}, we  obtain some estimate on ``how non-symplectic'' the automorphism group of a cubic fourfold can be. Finally, in \S\ref{subsection: maximal cases}, we give some arithmetic conditions for smooth cubic fourfolds to admit non-symplectic automorphisms of order $2$, $3$ or $4$, and then use this to find the full automorphism groups for smooth cubic fourfolds with $\rank(S)=20$.

\subsection{Basic structures of the full automorphism groups}
\label{subsection: general discussion full automorphism groups}
Let $X$ be a smooth cubic fourfold, and $G=\Aut(X)$ the automorphism group. The induced action of $G$ on $H^{3,1}(X)$ gives a character $\chi\colon G\longrightarrow \CC^{\times}$, with kernel the symplectic automorphism group $G_s=\Ker(\chi)$. The image of $\chi$ is a cyclic group which we denoted by $\overline{G}$. We have the following short exact sequence of finite groups:
\begin{equation*}
1\longrightarrow G_s\longrightarrow G\longrightarrow \overline{G}\longrightarrow 1.
\end{equation*}

As before, the symplectic part $G_s\subset \Aut(X)$ induces a Leech pair $(G_s,S)$. 
Denote by $T(X)\subset H^4(X, \ZZ)$ the transcendental lattice of $X$. Note 
$$T(X)\subset H^4(X, \ZZ)_{prim}^{G_s}=S^\perp_{\Lambda_0}.$$ The induced action of the full automorphism group $G$ on $H^4(X,\ZZ)$ (or $H^4(X,\ZZ)_{prim}$) preserves the algebraic and transcendental lattices. Since $G_s$ acts trivially on $T(X)$, the action of $G$ on $T(X)$ factors through an action of $\overline{G}$ on $T(X)$. Clearly, the action of $\overline{G}$ preserves the Hodge structure on $T(X)$, and in particular it preserves the subspace $H^{3,1}\cong \CC\subset T(X)$. Choosing $\sigma$ a generator of $H^{3,1}$ (i.e., $\sigma$ is the class of a $(3,1)$ form on $X$), we see that $\overline G$ acts on $\sigma$ by roots of unity, i.e., if $\xi \in \overline G$ is a generator then 
$$\xi.\sigma=\zeta\sigma$$
for some root of unity $\zeta(\neq 1)\in U(1)\subset \CC^*$. We then note:

\begin{lem}
The induced action of $\overline{G}$ on $T(X)$ is faithful and has no non-zero fixed vectors.
\end{lem}
\begin{proof}
Suppose not faithful, then there exists $g\in G\setminus G_s$ such that the induced action of $g$ on $H^4(X, \ZZ)$ leaves $T(X)$ invariant. But this implies that $g$ fixes $H^{3,1}(X)$, which is a contradiction to the assumption $g\notin G_s$. Suppose there is a non-zero vector $v\in T(X)$ fixed by $\overline{G}$. Then, denoting as above by $\xi$ a generator of $\overline G$, and $\sigma$ a generator of $H^{3,1}$, we have 
$$\langle \sigma, v\rangle=\langle \xi.\sigma, \xi.v\rangle=\langle \zeta\sigma,v\rangle=\zeta\langle \sigma,v\rangle,$$
which forces $\langle \sigma,v\rangle =0$. Thus, $v\in H^{2,2}\cap H^4(X, \ZZ)$, a contradiction. (Alternatively, the Hodge structure on $T(X)$ is irreducible. The fixed locus of $\overline G$ is a sub-Hodge structure, and thus can only be trivial.)
\end{proof}

Denote by $n$ the order of $\overline{G}$ (i.e., $\overline{G}\cong \ZZ/n$). Standard algebra leads to the following:

\begin{cor}
\label{corollary: euler of n}
We have $\varphi(n)\big{|} \rank (T(X))$. Here $\varphi$ is the Euler function.
\end{cor}
\begin{proof}
Let $\xi$ be a generator of $\overline{G}$, and $\zeta$ a primitive $n$-root of unity such that $\xi.\sigma=\zeta\sigma$ for $\sigma\in H^{3,1}(X)$. The arguments of the previous lemma, easily give that all the eigenvalues of $\xi$ on $T(X)$ are primitive $n$-roots of unity. The characteristic polynomial $p_\xi$  of $\xi$ as an isomorphism of $T(X)$ is rational. It follows that $p_\xi$ is a power of the cyclotomic polynomial. The claim follows.
\end{proof}

\subsection{Order of the non-symplectic part}
\label{section: order non-symplectic}
The list of smooth cubic fourfolds with prime order automorphism is known. Specifically, according to  \cite[Theorem 3.8]{gonzalez} there are $13$ irreducible families\footnote{The case $\mathcal{F}_5^2$ in \cite[Theorem 3.8]{gonzalez} should be excluded, as the corresponding family contains only singular cubic fourfolds. This was pointed out in \cite{boissiere2016classification}.}  of cubics with a prime order automorphism. In particular,
\begin{prop}
\label{proposition: order non-symp}
A prime factor of the order of the automorphism group of a smooth cubic fourfold can only be $2$, $3$, $5$, $7$, or $11$. A non-symplectic prime-order automorphism of a smooth cubic fourfold can have order $2$ or $3$.
\end{prop}
\begin{proof}
The list of prime orders is a consequence of \cite[Theorem 3.8]{gonzalez}. The second part follows by noticing that $7$ of the $13$ cases were already identified in Theorem \ref{theorem: lie fu} as the symplectic cases (see also Remark \ref{rmk_sympl_cond}). The symplectic cases cover all the cases involving the primes $5$, $7$, and $11$. The claim follows. 
\end{proof}

By Proposition  \ref{proposition: order non-symp}, the order of $\overline{G}$ has only prime  factors $2$ or $3$. Thus, we can write $n(=|\overline{G}|)=2^k 3^l$. From Corollary \ref{corollary: euler of n} and the fact $T(X)\subset S^\perp_{\Lambda_0}$ we get:
\begin{equation}
\label{equation: euler of n less that rank of T}
\varphi(n)=\varphi(2^k 3^l)\le 22-\rank(S).
\end{equation}

As mentioned the induced action of $G$ on $H^4(X,\ZZ)$ preserves the algebraic and transcendental lattices. In fact $G$ preserves also the covariant lattice $S(=S_{G_s}(X))$.

\begin{lem}
\label{lemma: non-symplectic preserve S}
The induced action of $G$ on $H^4(X,\ZZ)$ leaves $S$ stable.
\end{lem}
\begin{proof}
The subgroup $G_s$ is normal in $G=\Aut(X)$. Thus for any $g\in G$, $gG_sg^{-1}=G_s$. By definition, $S$ is the orthogonal complement of the invariant lattice $\Lambda^{G_s}$. Clearly $G_s=gG_sg^{-1}$ leaves every vector in $g\Lambda^{G_s}$ invariant. It follows that $g\Lambda^{G_s}=\Lambda^{G_s}$. By taking orthogonal complements, we get that $g$ leaves $S$ stable.
\end{proof}

The action of $G$ on $S$ induces a homomorphism $\pi\colon G\longrightarrow \Aut(q_S)$. Since $G_s$ acts trivially on $q_S$, the homomorphism $\pi$ descends to a morphism $\pi\colon \overline{G}\longrightarrow \Aut(q_S)$.
\begin{prop}
\label{theorem: embedding of G/G_s into Aut(q_S)}
When $\rank(S)\ge 13$, the homomorphism $\pi\colon \overline{G}\longrightarrow\Aut(q_S)$ is injective.
\end{prop}
\begin{proof}
Suppose $g\in G\setminus G_s$ acts trivially on $q_S$. Thus, the action of $g$ on $S$ is by isometries preserving the discriminant. As previously discussed, any such isometry of $S$ can be lifted to a symplectic automorphism of $X$. Thus, there exists $h\in G_s$, such that the restrictions of $g$ and $h$ to $S$ are the same. Replacing $g$ by $gh^{-1}$, we can assume (wlog) that $g$ acts trivially on $S$. 

Replacing $g$ by a power $g^k$, we can further assume that $g$ has prime order. By Proposition  \ref{proposition: order non-symp}, we can assume that $g$ is either of order $2$ or $3$.

By the classification in \cite{gonzalez} and the discussion in \cite[\S6]{yu2018moduli}, there are two conjugacy classes of non-symplectic involutions, with corresponding moduli spaces arithmetic quotient of type $\IV$ domains having dimensions $10$ and $14$ (N.B. the $14$ dimensional case is discussed in detail in \cite{laza2017moduli}). In particular, the invariant sublattice of $\Lambda_0$ (which contains $S$) is of rank $12$ or $8$ respectively, contradicting $\rank (S)\ge 13$. The order $3$ case is similar. Namely, there are $4$ conjugacy classes of 
of non-symplectic order three automorphisms, with corresponding moduli spaces arithmetic ball quotients of dimensions $4$, $6$, $7$ and $10$ (N.B. the $10$-dimensional case is \cite{allcock2011moduli}). Again, the automorphism $g$ can not leave a  sublattice of rank at least $13$ of $\Lambda_0$ invariant, a contradiction.
\end{proof}

The proposition above is very useful in the cases where $S$ is of large rank, or equivalently $G_s$ is relatively large; this is the case of interest in this paper. In fact, note that most of the cases in Theorem \ref{theorem: main} satisfy $\rank(S)\ge 13$. It would be interesting to classify the possible orders $n=2^k3^l$ of non-symplectic automorphisms on a cubic fourfold, especially we do not know what is the largest possible such $n$ (compare \eqref{equation: euler of n less that rank of T}). These cases will have essentially trivial symplectic automorphism group, thus they should be handled by different methods. 

\begin{rmk}\label{rem_covariant_anti}
A major difference between the lattice theoretic methods in the symplectic and anti-symplectic cases is that  the covariant lattice $N$ for an anti-symplectic automorphism contains the transcendental lattice $T(X)$, and thus (except the case $\rank(T(X))=2$) $N$ is indefinite (in particular, $O(N)$ is typically infinite).
\end{rmk}

\subsection{Maximal cases}
\label{subsection: maximal cases}
We conclude our discussion of the automorphism groups of cubic fourfolds, with a discussion of the full automorphism group for the $8$ maximal cases (with respect to symplectic automorphisms) identified in Theorem \ref{theorem: maximal uniqueness}. These are the most interesting cases from the perspective of this paper, and they are particularly suitable to classification (compare Prop. \ref{lemma: non-symplectic preserve S} and Rem. \ref{rem_covariant_anti}, and note $\rank(S)=20$, $\rank (T)=2$).

Since we assume $\rank(S)=20$, the transcendental lattice $T(X)$ is the orthogonal complement of $S(-1)$ in $H^4(X, \ZZ)_{prim}$ and has rank $2$. From Equation \eqref{equation: euler of n less that rank of T} we get that the possible orders for the non-symplectic part $\overline G$ are $n=2$, $3$, $4$, or $6$. We discuss first the case of anti-symplectic involutions. 

An involution on a cubic $X$ can be diagonalized to one of the following three types: $\diag(-1,1,1,1,1,1)$, $\diag(-1,-1,1,1,1,1)$, and $\diag(-1,-1,-1,1,1,1)$ (see also  \cite{gonzalez}). The involution $\diag(-1,-1,1,1,1,1)$ is symplectic, while the other two are anti-symplectic. 

\begin{rmk}[Eckardt points] An essential ingredient in the geometric classification of the automorphism groups of cubic surfaces are the {\it Eckardt points} (see \cite{Eckardt1876} and \cite{segre1942cubic}). The Eckardt points can be defined for cubics of any dimension (e.g. see  \cite{laza2017moduli}). From the perspective of automorphisms, a smooth cubic $n$-fold $V\subset \PP^{n+1}$ has an Eckardt point iff it is invariant with respect to an involution $\iota$ that fixes a hyperplane (thus of type $\diag(-1,1,\cdots, 1)$); the Eckardt point is the isolated fixed point of $\iota$. Explicitly, $V$ is defined by cubic polynomial $F(x_2,\cdots, x_{n+2})+x_1^2 L(x_2, \cdots, x_{n+2})$, where $\deg(F)=3$ and $\deg(L)=1$; $[1:0:\cdots:0:0]\in V$ is an Eckardt point. We refer to \cite{laza2017moduli} for further details. 
\end{rmk}

We have the following necessary condition for a smooth cubic fourfold with maximal symplectic symmetry to admit an anti-symplectic involution.

\begin{prop}
\label{proposition: rk(S)=20 with antisymplectic involution}
Let $X$ be a smooth cubic fourfold with $\rank(S)=20$. Suppose there exists an anti-symplectic involution on $X$, then the composition of $S\oplus E_6\hookrightarrow H^4_0(X,\ZZ)\oplus E_6\hookrightarrow \BB$ is not primitive.
\end{prop}
\begin{proof}
By Lemma  \ref{lemma: non-symplectic preserve S}, the induced involution $\iota^*$ on $H^4_0(X,\ZZ)$ preserves $S=S_{G_s}(X)$. Since $\iota^*$ equals to $-id$ on the orthogonal complement of $S$ in $H^4(X,\ZZ)$, the invariant sublattice $M=H^4_0(X,\ZZ)^{\iota^*}$ of $H^4_0(X,\ZZ)$ is contained in $S$. Suppose $j\colon S\oplus E_6\hookrightarrow \BB$ is primitive, then the inclusion $j\colon M\oplus E_6\hookrightarrow \BB$ is also primitive.

On the other hand, the involution $\iota^*$ on $H^4_0(X,\ZZ)$ extends to an involution on $\BB$, with restriction to $E_6$ trivial. The invariant sublattice of $\BB$ under the action of $\iota^*$ is $M\oplus E_6$. This is a contradiction, because the invariant sublattice (in a unimodular lattice) of an involution has $2$-group as its discriminant group, while $|A_{E_6}|=3$.
\end{proof}

In particular, this allows us to distinguish the two cases of Theorem \ref{theorem: maximal uniqueness}(2) with symplectic automorphism group $A_7$. Namely, we note that cubic fourfold with $A_7$ automorphisms identified by H\"ohn--Mason has an extra symplectic involution, while the other can not have.
\begin{cor}
\label{corollary: diagonal}
Let  $X=V (x_1^3+x_2^3+x_3^3+x_4^3+x_5^3+x_6^3-(x_1+x_2+x_3+x_4+x_5+x_6)^3)$ with symplectic automorphism group $A_7$ (cf. \cite[Table 2]{HM2}). Let  $S$ be the covariant sublattice of $H^4(X,\ZZ)$ with respect to the induced action by $A_7$. Then the orthogonal complement $T$ of $S$ in $H^4_0(X,\ZZ)$ is $-(2^1 18)$, and $q_T=5^{+1} 7^{+1}$.
\end{cor}
\begin{proof}
This is the (Clebsch) diagonal cubic, and thus it has $S_7$ automorphisms. Obviously, $\overline G=S_7/A_7\cong \ZZ/2$ (explicitly exchanging $x_1, x_2$ is an anti-symplectic involution of $X$).  By Proposition  \ref{proposition: rk(S)=20 with antisymplectic involution}, the inclusion $j\colon S\oplus E_6\hookrightarrow \BB$ is not primitive. From the proof of Theorem  \ref{theorem: maximal uniqueness}, we conclude $T=-(2^1 18)$ and $q_T=5^{+1} 7^{+1}$.
\end{proof}

Similarly, we get: 
\begin{cor}
\label{corollary: involution}
The cubic fourfold $X^2(A_7)$, and those with symplectic automorphism groups $G_s=L_2(11)$ and $M_{10}$, have no anti-symplectic involution (equivalently, order of $\overline G$ is odd). 
\end{cor}

We now switch our attention to the case of anti-symplectic involutions of order $3$ and $4$.  The main point here is that in these cases $T(X)$ has a decomposition into two conjugate eigenspaces, and in fact it acquires the structure of a (Hermitian) lattice over the Eisenstein $\ZZ[\omega]$ or respectively Gaussian $\ZZ[i]$ integers. This fact is the starting point of multiple works by Kond\=o (e.g. \cite{DK}) and Allcock--Carlson--Toledo (e.g. \cite{allcock2011moduli}). In our situation, $T(X)$ is of rank $2$, and thus of rank $1$ as Eisenstein/Gaussian lattice. This allows us to obtain the following simple criterion for $|\overline G|$ to be a multiple of $3$ or $4$. 

\begin{lem}
\label{lemma: order 3 and 4}
Let $T$ be a positive definite rank $2$ even lattice. Then $T$ admits an automorphism of order $3$ if and only if then there exists a positive integer $a$ such that $T\cong A_2(a)$, and $T$ admits an automorphism of order $4$ if and only if there exists a positive integer $a$ such that $T\cong A_1^2(2a)$.
\end{lem}

\begin{proof}
The lattice $A_2=(1^2 1)$ admits an order $3$ automorphism, explicitly
$\left(
\begin{array}{cc}
0 & -1 \\
1 & -1
\end{array}
\right)$.
The lattice $A_1^2=(2^0 2)$ admits an order $4$ automorphism, explicitly
$\left(
\begin{array}{cc}
0 & -1 \\
1 & 0
\end{array}
\right)$. Thus we have necessity. 

Suppose $T$ admits an automorphism $\rho$ of order $3$. Choose $v\in T$ with minimal norm. Take $a$ such that $(v, v)=2a$. A nontrivial order $3$ automorphism on $T$ is fixed-point free, hence $v+\rho(v)+\rho(\rho(v))=0$. Thus 
\begin{equation*}
(v,v)=(\rho(\rho(v)), \rho(\rho(v)))=(v+\rho(v), v+\rho(v))=2(v,v)+2(v, \rho(v))
\end{equation*}
which implies that $(v, \rho(v))=-a$. We claim that $(v, -\rho(v))$ is a basis for $T$. If not, then we can find non-zero numbers $\lambda, \mu\in [\frac{1}{2}, \frac{1}{2}]$ such that $\lambda v+\mu \rho(v)\in T$. But $(\lambda v+ \mu \rho(v), \lambda v+ \mu \rho(v))=2\lambda^2+2 \lambda \mu+ 2\mu^2 <2(|\lambda|+|\mu|)^2\le 2$. This contradicts to the fact that $v$ has minimal norm. We conclude that $T\cong A_2(a)$.

Suppose $T$ admits an automorphism $\rho$ of order $4$. Since $\rho$ is rational, it has two eigenvalues $\sqrt{-1}$ and $-\sqrt{-1}$. Thus $\rho^2=-1$. Now take $v\in T$ with minimal norm $2a$. Then $(v, \rho(v))=(\rho(v), \rho(\rho(v)))=(\rho(v), -v)$, which implies that $(v, \rho(v))=0$. Similar to the order $3$ case, $(v, \rho(v))$ is a basis for $T$. We conclude that $T=A_1^2(2)$.
\end{proof}

We conclude with the computation of the non-symplectic part $\overline{G}$ for the $8$  maximal cubic fourfolds appearing in Theorem \ref{theorem: maximal uniqueness}.

\begin{prop}
\label{proposition: non-sym order}
\begin{enumerate}[(1)]
\item For the Fermat cubic fourfold $X(3^4:A_6)=V(x_1^3+x_2^3+x_3^3+x_4^3+x_5^3+x_6^3)$ the order of $\overline{G}$ is $n=6$. 
\item For $X^1(A_7)=V(x_1^3+x_2^3+x_3^3+x_4^3+x_5^3+x_6^3-(x_1+x_2+x_3+x_4+x_5+x_6)^3)$ we have $n=2$; for $X^2(A_7)$ we have $n=1$.
\item For the cubic fourfold with symplectic automorphism group $G=3^{1+4}:2.2^2$, we have $n=4$.
\item For $X^1(M_{10})$ and $X^2(M_{10})$ which have symplectic automorphism group $G=M_{10}$, we have $n=1$.
\item For $X(L_2(11))=V(x_1^3+x_2^2 x_3+x_3^2 x_4+ x_4^2 x_5+x_5^2 x_6+x_6^2 x_2)$ we have $n=3$.
\item For $X(A_{3,5})=V(x_1^3+x_2^3+x_3^2 x_4+x_4^2 x_5+x_5^2 x_6+x_6^2 x_3)$ we have $n=6$.
\end{enumerate}
\end{prop}

\begin{proof}
By Theorem \ref{theorem: maximal uniqueness}, we know the transcendental lattices of the $8$ cubic fourfolds. By Lemma \ref{lemma: order 3 and 4}, we identify the cubic fourfolds which have order $3$ or $4$ non-symplectic automorphisms. Combining with Corollary \ref{corollary: involution} we conclude the proposition.
\end{proof}

\begin{rmk}
One easy way to produce geometrically a non-symplectic automorphism of order $3$ is to consider a cubic threefold $Y=V(f(x_2,\dots, x_6))\subset \PP^4$. Then the Allcock--Carlson--Toledo \cite{allcock2011moduli} construction associates to $Y$ the cubic fourfold $X=V(f+x_1^3)\subset \PP^5$ with an order $3$ anti-symplectic automorphism ($x_1\to \omega x_1$). Of the 
six items of Proposition \ref{proposition: non-sym order}, note that items (1),  (5), and (6) are of Allcock--Carlson--Toledo type. ((1) and (6) also have an anti-symplectic involution given by switching $x_1\to x_2$.) In other words, they are obtained from highly symmetric cubic threefolds. 
\end{rmk}
Kond\=o  \cite{kondo_ns} proved that the $K3$ surface 
\begin{equation}
V(x_1^4+x_2^4+x_3^4 +x_4^4+12 x_1 x_2 x_3 x_4)
\end{equation}
has finite automorphism group of maximal possible order $3,840$. Here we conclude an analogue of Kond\=o's result. Namely, the Fermat cubic fourfold has maximal order for the automorphism group, namely $|3^4:A_6|\times |\ZZ/6|=174,960$.

\begin{cor}\label{cor_max_order}
The maximal possible order for automorphism groups of smooth cubic fourfolds is $174,960$, which is reached only by the Fermat cubic fourfold.
\end{cor}
\begin{proof}
The order of automorphism group $G$ for a smooth cubic fourfold is given by the product of $|G_s|$ and $n=|\overline{G}|$. The value of $n$ is bounded by 
\eqref{equation: euler of n less that rank of T}. The claim follows by a straightforward inspection of Theorem \ref{theorem: main}, Theorem \ref{theorem: maximal uniqueness}, and Proposition \ref{proposition: non-sym order}.
\end{proof}

\appendix
\section{Some Lattice Theory}
\label{section: collection of results in lattice theory}
We review some of the basic results of Nikulin \cite{nikulin1980integral} on lattices and discuss the standardized notation of Conway--Sloane \cite{conway1999spherepackings} (which is less familiar in algebraic geometry). The Conway--Sloane notation is quite efficient and precise, and it is used in one of our primitive references \cite{HM1}. Thus, we are using it systematically throughout the paper. This appendix aims to set up the basics as used in our paper (for further details, we refer to \cite{nikulin1980integral} and \cite{conway1999spherepackings}). 
\subsection{Lattices}
We introduce some notations and results in lattice theory. Recall that a lattice over an integral ring $R$ is a free $R$-module of finite rank together with a non-degenerate bilinear form valued in $R$. An integral lattice is a lattice over $\ZZ$. An integral lattice is called {\it even} if the norms of all elements are even numbers; called {\it odd} if it is not even. Once an ordered basis for an $R$-lattice is chosen, there is an associated symmetric Gram (or intersection) matrix. The {\it discriminant} of an $R$-lattice is the absolute value of the determinant of the intersection matrix. The discriminant does not depend on the choices of the basis. An $R$-lattice is called {\it unimodular} if its discriminant is $1$. An integral lattice $M$ can be diagonalized as $\diag(1,\cdots, 1, -1, \cdots, -1)$ over $\RR$. Let $n_1$ be the number of $1$, and $n_2$ be the number of $-1$. Then $n_1+n_2$ is the rank of $M$, and $n_1-n_2$ is called the {\it signature of $M$}.

An element $v$ in an $R$-lattice $M$ is called {\it primitive} if $v$ is non-zero and for any integer $n\ge 2$, the quotient $v/n$ is not in $M$. A sublattice $N$ of $M$ is called {\it primitive}, if there does not exists an element $v\in M\setminus N$ and a positive integer $n$ such that $nv\in N$. An embedding of lattices $N\hookrightarrow M$ is called {\it primitive} if the image is a primitive sublattice.

We use $\langle n \rangle$ to denote the rank one lattice such that the norm of the generator equals to $n$. For an $R$-lattice $M$ and $n\in \ZZ$, we define $M(n)$ to be an $R$-lattice obtained from $M$ by multiplying the bilinear form by $n$. In the category of $R$-lattices, we have naturally direct sum $\oplus$. For a Dynkin diagram $A_k, D_k$ or $E_k$, there is the associated intersection matrix, which defines a positive integral lattice, still denoted by the corresponding $A_k, D_k$ or $E_k$. We use $U$ to denote the hyperbolic lattice, given by intersection matrix
$\left(
\begin{array}{cc}
0 & 1 \\
1 & 0
\end{array}
\right)$. 

We have a classification of integral unimodular lattices (e.g. \cite[Chapter 5]{serre1973course}):

\begin{thm}[Milnor]
\label{theorem: uniqueness of indefinite unimodular lattices}
An integral unimodular indefinite odd lattice of signature $(n_1,n_2)$ is isomorphic to $\I_{n_1,n_2}=(1)^{n_1}\oplus(-1)^{n_2}$. A unimodular indefinite even lattice of signature $(n_1,n_2)$ exists if and only if $n_1\equiv n_2$ ($\mo$ 8), and when this holds, the lattice is isomorphic to $\II_{n_1,n_2}=E_8^{\frac{n_1-n_2}{8}}\oplus U_2^{n_2}$ or $E_8(-1)^\frac{n_2-n_1}{8}\oplus U_2^{n_1}$.
\end{thm}

\begin{rmk}
The structure theory in definite case is much more complicated. For example, we have 24 Niemeier lattices (see Thm \ref{thm_niemeier}), all of which are positive definite, unimodular, even and of rank 24.
\end{rmk}

\subsection{Classification of p-adic lattices, and Conway-Sloane's notation}
For any prime $p$, we use $\ZZ_p$ for ring of $p$-adic integers, and $\QQ_p$ for field of $p$-adic rational numbers. We next discuss about the classification of $\ZZ_p$-lattices, and the standard notation of Conway and Sloane \cite{conway1999spherepackings}. We also call a lattice over $\ZZ_p$ a {\it $p$-adic lattice}. Let $Qu(\ZZ_p)$ be the semigroup of $p$-adic lattices (with respect to $\oplus$).

For $\theta\in \ZZ_p^*/(\ZZ_p^*)^2$, denote by $K_{\theta}(p^k)$ the $p$-adic lattice determined by the matrix $\langle \theta p^k\rangle$. For $p$ an odd prime, $\ZZ_p^*/(\ZZ_p^*)^2$ contains two elements. For $p=2$, $\ZZ_2^*/(\ZZ_2^*)^2$ has four elements represented by $1,3,5,7\in \ZZ_2^*$. For the case $p=2$, we need to also consider lattices
$U(2^k)=\left(
\begin{array}{cc}
0 & 2^k \\
2^k & 0
\end{array}
\right)$
and
$V(2^k)=\left(
\begin{array}{cc}
2^{k+1} & 2^k \\
2^k & 2^{k+1}
\end{array}
\right)$
for any $k\ge 0$. 
\begin{prop}[{\cite[Proposition 1.8.1]{nikulin1980integral}}]
For $p$ an odd prime, the semigroup $Qu(\ZZ_p)$ is generated by $K_{\theta}(p^k)$. For $p=2$, the semigroup $Qu(\ZZ_2)$ is generated by $K_{\theta}(2^k)$, $U(2^k)$ and $V(2^k)$.
\end{prop}

For $p$ odd, any $p$-adic lattice $K$ can be written as a direct sum of rank one $p$-adic lattices. Explicitly, the quadratic form $q$ can be decomposed as
\begin{equation*}
K=K_1\oplus p K_p \oplus p^2 K_{p^2}\oplus \cdots \oplus l K_l \oplus \cdots,
\end{equation*}
where $l$ are powers of $p$, the determinant of each $K_l$ is coprime to $p$. Here the p-adic lattice $l K_l$ can be write as a direct sum of several p-adic lattices of the form $K_{\theta}(l)$. Following \cite[Chapter 15, \S 7]{conway1999spherepackings}, {\it the $p$-adic quadratic form $l K_l$ is denoted by $l^{\epsilon_l n_l}$}. Here $n_l$ is the rank of $K_l$, and $\epsilon_l$ is $+$ if $\det(K_l)$ is a square in $\ZZ_p^*$; is $-$ otherwise. Then the $p$-adic form $K$ is written as $1^{\epsilon_1 n_1} p^{\epsilon_p n_p}\cdots l^{\epsilon_l n_l}\cdots$. We call this the {\it Conway--Sloane expression} for $K$.

\begin{prop}
\label{proposition: p odd unique expression}
For $p$ odd, a $p$-adic lattice $K$ has a unique Conway--Sloane expression $1^{\epsilon_1 n_1} p^{\epsilon_p n_p}\cdots l^{\epsilon_l n_l}\cdots$.
\end{prop}

For $p=2$ the notation is more complicated. In this case, any $2$-adic lattice $(M, q)$ can be written as a direct sum of rank one $2$-adic lattices or rank two $2$-adic lattices of the forms
$\left(
\begin{array}{cc}
2^k a & 2^k b \\
2^k b & 2^k c
\end{array}
\right)$,
where $a ,c$ are even and $b$ is odd. Explicitly, the $2$-adic quadratic form $K$ can be decomposed as 
\begin{equation*}
K=K_1\oplus 2 K_2 \oplus 2^2 K_{2^2}\oplus \cdots \oplus l K_l \oplus \cdots,
\end{equation*}
where $l$ are powers of $2$, and the determinant of each $K_l$ is odd. By \cite[Chapter 15, \S7]{conway1999spherepackings}, the $2$-adic quadratic form $l K_l$ is written as $l_{S_l}^{\epsilon_l n_l}$. Here $n_l$ is the rank of $K_l$, and $\epsilon$ is $+$ if $\det(K_l)$ is congruent to $1$ or $7$ modulo $8$; is $-$ otherwise. A $2$-adic lattice is called even, if the norm of each vector is even; odd otherwise. If the $2$-adic lattice $K_l$ is even, then $S_l=\II$, and we say $l K_l$ are of even type. When $lK_l$ is of even type, it can be decomposed as a direct sum of $2$-adic lattices of the form $U(l)$ or $V(l)$. If $K_l$ is odd, then $S_l=Tr(K_l)\in \ZZ/8\ZZ$, and we say $l K_l$ are of odd type. When $lK_l$ is of odd type, it can be decomposed as a direct sum of $2$-adic lattices of the form $K_{\theta}(l)$, for $\theta\in \ZZ_2^*/(\ZZ_2^*)^2$. The $2$-adic form $K$ can be written as $1_{S_1}^{\epsilon_1 n_1} 2_{S_2}^{\epsilon_2 n_2}\cdots l_{S_l}^{\epsilon_l n_l}\cdots$. The ways to express a $2$-adic form as above are not unique, but there is a canonical way to do this (see \cite[Chapter 15, \S 7.6]{conway1999spherepackings}).

\begin{rmk}
\label{remark: existence condition}
The following conditions must holds for any $2$-adic constituent $l_{S_l}^{\epsilon_l n_l}$ of rank $n$:

\begin{enumerate}[(1)]
\item If $n=0$, then $S_l=\II$ and $\epsilon_l=+$.
\item If $n=1$, then the form is of odd type. In this case, if $\epsilon_l=+$, then $S_l$ is congruent to $1$ or $7$ modulo $8$; if $\epsilon=-$, then $S_l$ is congruent to $3$ or $5$ modulo $8$.
\item if $n=2$ and the form is of odd type, then $\epsilon=+$ implies that $S_l$ is congruent to $0$, $2$ or $6$ modulo $8$, $\epsilon=-$ implies that $S_l$ is congruent to $2$, $4$ or $6$ modulo $8$.
\end{enumerate}
For further discussion, we refer to \cite[Chapter 15, \S7.8]{conway1999spherepackings} (esp. \cite[Table 15.5]{conway1999spherepackings}).
 \end{rmk}

Two integral lattices are said to have the same {\it genus} if they are equivalent over the $p$-adic integers for all $p$. Under mild conditions, for indefinite lattices, there exists a single isometry class in a given genus. For definite lattices, typically there are multiple isometry classes in a genus (e.g. compare Theorems \ref{theorem: uniqueness of indefinite unimodular lattices} and \ref{thm_niemeier} in the unimodular case). 

\subsection{Conway-Sloane's expression for finite quadratic forms} One of the main tools in the Nikulin's theory \cite{nikulin1980integral} is the systematic use of finite discriminant forms. Here we review the basics, and we connect it with the Conway--Sloane notation. 

Given an integral lattice $M$, we denote $M^*=Hom_{\ZZ}(M, \ZZ)$, {\it the dual lattice}. We have naturally $M\hookrightarrow M^*\hookrightarrow \Hom_{\QQ}(M_{\QQ},\QQ)$, where the first map sends $x$ to $(x, \cdot)$. Define {\it the discriminant group} of $M$ to be $A_M=M^*/M$, which is a finite group of order the discriminant of the lattice.  For a finite abelian group $A$, we denote by $l(A)$ the minimal number of generators in $A$. The bilinear form on $M$ induces a bilinear form a bilinear form on $A_M$ valued in $\QQ/\ZZ$, by sending $[v], [w]\in A_M$ (with $v, w\in M^*$) to $[(v,w)]\in \QQ/\ZZ$. If $M$ is even, we can define a quadratic form 
$$q_M: A_M\to \QQ/2\ZZ,$$ by sending $[v]\in A_M$ to $[(v,v)]\in \QQ/2\ZZ$. The quadratic form $q_M$ recovers $b_M$ via the relation:
\begin{equation*}
b_M([v], [w])=\frac{1}{2}(q_M([v+w])-q_M([v])-q_M([w])).
\end{equation*}
 The quadratic form $q_M$ is called {\it the discriminant form} of $M$. We sometimes write $q_M$ instead of $(A_M,q_M)$.For the remaining of the appendix, we  restrict ourselves to the case of even lattices.
 
Any finite quadratic form $(A, q)$ has a unique decomposition $(A, q)=\oplus_p (A_p, q_p)$. Here $A_p$ is the group of $p$-power order elements in $A$, and $q_p$ is the restriction of $q$ to $A_p$. The finite quadratic form $q_p$ takes value in $\QQ_p/\ZZ_p\cong \QQ^{(p)}/\ZZ$ if $p$ is odd, and in $\QQ_2/2\ZZ_2\cong \QQ^{(2)}/2\ZZ$ if $p=2$. 

Let $qu(\ZZ_p)$ be the semigroup of finite quadratic forms on an abelian group with order a $p$-th power. Denote by $q_{\theta}(p^k)$ the discriminant form of the $p$-adic lattice $K_{\theta}(p^k)$. Denote by $u(2^k), v(2^k)$ the discriminant form of the $2$-adic lattices $U(2^k), V(2^k)$ respectively. 

\begin{prop}[{\cite[Proposition 1.8.1]{nikulin1980integral}}]
The semigroup $qu(\ZZ_p)$ is generated by $q_{\theta}(p^k)$ if $p$ is an odd prime; by $q_{\theta}(2^k)$, $u(2^k)$ and $v(2^k)$ if $p=2$.
\end{prop}

The following theorem (see \cite[Theorem 1.9.1]{nikulin1980integral}) tells that except very special cases (happen when $p$=2), a finite quadratic form over $\ZZ_p$ is induced uniquely by a $p$-adic form:

\begin{thm}[Nikulin]
Let $p$ be a prime and $(A, q)\in qu(\ZZ_p)$. There exists a unique $p$-adic lattice $K(q)\in Qu(\ZZ_p)$ of rank $l(A)$ and whose discriminant form is isomorphic to $q$, except in the case when $p=2$ and $q$ is $q_{\theta}(2)\oplus q_2^{\prime}$ for some $\theta\in \ZZ_2^*/(\ZZ_2^*)^2$.  

If $q=q_{\theta}(2)\oplus q_2^{\prime}$, there are precisely two $2$-adic lattices $K_{\alpha_1}(q)$ and $K_{\alpha_2}(q)$ of rank $l(A)$ and whose discriminant forms are isomorphic to $q$. Here $disc(K_{\alpha_i}(q))=\alpha_i |A| (\ZZ_2^*)^2$ for $i=1,2$, where $\alpha_1, \alpha_2\in \ZZ_2^*/(\ZZ_2^*)^2$ and $\alpha_1 \alpha_2=5(\ZZ_2^*)^2$.
\end{thm}

Given $q\in Qu(\ZZ_p)$, we have then a $p$-adic lattice $K(p)$ of rank $l(q)$ and whose discriminant form is $q$. The Conway-Sloane expression of the $p$-adic lattice $K(q)$ is used also to denote $q$. Notice that when $p=2$ and $q=q_{\theta}(2)\oplus q^{\prime}$, the expression of $q$ is not unique.  A finite quadratic form $q\in Qu(\ZZ)$ can be uniquely decomposed as a direct sum of finite quadratic forms over $\ZZ_p$. Connecting together the Conway-Sloane expressions for those sub forms of $q$, we get a Conway-Sloane expression for $q$. 

\subsection{Nikulin's criterions} We repeatedly use in our arguments two key results of Nikulin: the criteria for existence and uniqueness of embeddings of lattices into unimodular lattices (in our case, the relevant unimodular lattices are the Leech $\LL$ and Borchers $\BB=\LL\oplus U^2$ lattices). Specifically, the following is Nikulin's vast generalization of  Theorem  \ref{theorem: uniqueness of indefinite unimodular lattices} for  even lattices.

\begin{thm}[{Nikulin \cite[Thm. 1.10.1]{nikulin1980integral}}]
\label{theorem: existence of lattice via discriminant form}
An even lattice of invariant $(n_1,n_2,A,q)$ exists if and only if the following conditions are fulfilled:
\begin{enumerate}[(1)]
\item $n_2-n_1\equiv \sig(q)$ ($\mo$ 8),
\item $n_1\ge 0, n_2\ge 0, n_1+n_2\ge l(A)$,
\item $(-1)^{n_2}|A|\equiv disc(K_{q_p})$ (mod $(\ZZ_p^*)^2$) for all odd prime $p$ with $n_1+n_2=l(A_p)$,
\item $|A|=\pm disc(K_{q_2})$ (mod $(\ZZ_2^*)^2$) if $n_1+n_2=l(A_2)$ and $q_2\ne q_{\theta}(2)\oplus q_2^{\prime}$.
\end{enumerate}
\end{thm}

An embedding of a lattice $M$ into a unimodular lattice exists iff a lattice with complementary invariants (most notably discriminant form $-q_M$) exists. The above theorem allows one to settle this question. If an embedding exists, the uniqueness of the embedding can be settled frequently by the following result. 
 
\begin{thm}[{Nikulin \cite[Thm. 1.14.4]{nikulin1980integral}}]
\label{theorem: uniqueness of embedding even case}
Let $S$ be an even lattice of signature $(n_1,n_2)$ and let $M$ be an even unimodular lattice of signature $(l_1,l_2)$. There exists a unique primitive embedding of $S$ into $M$ if the following conditions are fulfilled:
\begin{enumerate}[(i)]
\item $l_1>n_1, l_2>n_2$,
\item $l_1+l_2-n_1-n_2\ge l(A_S)+2$.
\end{enumerate}
\end{thm}

\section{Finite Groups}
\label{appendix: finite groups}
In this appendix, we make a quick review of finite groups as relevant in this paper. We follow the notations of Mukai \cite{mukaiaut} and H\"ohn--Mason \cite{HM1,HM2}. 

\subsection{Extensions of finite groups}
Let $N, Q$ be two finite groups. An extension of $Q$ by $N$ is a finite group $E$ with a short exact sequence:

\begin{equation}
\label{equation: exact seq}
1\longrightarrow N\longrightarrow E\stackrel{p}{\longrightarrow} Q\longrightarrow 1.
\end{equation}
 Suppose there is a group homomorphism $r\colon Q\longrightarrow E$ with $p\circ r=id$, then the sequence \eqref{equation: exact seq} is called split. In this case, $E$ is denoted by $N\rtimes Q$, which is called the a semidirect product of $N$ and $Q$. Semidirect products of $N$ and $Q$ are not unique, and are uniquely determined by group homomorphisms $Q\longrightarrow \Aut(N)$. Following \cite{HM1, HM2}, we use $N:Q$ to represent a semidirect product of $N$ and $Q$, and use $N.Q$ to represent an extension of $Q$ by $N$ which we are not sure whether split or not.

\subsection{Mathieu groups}
\label{smathieu}
The series of Mathieu groups consist of five sporadic groups denoted by $M_{11}$, $M_{12}$, $M_{22}$, $M_{23}$, $M_{24}$. These are the first series of sporadic groups, which were found by Mathieu (1861, 1873). No other sporadic groups were found until 1965 the first Janko group was found. There are also Mathieu groups $M_8$, $M_9$, $M_{10}$, $M_{20}$, $M_{21}$, with the first four not simple and the last isomorphic to $\PSL(3, \FF_4)$ being simple but not sporadic.

We give the definition of the largest Mathieu group $M_{24}$ (see \cite[Appendix B]{eguchi2011mathieu} for further details). Let $N$ be the Niemeier lattice with root lattice $L$ of type $24A_1$. Then $L\subset N\subset L^*$ and $L^*/L\cong (\mathbb{F}_2)^{24}$, where $\FF_2$ is the field with 2 elements. We have $\mathscr{G}=N/L\subset \FF_2^{24}$ a 12-dimensional $\FF_2$ subspace, such that each nonzero element in $\mathscr{G}$ is not vanishing at no less than $8$ coordinates. This vector subspace $\mathscr{G}$ is known as the extended binary Golay code. The permutation group $S_{24}$ acts on $\FF_2^{24}$ by permutating the 24 coordinates, and the Mathieu group $M_{24}$ is defined to be the subgroup of $S_{24}$ that leaves $\mathscr{G}$ stable. The smaller Mathieu groups $M_{24-i}$ can be defined as the stabilizer group of $i$ coordinates of $\FF_2^{24}$ under the action of $M_{24}$, where $i=1,2,3$ or $4$.

\begin{rmk}
For $K3$ surfaces, Kond\=o's approach \cite{kondo} to the Mukai's classification of symplectic automorphisms reduces to considering subgroups of $M_{24}$ [and in fact $M_{23}$] (and involved the associated Niemeier lattice of type $24A_1$). For hyper-K\"ahler manifolds, it is necessary to pass to the Conway group $\Co_0$ (N.B. $M_{24}$ can be embedded into $\Co_0$) and the associated Leech lattice $\LL$ (e.g. see \cite{huybrechtsaut}). Furthermore, the behavior in terms of Leech lattice is more uniform; this is the point of view taken in this paper. 
\end{rmk}
A dodecad of $\mathscr{G}$ is an element vanishing at exactly $12$ coordinates. The Mathieu group $M_{12}$ is by definition the stabilizer of a dodecad in $\FF_2^{24}$ under the action of $M_{24}$. Then $M_{12}$ is a subgroup of the permutation group $S_{12}$ acting on the $12$ coordinates where the dodecad is vanishing. The Mathieu group $M_{12-i}$ is isomorphic to the subgroup of $M_{12}$ stabilizing $i$ coordinates chosen among the $12$, for $i=1,2,3$ or $4$. The action of $M_{12-i}$ on the remaining $12-i$ coordinates is sharply $(5-i)$-transitive. 

For $k+l\le 12$ and $l\ge 8$, $M_{k,l}$ is the subgroup $M_{k+l}\cap S_k\times S_l$ of $M_{k+l}$. This is well-defined since $M_{k+l}$ is $k$-transitive. Moreover, we have an exact sequence:
\begin{equation*}
1\longrightarrow M_l\longrightarrow M_{k,l}\longrightarrow S_k\longrightarrow 1,
\end{equation*}
and thus $|M_{k,l}|=|M_l|\cdot k!$. Let us briefly discuss the groups $M_{3,8}$ and $M_{2,9}$, as they are relevant for our study. The group $M_8$ is isomorphic to the quaternion group $Q_8$, and we have a semdirect product $M_{3,8}\cong Q_8:S_3$. Mukai \cite{mukaiaut} denotes this group by $T_{48}$. The group $M_9$ is equal to $\PSU_3(\FF_2)$ (see \S\ref{subsection: linear group}), and it holds $M_9\cong 3^2:Q_8$. We have $M_{2,9}\cong 3^2:\QD_{16}$. It is natural to expect that $M_{2,9}$ is exactly the group in item 15 of Table \ref{table: MSS aut of HK}, but we have not checked all the details.

\subsection{Extraspecial group}
For $p$ prime, recall that a $p$-group is a finite group with order a power of $p$. 

\begin{defn}
An extraspecial group is a non-abelian $p$-group $G$ with center $Z(G)\cong p$ and the quotient $G/Z(G)$ elementary abelian. 
\end{defn}

Every extraspecial group has order $p^{1+2k}$ with $k$ a positive integer. Conversely, for any prime number $p$ and positive integer $k$, there exist two extraspecial groups of order $p^{1+2k}$. By convention, the symbol $p^{1+2k}$ represent for an extraspecial group of order $p^{1+2k}$. For $p=2$ and $k=1$, the two extraspecial groups $2^{1+2}$ are the dihedral group $D_8$ and quaternion group $Q_8$.

\subsection{Linear and projective groups over finite fields}
\label{subsection: linear group}
Linear and projective groups over a field $K$ refer to Zariski-closed subgroups of $\GL(n, K)$ or $\PGL(n, K)$. When $K$ is a finite field, these groups are finite and play an important role in the classification of finite simple groups. In the final section we collect such kinds of groups related to our classifications.

We introduce the unitary groups over finite fields. For a finite group $\FF_{q^2}$ where $q=p^r$ and $p$ is a prime number, there is an $\FF_q$-linear involution $\alpha\colon \FF_{q^2}\longrightarrow \FF_{q^2}$ sending $x$ to $x^q$ (this is the $r$-th power of the Frobenius automorphism of $\FF_q$). Let $V$ be an $n$ dimensional vector space over $\FF_{q^2}$, then there is a unique $\FF_q$-bilinear form (called Hermitian form over finite field) $H\colon V\times V\longrightarrow \FF_{q^2}$ satisfying $H(w, v)= \alpha(H(v,w))$ and $H(v, cw)=cH(v, w)$ for any $c\in \FF_{q^2}$. Explicitly,  
\begin{equation*}
H(v, w)= \sum_{i=1}^n v_i^q w_i
\end{equation*}
The unitary group\footnote{Some authors use $U(n, q^2)$ for the same group} $U(n, q)$  represents for the automorphism group of the Hermitian space $(V, H)$. We note that the projective special unitary group $\PSU(3, \FF_2)$ is isomorphic to the Mathieu group $M_9$, and appears as symplectic automorphism group of a degree $2$ $K3$ surface.

The group $\PSL(2, \FF_{11})$ is simple and appears as the automorphism group of the Klein cubic threefold $V(x_1^2 x_2+x_2^2 x_3+x_3^2 x_4+x_4^2 x_5+x_5^2 x_1)$ (see \cite{adler1978automorphism}). As shown in Theorem \ref{theorem: maximal uniqueness}, there is a unique cubic fourfold with an order $11$ automorphism which is a triple cover of $\PP^4$ branched along the Klein cubic threefold.

\bibliography{reference} 

\providecommand{\bysame}{\leavevmode\hbox to3em{\hrulefill}\thinspace}
\providecommand{\MR}{\relax\ifhmode\unskip\space\fi MR }
\providecommand{\MRhref}[2]{%
  \href{http://www.ams.org/mathscinet-getitem?mr=#1}{#2}
}
\providecommand{\href}[2]{#2}
\begin{thebibliography}{Mon13b}

\bibitem[ACT11]{allcock2011moduli}
D.~Allcock, J.~A. Carlson, and D.~Toledo, \emph{The moduli space of cubic
  threefolds as a ball quotient}, Mem. Amer. Math. Soc. \textbf{209} (2011),
  no.~985, xii+70.

\bibitem[Adl78]{adler1978automorphism}
A.~Adler, \emph{On the automorphism group of a certain cubic threefold}, Amer.
  J. Math. \textbf{100} (1978), no.~6, 1275--1280.

\bibitem[BCS16]{boissiere2016classification}
S.~Boissi\`ere, C.~Camere, and A.~Sarti, \emph{Classification of automorphisms
  on a deformation family of hyper-{K}\"ahler four-folds by {$p$}-elementary
  lattices}, Kyoto J. Math. \textbf{56} (2016), no.~3, 465--499.

\bibitem[BD85]{beauville1985variety}
A.~Beauville and R.~Donagi, \emph{La vari\'et\'e des droites d'une hypersurface
  cubique de dimension {$4$}}, C. R. Acad. Sci. Paris S\'er. I Math.
  \textbf{301} (1985), no.~14, 703--706.

\bibitem[Bea86]{beauville1986groupe}
A.~Beauville, \emph{Le groupe de monodromie des familles universelles
  d'hypersurfaces et d'intersections compl\`etes}, Complex analysis and
  algebraic geometry ({G}\"ottingen, 1985), Lecture Notes in Math., vol. 1194,
  Springer, Berlin, 1986, pp.~8--18.

\bibitem[BHT15]{bayer2015mori}
A.~Bayer, B.~Hassett, and Y.~Tschinkel, \emph{Mori cones of holomorphic
  symplectic varieties of {K}3 type}, Ann. Sci. \'{E}c. Norm. Sup\'{e}r. (4)
  \textbf{48} (2015), no.~4, 941--950.

\bibitem[BM14]{bayer2014mmp}
A.~Bayer and E.~Macr\`\i, \emph{M{MP} for moduli of sheaves on {K}3s via
  wall-crossing: nef and movable cones, {L}agrangian fibrations}, Invent. Math.
  \textbf{198} (2014), no.~3, 505--590.

\bibitem[Bor84]{borcherds1984leech}
R.~Borcherds, \emph{The {L}eech lattice and other lattices}, Ph.D. thesis,
  University of Cambridge, 1984.

\bibitem[Cha12]{charles2012remark}
F.~Charles, \emph{A remark on the {T}orelli theorem for cubic fourfolds},
  arXiv:1209.4509, 2012.

\bibitem[CS99]{conway1999spherepackings}
J.~H. Conway and N.~J.~A. Sloane, \emph{Sphere packings, lattices and groups},
  third ed., Grundlehren der Mathematischen Wissenschaften, vol. 290,
  Springer-Verlag, New York, 1999, With additional contributions by E. Bannai,
  R. E. Borcherds, J. Leech, S. P. Norton, A. M. Odlyzko, R. A. Parker, L.
  Queen and B. B. Venkov.

\bibitem[DIK00]{doi2000sextic}
H.~Doi, K.~Idei, and H.~Kaneta, \emph{Uniqueness of the most symmetric
  non-singular plane sextics}, Osaka J. Math. \textbf{37} (2000), no.~3,
  667--687.

\bibitem[DK07]{DK}
I.~V. Dolgachev and S.~Kond\=o, \emph{Moduli of {$K3$} surfaces and complex
  ball quotients}, Arithmetic and geometry around hypergeometric functions,
  Progr. Math., vol. 260, Birkh\"{a}user, Basel, 2007, pp.~43--100.

\bibitem[Dol96]{dolgachev_M}
I.~V. Dolgachev, \emph{Mirror symmetry for lattice polarized {$K3$} surfaces},
  J. Math. Sci. \textbf{81} (1996), no.~3, 2599--2630, Algebraic geometry, 4.

\bibitem[Eck76]{Eckardt1876}
F.~E. Eckardt, \emph{Ueber diegenigen flachen dritten grades, auf denen sich
  drei gerade linien in einen punkte schneiden}, Math. Ann \textbf{10} (1876),
  227--272.

\bibitem[EOT11]{eguchi2011mathieu}
T.~Eguchi, H.~Ooguri, and Y.~Tachikawa, \emph{Notes on the {$K3$} surface and
  the {M}athieu group {$M_{24}$}}, Exp. Math. \textbf{20} (2011), no.~1,
  91--96.

\bibitem[Fu16]{fu2016classification}
L.~Fu, \emph{Classification of polarized symplectic automorphisms of {F}ano
  varieties of cubic fourfolds}, Glasg. Math. J. \textbf{58} (2016), no.~1,
  17--37.

\bibitem[GAL11]{gonzalez}
V.~Gonz\'alez-Aguilera and A.~Liendo, \emph{Automorphisms of prime order of
  smooth cubic {$n$}-folds}, Arch. Math. (Basel) \textbf{97} (2011), no.~1,
  25--37.

\bibitem[GAL19]{gonzalez2}
\bysame, \emph{On automorphism and {E}ckardt points of cubic threefolds},
  preprint, 2019.

\bibitem[Gar13]{garbagnati2013elliptic}
A.~Garbagnati, \emph{Elliptic {K}3 surfaces with abelian and dihedral groups of
  symplectic automorphisms}, Comm. Algebra \textbf{41} (2013), no.~2, 583--616.

\bibitem[GHS07]{GHS}
V.~A. Gritsenko, K.~Hulek, and G.~K. Sankaran, \emph{The {K}odaira dimension of
  the moduli of {$K3$} surfaces}, Invent. Math. \textbf{169} (2007), no.~3,
  519--567.

\bibitem[GHV12]{gaberdiel2012symmetries}
M.~R. Gaberdiel, S.~Hohenegger, and R.~Volpato, \emph{Symmetries of {K}3 sigma
  models}, Commun. Number Theory Phys. \textbf{6} (2012), no.~1, 1--50.

\bibitem[Har14]{harui2014}
T.~Harui, \emph{Automorphism groups of smooth plane curves}, 1306.5842v1, 2014.

\bibitem[Has00]{hassett}
B.~Hassett, \emph{Special cubic fourfolds}, Compositio Math. \textbf{120}
  (2000), no.~1, 1--23.

\bibitem[Has12]{hashimoto2012K3}
K.~Hashimoto, \emph{Finite symplectic actions on the {$K3$} lattice}, Nagoya
  Math. J. \textbf{206} (2012), 99--153.

\bibitem[HL90]{harada1990leech}
K.~Harada and M.-L. Lang, \emph{On some sublattices of the {L}eech lattice},
  Hokkaido Math. J. \textbf{19} (1990), no.~3, 435--446.

\bibitem[HM14]{HM2}
G.~H\"ohn and G.~Mason, \emph{Finite groups of symplectic automorpshisms of
  hyperk\"ahler manifolds of type {$K3^{[2]}$}}, arXiv:1409.6055, 2014.

\bibitem[HM16]{HM1}
\bysame, \emph{The 290 fixed-point sublattices of the {L}eech lattice}, J.
  Algebra \textbf{448} (2016), 618--637.

\bibitem[Hos97]{hosoh}
T.~Hosoh, \emph{Automorphism groups of cubic surfaces}, J. Algebra \textbf{192}
  (1997), no.~2, 651--677.

\bibitem[HT16]{HT_Aut}
B.~Hassett and Y.~Tschinkel, \emph{Extremal rays and automorphisms of
  holomorphic symplectic varieties}, K3 surfaces and their moduli, Progr.
  Math., vol. 315, Birkh\"{a}user/Springer, [Cham], 2016, pp.~73--95.

\bibitem[Huy16]{huybrechtsaut}
D.~Huybrechts, \emph{On derived categories of {K}3 surfaces, symplectic
  automorphisms and the {C}onway group}, Development of moduli theory---{K}yoto
  2013, Adv. Stud. Pure Math., vol.~69, Math. Soc. Japan, [Tokyo], 2016,
  pp.~387--405.

\bibitem[JL17]{javanpeykar2017complete}
A.~Javanpeykar and D.~Loughran, \emph{Complete intersections: moduli,
  {T}orelli, and good reduction}, Math. Ann. \textbf{368} (2017), no.~3-4,
  1191--1225.

\bibitem[Kon98]{kondo}
S.~Kond\=o, \emph{Niemeier lattices, {M}athieu groups, and finite groups of
  symplectic automorphisms of {$K3$} surfaces}, Duke Math. J. \textbf{92}
  (1998), no.~3, 593--603, With an appendix by Shigeru Mukai.

\bibitem[Kon99]{kondo_ns}
\bysame, \emph{The maximum order of finite groups of automorphisms of {$K3$}
  surfaces}, Amer. J. Math. \textbf{121} (1999), no.~6, 1245--1252.

\bibitem[Laz09]{gitcubic}
R.~Laza, \emph{The moduli space of cubic fourfolds}, J. Algebraic Geom.
  \textbf{18} (2009), no.~3, 511--545.

\bibitem[Laz10]{laza2010moduli}
\bysame, \emph{The moduli space of cubic fourfolds via the period map}, Ann. of
  Math. (2) \textbf{172} (2010), no.~1, 673--711.

\bibitem[Laz18]{maxalg}
\bysame, \emph{Maximally algebraic potentially irrational cubic fourfolds},
  arXiv:1805.04063, 2018.

\bibitem[LO16]{LOG1}
R.~Laza and K.~G. O'Grady, \emph{{GIT} versus {B}aily-{B}orel compactification
  for quartic {$K3$} surfaces}, to appear in Compositio Math.
  (arXiv:1612.07432), 2016.

\bibitem[Loo09]{lcubic}
E.~Looijenga, \emph{The period map for cubic fourfolds}, Invent. Math.
  \textbf{177} (2009), no.~1, 213--233.

\bibitem[LPZ18]{laza2017moduli}
R.~Laza, G.~Pearlstein, and Z.~Zhang, \emph{On the moduli space of pairs
  consisting of a cubic threefold and a hyperplane}, Adv. Math. \textbf{340}
  (2018), 684--722.

\bibitem[LS07]{looijenga2007period}
E.~Looijenga and R.~Swierstra, \emph{The period map for cubic threefolds},
  Compos. Math. \textbf{143} (2007), no.~4, 1037--1049.

\bibitem[Mon13a]{mongardi2013thesis}
G.~Mongardi, \emph{Automorphisms of {H}yperk\"ahler manifolds}, Ph.D. thesis,
  Universit\`a degli Studi di Roma Tre, 2013, arXiv:1303.4670.

\bibitem[Mon13b]{mongardi2013symplectic}
\bysame, \emph{On symplectic automorphisms of hyper-{K}\"ahler fourfolds of
  {$\rm K3^{[2]}$} type}, Michigan Math. J. \textbf{62} (2013), no.~3,
  537--550.

\bibitem[Mon16]{mongardiaut}
\bysame, \emph{Towards a classification of symplectic automorphisms on
  manifolds of {$K3^{[n]}$} type}, Math. Z. \textbf{282} (2016), no.~3-4,
  651--662.

\bibitem[MPK16]{MPK}
S.~Marcugini, F.~Pambianco, and H.~Kaneta, \emph{Projective automorphism groups
  of nonsingular quartic surfaces}, arXiv:1611.10101, 2016.

\bibitem[Muk88]{mukaiaut}
S.~Mukai, \emph{Finite groups of automorphisms of {$K3$} surfaces and the
  {M}athieu group}, Invent. Math. \textbf{94} (1988), no.~1, 183--221.

\bibitem[Nik79a]{nikulin}
V.~V. Nikulin, \emph{Finite groups of automorphisms of {K}\"ahlerian {$K3$}\
  surfaces}, Trudy Moskov. Mat. Obshch. \textbf{38} (1979), 75--137.

\bibitem[Nik79b]{nikulin1980integral}
\bysame, \emph{Integer symmetric bilinear forms and some of their geometric
  applications}, Izv. Akad. Nauk SSSR Ser. Mat. \textbf{43} (1979), no.~1,
  111--177, 238.

\bibitem[Ouc19]{Ouchi}
G.~Ouchi, \emph{Automorphism groups of cubic fourfolds and {$K3$} categories},
  arXiv:1909.11033, 2019.

\bibitem[Sca87]{scattone}
F.~Scattone, \emph{On the compactification of moduli spaces for algebraic
  {$K3$} surfaces}, Mem. Amer. Math. Soc. \textbf{70} (1987), no.~374, x+86.

\bibitem[Seg42]{segre1942cubic}
B.~Segre, \emph{The {N}on-singular {C}ubic {S}urfaces}, Oxford University
  Press, Oxford, 1942.

\bibitem[Ser73]{serre1973course}
J.~P. Serre, \emph{A course in arithmetic}, Springer-Verlag, New
  York-Heidelberg, 1973, Translated from the French, Graduate Texts in
  Mathematics, No. 7.

\bibitem[vGS07]{vGS}
B.~van Geemen and A.~Sarti, \emph{Nikulin involutions on {$K3$} surfaces},
  Math. Z. \textbf{255} (2007), no.~4, 731--753.

\bibitem[Voi86]{voisin}
C.~Voisin, \emph{Th\'eor\`eme de {T}orelli pour les cubiques de {${\bf P}^5$}},
  Invent. Math. \textbf{86} (1986), no.~3, 577--601.

\bibitem[Wil83]{wilson1983suzuki}
R.~A. Wilson, \emph{The complex {L}eech lattice and maximal subgroups of the
  {S}uzuki group}, J. Algebra \textbf{84} (1983), no.~1, 151--188.

\bibitem[Wim96]{wiman1896a6}
A.~Wiman, \emph{Ueber eine einfache {G}ruppe von 360 ebenen {C}ollineationen},
  Math. Ann. \textbf{47} (1896), no.~4, 531--556.

\bibitem[Xia96]{xiao1996k3}
G.~Xiao, \emph{Galois covers between {$K3$} surfaces}, Ann. Inst. Fourier
  (Grenoble) \textbf{46} (1996), no.~1, 73--88.

\bibitem[YZ18]{yu2018moduli}
C.~Yu and Z.~Zheng, \emph{Moduli spaces of symmetric cubic fourfolds and
  locally symmetric varieties}, arXiv:1806.04873v2, 2018.

\bibitem[Zhe17]{zheng2017}
Z.~Zheng, \emph{{Orbifold Aspects of Certain Occult Period Maps}}, to appear in
  Nagoya Math. J. (arXiv:1711.02415), 2017.

\end{thebibliography}

\end{document}